\documentclass[10pt]{amsart}
\usepackage{geometry}                % See geometry.pdf to learn the layout options. There are lots.
\usepackage{graphicx}

\usepackage{amsfonts}
\usepackage{amssymb}
\usepackage{amsmath}
\usepackage{amsthm}
\usepackage{hyperref}
\usepackage{mathrsfs}
\usepackage{epstopdf}
\usepackage{url}
\usepackage[small,nohug, head=curlyvee]{diagrams}
\diagramstyle[labelstyle=\scriptstyle]
\newarrow{Mapsto} |--->
\usepackage{tikz}
\theoremstyle{plain}
\newtheorem{theorem}{Theorem}[section]
\newtheorem{lemma}[theorem]{Lemma}
\newtheorem{proposition}[theorem]{Proposition}
\newtheorem{corollary}[theorem]{Corollary}
\theoremstyle{remark}
\newtheorem{remark}[theorem]{Remark}
\theoremstyle{definition}
\newtheorem{definition}[theorem]{Definition}
\theoremstyle{definition}
\newtheorem{assumption}[theorem]{Assumption}
\theoremstyle{remark}
\newtheorem{example}[theorem]{Example}
\theoremstyle{plain}
\newtheorem{Question}[theorem]{The Basic Question}

\DeclareMathOperator{\spec}{Spec}
\DeclareMathOperator{\proj}{Proj}
\DeclareMathOperator{\spf}{Spf}
\DeclareMathOperator{\sheafproj}{\mathbf{Proj}}

\DeclareMathOperator{\pic}{NS}

\DeclareMathOperator{\NS}{NS}

\DeclareMathOperator{\Div}{Div}
\DeclareMathOperator{\effcurve}{NE}
\DeclareMathOperator{\nef}{\overline{\mathcal{K}}}
\DeclareMathOperator{\effdivisor}{Eff}
\DeclareMathOperator{\Ncurve}{N_1}

\DeclareMathOperator{\K}{\mathcal{K}}
\DeclareMathOperator{\Hom}{Hom}
\DeclareMathOperator{\ann}{ann}
\newcommand{\sat}{\mathrm{sat}}

\DeclareMathOperator{\Tr}{Tr}
\DeclareMathOperator{\Sp}{Sp}
\DeclareMathOperator{\GL}{GL}

\DeclareMathOperator{\M}{M}
\DeclareMathOperator{\Sym}{Sym}

\DeclareMathOperator{\Aut}{Aut}

\DeclareMathOperator{\sheafhom}{\mathscr{H}om}

\newcommand{\ud}{\,\mathrm{d}}

\newcommand{\Id}{\mathrm{Id}}

\DeclareMathOperator{\conv}{Conv}
\DeclareMathOperator{\Four}{Four}
\newcommand{\ppav}{0}
\newcommand{\DEG}{\mathrm{DEG}}
\newcommand{\DD}{\mathrm{DD}}

\DeclareMathOperator{\Isom}{\underline{Isom}}
\newcommand{\ad}{\mathfrak{A}}
\newcommand{\gp}{\mathrm{gp}}

\renewcommand{\H}{\mathrm{H}}

\newcommand{\rc}{\mathrm{rc}}

\makeatletter

\newcommand{\Rmnum}[1]{\expandafter\@slowromancap\romannumeral #1@}
\makeatother

\title{Compactification of the Moduli of Polarized Abelian Varieties and Mirror Symmetry}

\begin{document}
\author{Yuecheng Zhu} \email{\url{yuechengzhu@math.utexas.edu}} \address{Department of Mathematics, the University of Texas at Austin, 2515 Speedway Stop C1200, Austin, TX, 78712, USA}
%\classification{14K10}
%\keywords{moduli space, polarized abelian varieties, compactification, mirror symmetry}

\begin{abstract}
We show that Martin Olsson's compactification of moduli space of polarized abelian varieties in \cite{ols08} can be interpreted in terms of KSBA stable pairs. We find that any degenerating family of polarized abelic sheme over a local normal base is equipped with a canonical set of divisors $S(K_2)$. Choosing any divisor $\Theta$ from the set $S(K_2)$, we get a KSBA stable pair. Then the limit in the moduli space of KSBA pairs $\overline{\mathscr{AP}}_{g,d}$ agrees with the canonical degeneration given by Martin Olsson's compactification. Moreover, we give an alternative construction of the compactification by using mirror symmetry. We construct a toroidal compactification $\overline{\mathscr{A}}_{g,\delta}^m$ that is isomorphic to Olsson's compactification over characteristic zero. The collection of fans needed for a toroidal compactification is obtained from the Mori fans of the minimal models of the mirror families. 
\end{abstract}

\maketitle

\tableofcontents

\section{Introduction}\label{1}
\subsection{Results and Motivations}
The basic question that concerns us is:
\begin{Question}\label{the basic question}
Given a family of polarized Calabi--Yau manifolds $(\mathcal{X}_\eta,\mathcal{L}_\eta)$ over a punctured disk $\Delta^*$, is there a \emph{canonical} way to fill in the central fiber? If the answer is ``yes", then how is the \emph{canonical} central fiber produced?
\end{Question}
Notice that Question~\ref{the basic question} is already interesting even for dimension $1$. There can be many polarized extensions of a polarized family of genus one curves with only nodal singularities on the central fibers. So there are too many ways of filling in the central fiber. See Example~\ref{Hesse family} for a simple example. 

However, things are different if we add marked points. An actual ample divisor picks out the \emph{canonical} central fiber. That is why we have the Deligne--Mumford's compactification $\overline{\mathscr{M}}_{1,d}$, the moduli of genus one stable curves with marked points. Our key observation is that any degenerating polarized family $(\mathcal{X}_\eta,\mathcal{L}_\eta)$, where $\mathcal{X}_\eta$ is isomorphic to a family of abelian varieties, comes with a \emph{canonical} set of divisors. So they are ``secretly'' pairs. The choice of a divisor from the set is not unique, but it doesn't matter. Using any of the pairs from the set to fill in the central fiber, we will get the same underlying polarized extension. We illustrate this through Example~\ref{Hesse family}.

In \cite{ols08}, Olsson constructs a modular compactification $\overline{\mathscr{A}}_{g,d}$ of $\mathscr{A}_{g,d}$ over $\mathbf{Z}$, and so in particular gives a positive answer to Question~\ref{the basic question} when $(\mathcal{X}_\eta,\mathcal{L}_\eta)$ is isomorphic to a family of polarized abelian varieties. The compactification uses the AN construction, which is a general construction of the degeneration of abelian varieties over a complete normal domain\footnote{It was first invented by Mumford \cite{Mum72}, and later improved in \cite{FC} and \cite{AN}. We use the version in \cite{AN}, and call it the AN construction.}. The AN construction is complicated, and not canonical when the degree of the polarization is $d>1$. From a degenerating family over the generic point, one gets two lattices $Y\subset X$ of index $d$, where $d$ is the degree of $\mathcal{L}_\eta$, and a quadratic function $A: Y\to \mathbf{Z}$. Then one needs to choose a function $\varphi$ over $X$ that extends $A$. Different choices usually give different central fibers.  For example, if one picks different extensions $\varphi$, one can get all central fibers in Figure~\ref{fig1}. Besides the log geometry, one of the observations in \cite{ols08} is that there is a canonical choice of the extension $\varphi$: one simply requires the extension to be also quadratic over $X$. This choice gives $I_3$ for Example~\ref{Hesse family}. However, from our point of view, this answer is unsatisfying for two things. First, the solution is implicit: if you want to know what the limit is, you need to run the machinery of the AN construction. There is no geometric explanation why the limit is canonical. Secondly, it is not obvious how to generalize the constructions to more general Calabi--Yau manifolds because the AN construction is special for abelian varieties.

In this paper, we first use mirror symmetry to give a different construction of Olsson's compactification. More precisely, we construct a compactification $\overline{\mathscr{A}}^m_{g,\delta}$ of the moduli space of polarized abelian varieties $\mathscr{A}_{g,\delta}$, which is isomorphic to the connected component $\overline{\mathscr{A}}_{g,\delta}[1/d]$ of $\overline{\mathscr{A}}_{g,d}$ (Proposition~\ref{relation between two compactifications}) over a field of characteristic zero. We expect our approach to work for more general Calabi--Yau manifolds. Moreover, the new thing is we give an alternative way of constructing the central fiber. We show that the canonical central fiber can be obtained by using KSBA stable pairs. This is a little surprising because a stable pair is a scheme plus a divisor $\Theta_\eta$ from the linear system $|\mathcal{L}_\eta|$. People usually do not expect that there is a natural choice of divisors $\Theta_\eta$ when $\mathcal{L}_\eta$ has higher degree.

\begin{example}\label{Hesse family}
Consider the Hesse family near $0$, $\mathcal{X}_\eta=\{(x,y,z)\in \mathbf{CP}^2; t(x^3+y^3+z^3)+xyz=0\}$ over $0<|t|<1$, with the relatively ample invertible sheaf $\mathcal{L}_\eta=\mathcal{O}_{\mathcal{X}_\eta}(1)$. We can add any polarized variety in Figure~\ref{fig1} as the central fiber over $t=0$. Suspicious readers may check the following. Take $t=0$ in the above embedding in $\mathbf{CP}^2$, we get the central fiber $I_3=\{xyz=0\}$, with the line bundle $\mathcal{L}=\mathcal{O}_{I_3}(1)$. Notice that the three irreducible components $L_1$, $L_2$, $L_3$ of $I_3$ are all $-2$-curves. By twisting $\mathcal{L}$ by $\mathcal{L}\otimes\mathcal{O}(L_1+L_2)$, and contracting $L_1+L_2$, we get $I_1$.  Moreover, blow up the $A_2$ singularity $p$ of $I_1$ to get a nonreduced exceptional divisor $E$\footnote{This is not the ordinary blow up of the surface along $p$. Look at the cone $\sigma$ for the toric picture of an $A_2$ singularity. We add the ray $v+v'$, where $v$ and $v'$ are rays of $\sigma$.}. Then make a base change of degree $2$. We can extend $\mathcal{L}$ such that $\deg\mathcal{L}|_{E}=2$ and $\deg\mathcal{L}|_{\widetilde{L}_3}=1$. Therefore, the three central fibers in Figure~\ref{fig1} are equally good. There is no obvious canonical limit. 

\begin{figure}
\begin{tikzpicture}[yscale=0.866]
\draw[very thick] (0,1) --(4,1);
\node[below] at (2,1) {$L_3$};
\draw[very thick] (0.5,0) --(2.5, 4);
\node[above left] at (1.5, 2) {$L_1$};
\draw[very thick] (1.5, 4) --(3.5,0);
\node[above right] at (2.5, 2) {$L_2$};
\node[below] at (2,0) {$I_3$};
\end{tikzpicture}
\quad
\begin{tikzpicture}[yscale=0.866]
\draw [very thick] (-2,4) -- (-1.9, 3.61) -- (-1.8, 3.24) -- (-1.75, 3.0625) -- (-1.7, 2.89) -- (-1.6, 2.56) -- (-1.5, 2.25) -- (-1.4, 1.96) -- (-1.3, 1.69) -- (-1.25, 1.5625) -- (-1.2, 1.44) --(-1.1, 1.21) -- (-1,1)-- (-0.9, 0.81) -- (-0.8, 0.64) -- (-0.7, 0.49) -- (-0.6, 0.36) -- (-0.5, 0.25) -- (-0.4, 0.16) -- (-0.333, 0.111) -- (-0.3, 0.09) -- (-0.2, 0.04) -- (-0.1667, 0.0278) -- (-0.15, 0.0225) -- (-0.1, 0.01)--(0,0)--(0.1, 0.01)-- (0.15, 0.0225) -- (0.1667, 0.0278) -- (0.2, 0.04) -- (0.3, 0.09) --(0.333, 0.111) --(0.4, 0.16)-- (0.5, 0.25) -- (0.6, 0.36) -- (0.7, 0.49) -- (0.8, 0.64) -- (0.9, 0.81) -- (1,1) -- (1.1, 1.21) -- (1.2, 1.44) -- (1.25, 1.5625) -- (1.3, 1.69)-- (1.4, 1.96)--(1.5, 2.25) --(1.6, 2.56) -- (1.7, 2.89) -- (1.75, 3.0625) -- (1.8, 3.24) -- (1.9, 3.61) -- (2,4);
\node[above] at (0,0) {$E$};
\draw[very thick] (-2,2) -- (2,2);
\node[above] at (0,2) {$\widetilde{L}_3$};
\node[below] at (0,0) {$I_2$};
\end{tikzpicture}
\quad
\begin{tikzpicture}[yscale=0.866]
\draw [very thick] (0,2)-- (0.0417, 2.2827) -- (0.0556, 2.3241) -- (0.083,2.3912) -- (0.167, 2.5292) -- (0.25, 2.6187) -- (0.333,2.68)--(0.5, 2.75)--  (0.667, 2.7698) -- (1,2.7071) -- (1.333, 2.5443) -- (1.667, 2.3042) -- (2,2) --(2.333, 1.64) -- (2.667, 1.2302) -- (3,0.775) --(3.333, 0.2787) -- (3.667, -0.25668);
\draw [very thick] (0,2) -- (0.0417, 1.7173) -- (0.0556, 1.6759) -- (0.083, 1.6088) -- (0.167, 1.4708) -- (0.25, 1.3813) -- (0.333,1.3196) -- (0.5, 1.25) -- (0.667, 1.2302) -- (1, 1.2929) -- (1.333, 1.4557) -- (1.667, 1.6957) -- (2,2) -- (2.333, 2.36) -- (2.667, 2.7698) -- (3, 3.2247) -- (3.333, 3.7213) -- (3.667, 4.2567);
\node[above] at (2,2) {$p$};
\node[below] at (2,0) {$I_1$};
\end{tikzpicture}
\caption{The Possible Central Fibers For the Hasse Family.}
\label{fig1}
\end{figure}
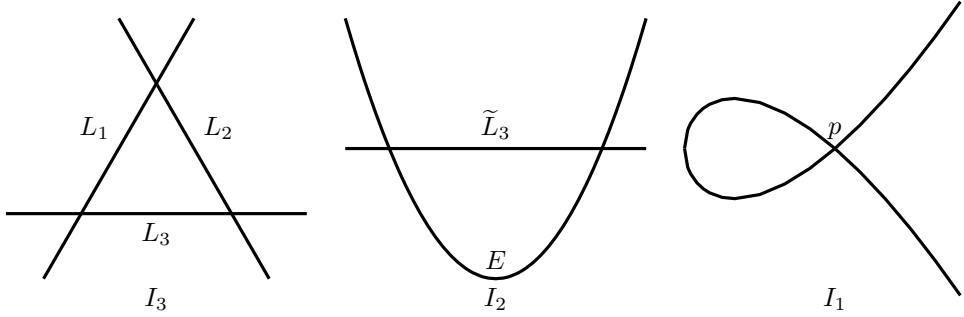
\end{example}

Here is our alternative interpretation of the extension for Example~\ref{Hesse family}. If one considers the monodromy, the family is actually equipped with a canonical set of marked points. Let the set be $S$. We obtain an object $(\mathcal{X}_\eta,\Theta_\eta)$ in $\mathscr{M}_{1,3}$ for any $\Theta_\eta\in S$. Using the stable reduction theorem for $\overline{\mathscr{M}}_{1,3}$, we get the pair $(\mathcal{X},\Theta)$ over $\Delta$. Now if we forget about $\Theta$, the underlying polarized family $(\mathcal{X},\mathcal{O}(\Theta))$ is always $I_3$, independent of the initial choice of $\Theta_\eta$ from $S$. The set $S$ is also a classical thing. The local monodromy determines the space of vanishing cycles, an isotropic subspace $U$ of $\H_1(X_t,\mathbf{R})$, where $X_t$ is a general fiber. According to the classical theory of theta functions, $U$ determines a finite collection of bases of $\H^0(X_t,\mathcal{L}_t)$, and each basis consists of theta functions (\cite{BL} Theorem 3.2.7)\footnote{We will define this set in terms of the representation of the theta group. The set of the divisor will be denoted by $S(K_2)$.}. $S$ is the set of divisors obtained as the sums of the basis vectors. For example, the set $\{X,Y,Z\}$ is one such basis (\cite{BL} Exercise 7.7 (8)). The zero locus of the sum $X+Y+Z$ is a divisor $\Theta_\eta$ in $S$. Figure~\ref{fig2} is the picture of the central fiber $I_3$. Notice that the two dashed lines are different choices of divisors from $S$, but both pairs are stable: each irreducible component has a marked point in the smooth locus. 
\begin{figure}
\begin{tikzpicture}
\draw[very thick] (0,0.866) --(4,0.866);
\draw[very thick] (0.5,0) --(2.5, 3.464);
\draw[very thick] (1.5, 3.464) --(3.5,0);
\draw[dashed, very thick][blue] (0,2) -- (4,0);
\draw[dashed, very thick][orange] (0, 0.55) -- (4,3);
\end{tikzpicture}
\caption{The Canonical Central Fiber With Different Choices of Divisors.}
\label{fig2}
\end{figure}
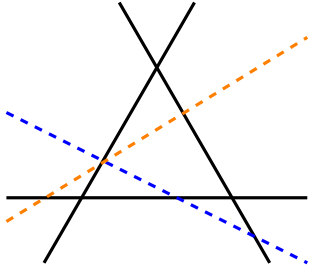

Mirror symmetry predicts that, near the maximal degenerations, the polarized Calabi--Yau varieties $(\mathcal{X}_\eta,\mathcal{L}_\eta)$ are ``secretly" stable pairs: there is a canonical set of divisors $S$ given by the mirror maps. For K3 surfaces, this is Tyurin's conjecture(\cite{Tyu} Page 36. Remark), and is proved in \cite{GHKtheta}. Naturally, one can pick a divisor $\Theta_\eta$ from this canonical set $S$, and take the limit of $(\mathcal{X}_\eta,\Theta_\eta)$ in the proper moduli space of KSBA stable pairs.

Now the generalization of $\overline{\mathscr{M}}_{1,d}$ to higher dimensions is $\overline{\mathscr{AP}}_{g,d}$, the moduli of KSBA stable pairs with actions of semiabelian varieties $(\mathcal{X}, G, \Theta, \varrho)$ (\cite{Alex96}, \cite{Alex02}). Take the extension $(\mathcal{X},\Theta)$ in $\overline{\mathscr{AP}}_{g,d}$. The underlying polarized schemes $(\mathcal{X},\mathcal{O}(\Theta))$ are all isomorphic, and the isomorphism class is the limit produced by Olsson's compactification $\overline{\mathscr{A}}_{g,d}$. Moreover, for abelian varieties, the set $S$ can be defined in terms of the local monodromy and the representation of Mumford's theta group $G(M)$. Thus we can generalize the set $S$ to an arbitrary degeneration, not just a maximal degeneration. We denote it by $S(K_2)$ (Definition~\ref{the balanced set}), and call it a balanced set. Our main result is Theorem~\ref{stable pairs moduli interpretation}:

\begin{theorem}\label{third result}
Suppose $S$ is a normal local integral scheme over $k=\mathbf{Z}[1/d,\zeta_M]$, (more precisely, see Assumption~\ref{assumption for the only if part}). Let $(\mathcal{X},G,\mathcal{L},\varrho)/S$ be a polarized stable semiabelic scheme over $S$, with the generic fiber abelian. Over the generic point $\eta$, the balanced set of divisors $S(K_2)$ is defined by the local monodromy. Then $(\mathcal{X},G,\mathcal{L},\varrho)$ is the pull-back of the AN family along a unique morphism $g:S\to \overline{\mathscr{A}}_{g,\delta}^m$ if and only if the group scheme $G(M)$ can be extended over $S$ (thus $S(K_2)$ is also extended over $S$), and for one (equivalently any) divisor $\Theta$ from $S(K_2)$, $(\mathcal{X},G,\Theta,\varrho)$ is an object in $\overline{\mathscr{AP}}_{g,d}$. 
\end{theorem}

Therefore, we are replacing the AN construction by the stable reduction theorem for KSBA stable pairs. We sketch the definition (Definition~\ref{the balanced set}) of $S(K_2)$. For any polarized abelian variety $(X,\mathcal{L})$, the space of global sections $\H^0(X,\mathcal{L})$ is an irreducible representation $S^*$ of a theta group $G(M)$. The local monodromy picks out a subgroup $K_w<G(M)$: the subgroup that preserves the vanishing cycles. Then we can decompose $\H^0(X,\mathcal{L})$ into $K_w$-irreducible representations $\H^0(X,\mathcal{L})=\oplus_{\alpha\in I} V_\alpha$, where the index set $I$ is the quotient of $G(M)$ by $K_w$. Take an arbitrary nonzero divisor $\vartheta_0$ from one of the irreducible representation $V_\alpha$, choose an arbitrary lift $\sigma: I\to G(M)$, and define a divisor $\Theta$ to be the zero locus of the section
\[
\vartheta=\sum_{\alpha\in I}S^*_{\sigma(\alpha)}\vartheta_0.
\]
The set $S(K_2)$ is defined to be the set of all divisors obtained this way. It only depends on the local monodromy. This is why we need the base to be local and normal. The key point is that the tropicalizations of the divisors $\Theta_\eta$ are \emph{canonical}: the piecewise affine function given by the quadratic function $\varphi$ in the AN construction. And the limit polarized variety only depends on this tropicalization. Therefore, Olsson's choice of the quadratic extension is explained by the action of the theta group. 

Notice when the polarization is principal, the linear system $|\mathcal{L}|$ contains only one element $\Theta$, and $S(K_2)$ is simply $\{\Theta\}$. This recovers Alexeev's compactification of $\mathscr{A}_g$ inside $\overline{\mathscr{AP}}_g$, up to a normalization\footnote{Although, in \cite{Alex02}, the main component of the coarse moduli space $\overline{\mathcal{AP}}_{g}$ is claimed to be isomorphic to the second Voronoi compactification, the proof is incomplete. The author thanks Iku Nakamura for pointing this out. See (\cite{Nak14} Subsection 14.3).}.

Our compactification $\overline{\mathscr{A}}^m_{g,\delta}$ is also constructed in a different way. Since the coarse moduli space $\mathcal{A}_{g,\delta}$ over $\mathbf{C}$ is an arithmetic quotient of a Hermitian symmetric domain, we use the toroidal compactification. The problem with the toroidal compactification is that it depends on extra data: an admissible collection of fans $\{\Sigma(F)\}$, and a priori, there are infinitely many choices for the fan $\Sigma(F)$.  However, mirror symmetry provides a candidate. According to the predictions of mirror symmetry in \cite{mor93}, the cone $\mathcal{C}(F)$ associated with each $0$-cusp $F$, on which the fan $\Sigma(F)$ is supported, can be identified with the K\"{a}hler cone of the mirror family associated with $F$\footnote{Here we follow the basic philosophy of mirror symmetry that the modular parameters of a Calabi--Yau space correspond to the K\"{a}hler parameters of the mirror, near each maximal degeneration.}. Our fans are induced from the Mori fans for the minimal models of the mirror families. The Mori fan is canonical for each mirror family (independent of the choice of a minimal model). Thus mirror symmetry plus Mori theory provide a canonical collection of fans. 

\begin{theorem}\label{first result}[Propostion~\ref{admissible collection of fans} and Theorem~\ref{the toroidal compactification over C}]
The collection of fans $\widetilde{\Sigma}=\{{\Sigma}(F)\}$ induced from the Mori fans of the mirror families is an admissible collection. Thus it gives a toroidal compactification $\overline{\mathcal{A}}_{\Sigma}$ of the coarse moduli space $\mathcal{A}_{g,\delta}$ over $\mathbf{C}$. Moreover, this compactification $\overline{\mathcal{A}}_{\Sigma}$ is projective. 
\end{theorem}

In the case of principal polarization, we get the interesting fact that the second Voronoi fan for the toroidal compactification of $\mathcal{A}_g$ is equal to the Mori fan of the mirror (Theorem~\ref{Two fans are the same}). By using the collection of fans $\{\Sigma(F)\}$ and the AN construction, we construct the compactification $\overline{\mathscr{A}}^m_{g,\delta}$ as an algebraic stack. 
\begin{theorem}\label{second result}[Theorem~\ref{compactification}]
Over $k=\mathbf{Z}[1/d,\zeta_M]$\footnote{For the definiton of the roots of unity $\zeta_M$, see the beginning of Section~\ref{roots of unity}}, we have a proper Deligne-Mumford $\overline{\mathscr{A}}_{g,\delta}^m$ that contains $\mathscr{A}_{g,\delta}$ as a dense open substack. The coarse moduli space over $\mathbf{C}$ is the toroidal compactificaiton $\overline{\mathcal{A}}_{\Sigma}$. There is an AN family $(\mathcal{X},G,\mathcal{L},\varrho)$ extending the universal family over $\mathscr{A}_{g,\delta}$. 
\end{theorem}

We expect that this approach of using mirror symmetry can be applied to Question~\ref{the basic question} for more general Calabi--Yau manifolds. One of the projects of Gross--Hacking--Keel (\cite{GHKtheta}) is to compactify the moduli of polarized K3 surfaces by using similar ideas. Therefore, the compactification $\overline{\mathscr{A}}_{g,\delta}^m$ should be regarded as a successful example of this very general approach.

%%%%%%%%%%%%%%%%%%%%%%%%%%%%%%%%%%%%%%%%%%%%%%%%%%%%%%%%%%%%%%%%%%%%%%%%%%%%%%%%%%%%%%%%%%%%%
%%%%%%%%%%%%%%%%%%%%%%%%%%%%%%%%%%%%%%%%%%%%%%%%%%%%%%%%%%%%%%%%%%%%%%%%%%%%%%%%%%%%%%%%%%%%%

\subsection{Outline of the Paper} 
In Sect.~\ref{2} we construct the toroidal compactification for the coarse moduli space. Sect.~\ref{2.1} is for fixing the notations and introducing the cones $\mathcal{C}(F)$ needed in the toroidal compactification. In order to deal with the complicated group actions for higher degree polarizations, we interpret them as moduli spaces of polarized tropical abelian varieties in Sect.~\ref{2.2}. This is Proposition~\ref{interpretation of the cone}. In Sect.~\ref{2.3}, we study the mirror symmetry for abelian varieties. We construct the mirror family $\mathcal{Y}_\eta/\Delta^*$ and its minimal models $\mathcal{Y}_\mathscr{P}/\Delta$. The goal is to get the collection of fans $\widetilde{\Sigma}=\{\Sigma(F)\}$ (achieved in Definition~\ref{the canonical fan from mirror symmetry}) for the toroidal compactification from the Mori fans of $\mathcal{Y}_\mathscr{P}/\Delta$ through a linear section (Proposition~\ref{the section is linear}). We then prove that the collection of fans $\widetilde{\Sigma}$ gives a toroidal compactification (Proposition~\ref{admissible collection of fans}) in Sect.~\ref{3.1}. Sect.~\ref{3} is devoted to the construction of the compactification as an algebraic stack. Once we get the correct fans, it is actually not difficult to construct the families by using AN constructions. However this is very technical, and Sect.~\ref{3} takes up a big part of the paper. We generalize the AN construction to a complete local base (Sect.~\ref{3.2}), and obtain the standard degeneration data compatible with our toroidal compactification (Sect.~\ref{3.3}). Following the procedures in \cite{FC}, we construct the algebraic stack $\overline{\mathscr{A}}^m_{g,\delta}$ (Theorem~\ref{compactification}) in Sect.~\ref{3.4}. The point of the efforts is to relate this compactification to Olsson's compactification. The relation to Olsson's compactification is stated in Proposition~\ref{relation between two compactifications}. Finally, in Sect.~\ref{4} we explain the geometric characterization of the extended family over the boundary in terms of KSBA stable pairs. The main theorem is Theorem~\ref{stable pairs moduli interpretation}. 

\subsection{Acknowledgements}
This paper is based on the author's thesis. Special thanks go to the advisor, Sean Keel for suggesting this topic, and for all the help and support. The author is grateful for discussions with and help from a lot of people, including Valery Alexeev, Kenji Fukaya, Mark Gross, Travis Mandel, Tim Magee, Iku Nakamura, Andrew Neitzke, Martin Olsson, James Pascaleff, Timothy Perutz, Helge Ruddat, Bernd Siebert, and Yuan Yao. The author thanks the referee for carefully reading the manuscript and for pointing out a wrong lemma in the original Section~\ref{4}. The author also thanks the Fields institute and Institute for Mathematical Sciences for hospitality.

%%%%%%%%%%%%%%%%%%%%%%%%%%%%%%%%%%%%%%%%%%%%%%%%%%%%%%%%%%%%%%%%%%%%%%%%
%%%%%%%%%%%%%%%%%%%%%%%%%%%%%%% Conventions and Notations %%%%%%%%%%%%%%%%%%%%%%%%%%%
%%%%%%%%%%%%%%%%%%%%%%%%%%%%%%%%%%%%%%%%%%%%%%%%%%%%%%%%%%%%%%%%%%%%%%%%

\subsection{Conventions and Notations}\label{Conventions}
If $X$ is a finitely generated free $\mathbf{Z}$-module, $X_k:=X\otimes_\mathbf{Z}k$ for a field $k$. The natural inclusion $X\subset X_k$ is an integral structure on the affine space $X_k$. A polytope $Q\subset X_\mathbf{R}$ is the convex hull of finitely many points in $X_\mathbf{R}$ and is always bounded. If all the points can be chosen from the lattice $X$, then $Q$ is called a lattice polytope. A polyhedron is the intersection of (not necessarily finitely many) closed half spaces. For an arbitrary subset $S$ of an affine space, the cone generated by $S$ is denoted by $C(S)$, and the convex hull is denoted by $\conv(S)$. We denote $\mathbf{Z}\oplus X$ by $\mathbb{X}$, and regard $X\subset \mathbb{X}$ as the hyperplane of height $1$. If $\sigma$ is a polytope in $X_\mathbf{R}$, the cone $C(\sigma)$ is inside $\mathbb{X}_\mathbf{R}$. Any piecewise affine function $\varphi$ over $X_\mathbf{R}$ has a unique piecewise linear extension over $\mathbb{X}_\mathbf{R}$, denoted by $\tilde{\varphi}$.  

A paving $\mathscr{P}$ is a set of polyhedrons of $X_\mathbf{R}$, such that
\begin{enumerate}
\item For any two elements $\sigma, \tau\in \mathscr{P}$, the intersection $\sigma\cap\tau$ is a proper face of both $\sigma$ and $\tau$. 
\item Any face of a polyhedron $\sigma\in \mathscr{P}$ is again an element of $\mathscr{P}$.
\item The union $\cup_{\sigma\in\mathscr{P}}\sigma$ is a polyhedron $Q$ of $X_\mathbf{R}$.
\item For any bounded subset $W\subset X_\mathbf{R}$, there exists only finitely many $\sigma\in \mathscr{P}$ with $W\cap \sigma\neq \emptyset$. 
\end{enumerate}

A paving is called bounded if all cells are polytopes. A paving is called a triangulation if all cells are simplices. A paving or a triangulation is called integral if all cells are lattice polytopes. We also assume that all pavings and triangulations are regular, that is, they are obtained as the affine regions of piecewise affine functions. 

A function $f$ is convex if $f(tx+(1-t)y)\leqslant tf(x)+(1-t)f(y)$. Notice that it is different from the usual convention in the literature of toric geometry. For a Cartier divisor $\{m_\sigma\}$ on a toric variety, we define the associated piecewise linear function to be $\varphi=-m_\sigma$. For a Weil divisor $D=\sum a_\omega D_\omega$, the associated function is $\psi(\omega)=a_\omega$. $Aff$ stands for the vector space of affine functions, $PA$ the space of piecewise affine functions, and $PL$ the space of piecewise linear functions. Usually there are other decorations in the notations. 

If $M$ is a monoid, then $M^\gp$ denotes the group associated with the monoid, $M^\sat$ the saturation, and $M^*$ the set of invertible elements.

The real valued functions are generalized to functions with values in vector spaces in \cite{GHK11}. Let $P$ be a toric monoid, i.e. $P$ is fine and saturated and the associated group $P^{\gp}$ is torsion free. Let $\sigma_P:=\conv(P)$ in $P^{\gp}_\mathbf{R}$. $\sigma_P$ is a polyhedral cone. We introduce a partial order on $P^{\gp}_\mathbf{R}$.

\begin{definition}
 For $u,v \in P^{\gp}_{\mathbf{R}}$, we say $u$ is $P$-above $v$, and denote it by $u\stackrel{P}{\geqslant} v$, if $u-v\in \sigma_P$. We say $u$ is strictly $P$-above $v$ if in addition, $u-v\in \sigma_P\backslash P^*_\mathbf{R}$. 
 \end{definition}
 
 Let $\varphi$ be a piecewise affine function over $X_\mathbf{R}$ with values in $P^{\gp}_\mathbf{R}$. Assume the region where $\varphi$ is affine gives a paving $\mathscr{P}$, i.e. for each $\sigma\in \mathscr{P}_{\max}$, $\varphi\vert_\sigma$ is an element in $\mathbb{X}^*_\mathbf{R}\otimes P^{\gp}_\mathbf{R}$. $\varphi$ is called integral if all $\varphi\vert_\sigma\in \mathbb{X}^*\otimes P^{\gp}$. 

\begin{definition}[bending parameters]\label{definition of bending parameter}
For each codimension $1$ cell $\rho\in\mathscr{P}$ contained in maximal cells $\sigma_+,\sigma_-\in \mathscr{P}$, we can write
\[
\varphi\vert_{\sigma_+}-\varphi\vert_{\sigma_-}=n_\rho\otimes p_\rho,
\]

where $n_\rho$ is the unique primitive element that defines $\rho$ and is positive on $\sigma_+$, and $p_\rho\in P^{\gp}_\mathbf{R}$ is called the bending parameter.
\end{definition} 

\begin{definition}
A piecewise affine function $\varphi$ is $P$-convex if for every codimension one cell $\rho\in \mathscr{P}$, $p_\rho\in P$. It is strictly $P$-convex if all $p_\rho\in P\backslash P^*$. 
\end{definition}

For toric varieties, we use the definitions and notations in \cite{CLS} unless specified otherwise. However, for a fan used in the toroidal compactification, $\sigma$ is defined to be the interior of the cone following the convention of \cite{FC}. Our main example is the second Voronoi fan (Definition~\ref{the second Voronoi fan}), where we define $C(\mathscr{P})$ to be the closure of the cone $\sigma(\mathscr{P})$ (For definition, see Definition~\ref{the second Voronoi fan}) for a Delaunay decomposition $\mathscr{P}$. 

Assume $\mathscr{P}$ is a paving of a lattice polytope $Q\subset X_\mathbf{R}$, and $\varphi: Q\to P^{\gp}_{\mathbf{R}}$ is a $\mathscr{P}$-piecewise affine, $P$-convex, integral function. Define
\[
Q_\varphi:=\Big\{(\alpha,h)\in Q\times P^{\gp}_\mathbf{R}: h\stackrel{P}{\geqslant} \varphi(\alpha)\Big\}.
\]

Take the cone $C(Q_\varphi)$ over the polyhedron $Q_\varphi$. Define the toric monoid $S(Q_\varphi):=C(Q_\varphi)\cap \mathbb{X}\times P^{\gp}$, the graded commutative ring $R_\varphi:=k[S(Q_\varphi)]$ and the base $S:=\spec k[P]$. Notice that $R_{\varphi,0}=k[P]$. $R_\varphi$ is of finite type over $k[P]$. Define $\mathcal{X}_{\varphi}=\proj R_\varphi$ and we have a projective morphism $\pi:\mathcal{X}_{\varphi}\to S$. The line bundle $\mathcal{L}:=\mathcal{O}(1)$ is $\pi$-ample. The $\mathbb{X}$-grading on $k[S(Q_\varphi)]$ induces an action $\varrho$ of $\mathbb{T}$ on $R_\varphi$, where $\mathbb{T}$ is the algebraic torus $\spec k[\mathbb{X}]$. 

For a normal scheme $S$, $\Div S$ is the group of Cartier divisors of $S$. For other notations in birational geometry, we follow \cite{KM98}. 

\begin{definition}\label{the standard symplectic form}
Fix $\delta=(\delta_1,\ldots,\delta_g)$ a sequence of positive integers $\delta_i$ such that $\delta_i\mid \delta_{i+1}$. A polarization $\lambda:X\to X^t$ of an abelian variety is of type $\delta$ if the kernel of $\lambda$ is isomorphic to $K(\delta)=H(\delta)\times\widehat{H(\delta)}$ for $H(\delta)=\oplus_{i=1}^{g}\mathbf{Z}/\delta_i\mathbf{Z}$. $d:=\prod_{i=1}^g\delta_i$ is called the degree of the polarization. Let $M=2\delta_g$.  
\end{definition}

We use the font $\mathscr{A}$ for an algebraic stack, and the font $\mathcal{A}$ for its coarse moduli space, if it exists. 

%%%%%%%%%%%%%%%%%%%%%%%%%%%%%%%%%%%%%%%%%%%%%%%%%%%%%%%%%%%%%%%%%%%%
%%%%%%%%%%% moduli of AV and mirror symmetry %%%%%%%%%%%%%%%%%%%%%%%%%%%%%%%%%%%%%%%
%%%%%%%%%%%%%%%%%%%%%%%%%%%%%%%%%%%%%%%%%%%%%%%%%%%%%%%%%%%%%%%%%%%%

\section{The Toroidal Compactification and Mirror Symmetry}\label{2}
\subsection{The Toroidal Compactification: The Setup}\label{2.1}
Fix a lattice $\Lambda\cong \mathbf{Z}^{2g}$ in a vector space $V\cong \mathbf{R}^{2g}$. Let $E$ be a non-degenerate skew-symmetric bilinear form over $V$, that takes integral values on $\Lambda$. By the elementary divisor theorem, there exists a symplectic basis $\{\lambda_1,\lambda_2\ldots,\lambda_g,\mu_1,\mu_2,\ldots,\mu_g\}$ of $\Lambda$, such that with respect to this basis,
\begin{equation}\label{standard symplectic form}
E=\begin{pmatrix}
0&\delta\\
-\delta&0
\end{pmatrix},
\text{  where }
\delta=\begin{pmatrix}
\delta_1& & \\
 & \ddots& \\
 & & \delta_g
 \end{pmatrix},
 \end{equation}

for some $\delta_i\in \mathbb{N}$ such that $\delta_i|\delta_{i+1}$. 
Fix a symplectic basis, and identify $\Lambda$ (resp. $V$) with $\mathbf{Z}^{2g}$ (resp. $\mathbf{R}^{2g}$). Define
\[
\Sp(E, \mathbf{R}):=\{R\in \M(2g,\mathbf{R}); \: RER^T=E\}.
\]

The subgroup $\Sp(E,\mathbf{R})$ is conjugate to the symplectic group $\Sp(2g,\mathbf{R})$ in $\GL(2g, \mathbf{R})$ by the following transformation
\begin{equation}
R=\begin{pmatrix}a&b\\c&d\end{pmatrix}\mapsto M=\begin{pmatrix}\alpha&\beta\\ \gamma&\epsilon\end{pmatrix}= \begin{pmatrix}1&0\\0&\delta\end{pmatrix}^{-1}R\begin{pmatrix}1&0\\0&\delta\end{pmatrix}\in\Sp(2g,\mathbf{R}).   \label{map}
\end{equation}

Denote $\Sp(E,\mathbf{Z})$ by $\Gamma(\delta)$. $\Gamma(\delta)$ is an arithmetic subgroup of $\Sp(2g,\mathbf{R})$ through the map (\ref{map}).

The following lemma is (\cite{BL} Lemma 17.2.3)\footnote{Notice that our convention is different from theirs by a transposition.}. We choose the convention that a basis in $\H_1(A, \mathbf{Z})$ is written in a column, and a period matrix is thus a $2g\times g$-matrix. The matrix multiplication on $\mathbf{R}^{2g}$ is from the right. 

\begin{lemma}\label{data in Siegel upper half space}
The following data are equivalent,
\begin{itemize}
\item[a)] a complex structure $J: \mathbf{R}^{2g}\to \mathbf{R}^{2g}, J^2=-1$ such that $E= \Im(H)$, $H$ a positive definite Hermitian form for this complex structure. The existence of $H$ is equivalent to: 
\begin{align*}
E(uJ, vJ)&=E(u,v),  \quad\forall u,v \in \mathbf{R}^{2g},\\
E(vJ, v)&>0, \quad\forall v\in \mathbf{R}^{2g}-\{0\}.
\end{align*}
In this case, $H(u, v)=E(uJ, v)+iE(u,v)$. The set of such $J$ is denoted by $\mathcal{C}_0(\Sp(E,\mathbf{R}))$.
\item[b)] a $g\times g$ complex symmetric matrix $\tau$ with $\Im(\tau)$ positive definite. The set of such $\tau$ is denoted by $\mathfrak{S}_g$. 
\end{itemize}

The bijection is determined by the commutative diagram,
\begin{diagram}
\mathbf{R}^{2g}&\rTo^{\big(\begin{smallmatrix}\tau\\ \delta\end{smallmatrix}\big)}& \mathbf{C}^{g}\\
\dTo_{J}&&\dTo^{iI_g}\\
\mathbf{R}^{2g}&\rTo^{\big(\begin{smallmatrix}\tau\\ \delta\end{smallmatrix}\big)}& \mathbf{C}^{g}.
\end{diagram}

Moreover, the Lie group $\Sp(E,\mathbf{R})$ is acting on $\mathcal{C}_0(\Sp(E,\mathbf{R}))$ from the left by conjugation $J\mapsto RJR^{-1}$. The Lie group $\Sp(2g, \mathbf{R})$ is acting on $\mathfrak{S}_g$ by $\tau\mapsto (\alpha\tau+\beta)(\gamma\tau+\epsilon)^{-1}$. The bijection $\mathcal{C}_0(\Sp(E,\mathbf{R}))\to \mathfrak{S}_g$ is equivariant with respect to the action of Lie groups if we identify $\Sp(E,\mathbf{R})$ and $\Sp(2g,\mathbf{R})$ by the map \eqref{map}. 

\end{lemma}

From now on, we identify $\mathcal{C}_0(\Sp(E,\mathbf{R}))$ with the Siegel upper half space $\mathfrak{S}_g$. The point is that the tropicalization is defined in terms of $J$, while the literature on toroidal compactifications uses $\mathfrak{S}_g$. We need to build a connection between these two in Section~\ref{2.2}.

The group $\Gamma(\delta)$ is acting on the set of symplectic bases. Therefore, we have the action on $\mathfrak{S}_g$. For $R=\bigl(\begin{smallmatrix}a& b\\c& d\end{smallmatrix}\bigr)\in \Gamma(\delta)$, the action is
\begin{equation}
R(\tau)=(a\tau+b\delta)(c\tau+d\delta)^{-1}\delta. \label{action}
\end{equation}

For a polarized abelian variety $A$, we get $\Lambda=\H_1(A,\mathbf{Z})$, $V=(\H^{1,0}(A))^*$, and polarization $E$. Over $\mathbf{C}$, the coarse moduli space of abelian varieties with polarization type $\delta$ is the quotient $\Gamma(\delta)\backslash \mathfrak{S}_g$ (\cite{BL} Theorem 8.2.6, and Remark 8.10.4). As an arithmetic quotient of a Hermitian symmetric domain, $\Gamma(\delta)\backslash \mathfrak{S}_g$ admits a type of compactifications called the toroidal compactifications. We won't get into the details of the constructions. The important thing is, in order to construct a toroidal compactification, we need an admissible collection of fans $\{\Sigma(F)\}$ supported on cones $\{\mathcal{C}(F)\}$ for each cusp $F$. We explain the cones $\mathcal{C}(F)$ and various discrete groups acting on it. Our basic reference is \cite{HKW}\footnote{Although the book concerns the moduli of abelian surfaces, the statements of the toroidal compactification is true for all degrees. For general statements, see \cite{AMRT}}. We have a minimal compactification\footnote{Better known as the Baily--Borel--Satake compactification.} by adding rational boundary components to $\mathfrak{S}_g$. The best way to see the rational boundary components is to realize $\mathfrak{S}_g$ as a bounded domain $\mathfrak{D}_g$ by the Cayley transform.
\begin{definition}
\[
\mathfrak{D}_g:=\{Z\in \M(g,\mathbf{C}); Z=Z^T, I_g-Z\overline{Z}>0\}.
\]
\end{definition}

\begin{definition}
The Cayley transformation $\mathfrak{S}_g\to \mathfrak{D}_g$ is the map
\[
\tau\mapsto Z=(\tau-iI_g)(\tau+iI_g)^{-1}. 
\]
\end{definition}

The rational boundary components are all inside $\overline{\mathfrak{D}}_g\backslash \mathfrak{D}_g$. We also call the rational boundary components cusps. If the rational boundary component is of dimension $g'(g'+1)/2$, it is called a $g'$-cusp. Each cusp $F$ corresponds to a rational isotropic subspace $U(F)$ of $V$. The basic result is 

\begin{proposition}[\cite{HKW} Proposition. 3.16, Proposition. 3.19]
$U$ is a bijection between the following objects:
\begin{enumerate}
\item \begin{itemize}
\item[a)] rational boundary components $F$ of dimension $g'(g'+1)/2$,
\item[b)] rational $E$-isotropic subspaces $U(F)\subset\mathbf{R}^{2g}$ of dimension $g-g'$.
\end{itemize}
\item\begin{itemize}
\item[a)] pairs of adjacent boundary components $F'\succ F$ ($F\neq F', F\subset \overline{F}'$),
\item[b)] pairs of $E$-isotropic subspaces $U(F')\subsetneqq U(F)$.
\end{itemize}
\end{enumerate}
\end{proposition}

Moreover, the correspondence is $\Sp(2g, \mathbf{R})$-equivariant. We have $U(R(F))=U(F)\cdot R^{-1}$, for $R\in \Gamma(\delta)$. Each rational boundary component is equivalent under $\Sp(E,\mathbf{Q})$ (left action) to one of the following boundary components,
\[
F^{(g')}=\left\{\begin{pmatrix}Z'&0\\0&I_{g-g'}\end{pmatrix}; Z'\in \mathfrak{D}_{g'}\right\}, \  0\leqslant g'\leqslant g. 
\]
\[
U^{(g')}=U(F^{(g')})=\left\{(0\ x); 0\in \mathbf{R}^{g+g'}, x\in\mathbf{R}^{g-g'}\right\}.
\]

\begin{remark}
Notice that the map $U$ depends on the choice of $\delta$. If we denote the corresponding isotropic subspace in the principally polarized case by $U^\ppav$, then $U\bigl(\begin{smallmatrix} I & 0\\ 0 &\delta\end{smallmatrix}\bigr)=U^\ppav$. 
\end{remark}

For any cusp $F$, the stabilizer $\mathcal{P}(F)\subset \Sp(2g,\mathbf{R})$ is a parabolic subgroup. If $F'=M(F)$, for $M\in \Sp(2g,\mathbf{R})$, we have $U(F')=U(F)\cdot M^{-1}$ and $\mathcal{P}(F')=M\mathcal{P}(F)M^{-1}$. The center $\mathcal{P}'(F)$ of the unipotent radical is a real vector space. The normalizer of $\mathcal{P}'(F)$ in $\Sp(2g,\mathbf{R})$ is $\mathcal{P}(F)$, and $\mathcal{P}(F)$ is acting on $\mathcal{P}'(F)$ by conjugation. Let the centralizer of $\mathcal{P}'(F)$ in $\mathcal{P}(F)$ be $Z(\mathcal{P}')$ and $\mathcal{P}(F)=\mathcal{G}_l(F)\ltimes Z(\mathcal{P}')$, where $\mathcal{G}_l(F)$ is the quotient $\mathcal{P}(F)/Z(\mathcal{P}')$. Since the conjugation action of $Z(\mathcal{P}')$ is trivial, it factors through an action of $\mathcal{G}_l(F)$ on $\mathcal{P}'(F)$. Denote $\Gamma(\delta)\cap\mathcal{G}_l(F)$ by $G_l(F)$ and the image of $G_l(F)$ in $\GL(\mathcal{P}')$ by $\overline{P}(F)$.
 
The cone $\mathcal{C}(F)$ that supports the fan is a special orbit of the $\mathcal{P}$-action in the vector space $\mathcal{P}'$, and is given by the Harish--Chandra map. While it is hard to describe the Harish--Chandra map in general, it is easy to write down everything for the special case $F^{(g')}$. 

\begin{example}\label{groups in the principal cusp}
Over $F^{(g')}$, we have, 
\begin{enumerate}
\item $U^{(g')}=\left\{(0\ x); 0\in\mathbf{R}^{g+g'}, \ x\in \mathbf{R}^{g-g'}\right\}$,
\item $\mathcal{P}'^{(g')}=\left\{[Q]:=\begin{pmatrix} I_{g'} & 0 & 0 &0\\0 & I_{g-g'} &0 & Q\\ 0& 0 & I_{g'} &0\\ 0& 0& 0 & I_{g-g'}\end{pmatrix}; Q\in \Sym_{g-g'}(\mathbf{R})\right\}\cong \Sym_{g-g'}(\mathbf{R})$,
\item $\mathcal{C}^{(g')}=\left\{[Q] ; 0<Q\in\Sym_{g-g'}(\mathbf{R})\right\}$,
\item $\mathcal{G}_l^{(g')}=\left\{\begin{pmatrix} I_{g'}& 0& 0& 0\\ 0& (u^T)^{-1}& 0 &0\\ 0 & 0& I_{g'}& 0\\0& 0 & 0& u\end{pmatrix}; u\in \GL(g-g',\mathbf{R})\right\}\cong\GL(g-g',\mathbf{R})$,
%\item $\mathcal{G}_h^{(g')}=\left\{\begin{pmatrix} A' & 0& B'& 0\\ 0& I_{g-g'}& 0 &0\\ C' & 0& D' & 0\\0& 0 & 0& I_{g-g'}\end{pmatrix}; \begin{pmatrix} A'& B'\\C'&D'\end{pmatrix}\in\Sp(2g',\mathbf{R})\right\}$,
%\item $\mathcal{V}^{(g')}=\left\{\begin{pmatrix} I_{g'} &0 & 0 & n^T\\ m & I_{g-g'} & n & 0\\ 0 &0 & I_{g'} & -m^T\\ 0 & 0 & 0 & I_{g-g'}\end{pmatrix}; m,n \in \M((g-g')\times g',\mathbf{R}), mn^T=(mn^T)^T\right\}$,
\item $\delta_{g'}=\begin{pmatrix} \delta_{g'+1} & & \\ & \ddots &\\ & & \delta_g\end{pmatrix}$,
\item $\overline{P}^{(g')}\cong\GL(\delta_{g'}):=\left\{ u\in \GL(g-g',\mathbf{Z}); \delta_{g'} u\delta_{g'}^{-1}\in \GL(g-g',\mathbf{Z})\right\}$.
 
\end{enumerate}

The action of $\overline{P}^{(g')}$ on the vector space $\mathcal{P}'^{(g')}$ is 
\[
Q\mapsto (u^T)^{-1}Qu^{-1}.
\]
\end{example}

\begin{definition}
For an arbitrary boundary component $F$, let $M\in \Sp(2g,\mathbf{R})$ be such that $F=M(F^{(g')})$, and define
\[
\mathcal{C}(F):=M\mathcal{C}^{(g')}M^{-1}.
\]
\end{definition}

\begin{remark}
The definition of $\mathcal{C}(F)$ is independent of the choice of $M$.
\end{remark}

For each cone $\mathcal{C}(F)$, we need a fan $\Sigma(F)$ whose support is the rational closure of $\mathcal{C}(F)$. If the collection of fans is admissible (\cite{HKW} Definition 3.61, Definition 3.66), then we can construct a proper algebraic space, called the toroidal compactification of $\Gamma(\delta)\backslash\mathfrak{D}_g$. If $g=2$, this is (\cite{HKW} Theorem 3.82). The general case is the main theorem (\cite{AMRT} Theorem 5.2). 

Since it is difficult to deal with the cones and various discrete groups in a Lie group, we interpret the cones $\mathcal{C}(F)$ as the moduli of tropical abelian varieties.

%%%%%%%%%%%%%%%%%%%%%%%%%%%%%%%%%%%%%%%%%%%%%%%%%%%%%%
%%%%%%%%%%%%%%%%%%% tropical cone %%%%%%%%%%%%%%%%%%%%%%%%%%%%
%%%%%%%%%%%%%%%%%%%%%%%%%%%%%%%%%%%%%%%%%%%%%%%%%%%%%%

\subsection{A Tropical Interpretation of the Cones}\label{2.2}
In this section, we identify the cone $\mathcal{C}(F)$ with the moduli of polarized tropical abelian varieties. Fix a rational isotropic subspace $U$ of dimension $r=g-g'$ in $V$. We define the tropicalization in the direction of $U$ first. 
\begin{lemma}\label{splitting of V}
$V=UJ\oplus U^{\perp}$.  
\end{lemma}
\begin{proof}
If $u\in UJ\cap U^{\perp}$, then $uJ\in U$. Since $u\in U^{\perp}$, $E(uJ,u)=0$. $E(\cdot J,\cdot)$ is nondegenerate, so $u=0$, and $UJ\cap U^{\perp}=\{0\}$. Since $\dim UJ+\dim U^{\perp}=2g$, $V=UJ\oplus U^{\perp}$. 
 \end{proof}
 
\begin{definition}
For $J\in \mathcal{C}_0(\Sp(E,\mathbf{R}))$, define $g_J$ to be the isomorphism $UJ\to V/U^{\perp}$ from Lemma~\ref{splitting of V}. Define the map $f_J$, 
\begin{align*}
f_J:U&\to V/U^{\perp},\\
\lambda&\mapsto g_J(\lambda J).
\end{align*}
\end{definition}

By Lemma~\ref{splitting of V}, $f_J$ is an isomorphism. Denote the inverse by $\check{\phi}:V/U^{\perp}\to U$. Define $Y:=\Lambda/\Lambda\cap U^{\perp}$. It is a full rank lattice in $V/U^{\perp}$. Consider the real torus $\check{B}:=U/\check{\phi}(Y)$. Denote the lattice $U\cap \Lambda$ by $X^*$, and $\Hom(\Lambda\cap U, \mathbf{Z})$ by $X$. $\check{B}$ is a tropical torus with the lattice $X^*$ in the tangent bundle of $\check{B}$. The non-degenerate pairing $E(\cdot, \cdot)$ induces an injection $\phi: Y \to X$ as follows. For each $\lambda\in \Lambda, v\in X^*$ define $\phi(\lambda)(v)=E(\lambda, v)$. So $\phi(\lambda)$ is an element in $X$. Since the kernel is $U^\perp\cap \Lambda$, the morphism factors through $Y$. The data $\phi$ is equivalent to the positive symmetric pairing $\check{g}=E(\cdot J, \cdot)$ on $U$. Therefore $\phi$ is a polarization for the tropical torus $\check{B}$, and $(\check{B},\check{\phi},\phi)$ is a polarized tropical abelian variety (\cite{MZ} p.15). Denote the image $\phi(Y)$ of $Y$ by $\check{Y}$. Assume $X/\check{Y}\cong \mathbf{Z}/d_1\times\mathbf{Z}/d_2\times\ldots\times\mathbf{Z}/d_r$ for $d_i\mid d_{i+1}$. Let $\mathfrak{d}$ be the diagonal matrix
\begin{equation}\label{dmatrix}
\mathfrak{d}:=\begin{pmatrix}
d_1&&\\
& \ddots &\\
& & d_r 
\end{pmatrix}.
\end{equation}

$\mathfrak{d}$ is called the type of $\phi$. \\

\begin{corollary}
$(\check{B}, \check{\phi}, \phi)$ is a $r$-dimensional polarized tropical abelian variety of type $\mathfrak{d}$. 
\end{corollary}

\begin{definition}\label{definition of tropicalization}
Fix a rational isotropic subspace $U\subset V$, the tropical abelian variety $(\check{B}, \check{\phi}, \phi)$ constructed from $f_J$ for $J\in \mathcal{C}_0(\Sp(E,\mathbf{R}))$ is called the tropical abelian variety associated to $J$. 
\end{definition}

\begin{definition}
Let $\check{Y}\subset X$ be two lattices of the same rank $r$. A basis of $X$ $\{\alpha_1,\ldots,\alpha_r\}$ is called a compatible basis if there exists $\{d_1,\ldots,d_r\}$ such that $d_i|d_{i+1}$ and $\{d_1\alpha_1,\dots,d_r\alpha_r\}$ is a basis of $\check{Y}$. In other words, there exists an isomorphism $X/\check{Y}\cong \mathbf{Z}/d_1\times\mathbf{Z}/d_2\times\cdots\times\mathbf{Z}/d_r$, such that $\alpha_i$ is sent to the generator of $\mathbf{Z}/d_i$.
\end{definition}

\begin{proposition}\label{0-cusps}
Up to the action of $\Gamma(\delta)$, there is a unique rational boundary component that corresponds to the polarized tropical abelian varieties of type $\delta$, and this is the orbit of $F^{(0)}$. For any $0$-cusp, we have $d_1d_2\cdots d_g=\delta_1\delta_2\cdots \delta_g$.
\end{proposition}

\begin{proof}
For the first statement, it suffices to find a symplectic basis $\{\lambda_i,\mu_j\}$ such that $\{\mu_j\}$ is a basis of $\Lambda\cap U$.  Choose a compatible basis $\{y_j\}_{j\in 1,\ldots,g}$ of $X$ such that $\{x_i=\delta_iy_i\}$ is a basis of $\check{Y}$. Let's denote the sublattice generated by $\{x_{i+1},\ldots,x_g\}$ by $Y_i$. Let $\{\mu_j\}$ be the dual basis of $\{y_j\}$. Consider the map $E(\cdot, \cdot\vert_U): \Lambda\to X$, and lift $x_1$ to an element $\lambda_1\in \Lambda$. So we have $E(\lambda_1,\mu_1)=\delta_1$ and $E(\lambda_1,\mu_j)=0$ for $j\neq1$. Now recall how $\delta_i$ is defined (\cite{GH} pp. 304-305.). The set of values $E(\Lambda,\Lambda)$ is an ideal in $\mathbf{Z}$ generated by $\delta_1$. So $\delta_1$ divides $E(\lambda,\mu)$ for any $\lambda,\mu\in \Lambda$. And we have a splitting $\Lambda=\{\lambda_1,\mu_1\}\oplus \{\lambda_1,\mu_1\}^{\perp}$. We denote $\{\lambda_1,\mu_1\}^{\perp}$ by $\Lambda_1$. The restriction of $E$ to $\Lambda_1$ is still non-degenerate, thus induces an isomorphism $\Lambda_1/\Lambda_1\cap U\cong Y_1$. By the definition of $\delta_i$ we can do induction, and find all $\{\lambda_1,\lambda_2\cdots,\lambda_g\}$ such that $\{\lambda_i,\mu_j\}$ is a symplectic basis of $\Lambda$.

For the second statement, choose a compatible basis $\{y_1,\ldots,y_g\}$ of $X$. Then there exists $\{u'_1,\ldots, u'_g\}\subset \Lambda$ that are lifts of $d_iy_i$. Let $\{v_1,\ldots,v_g\}\subset U$ be the dual basis of $\{y_1,\ldots,y_g\}$. Then $\{u'_1,\ldots,u'_g,v_1,\ldots,v_g\}$ is a basis of $\Lambda$. With respect to this new basis, $E$ is the matrix
\[
\begin{pmatrix}
S& \mathfrak{d}\\
-\mathfrak{d}&0
\end{pmatrix},
\]

where $\mathfrak{d}$ is the matrix \eqref{dmatrix}, and $S$ is some skew symmetric integral matrix (nonzero if $d\neq\delta$). Because the transformation matrix between two basis $\{u'_1,\ldots,u'_g,v_1,\ldots,v_g\}$ and $\{\lambda_1,\ldots,\lambda_g,\mu_1,\ldots, \mu_g\}$ is in $\GL(2g,\mathbf{Z})$, and has determinant $\pm 1$, we have $d_1d_2\cdots d_g=\delta_1\delta_2\cdots \delta_g$ by computing the determinant. 
\end{proof}

\begin{corollary}\label{splitting}
For any maximal isotropic subspace $U\subset \mathbf{R}^{2g}$, if $U\cap\Lambda$ has an isotropic complement, i.e. there is an isotropic subspace $U'\subset \mathbf{R}^{2g}$ such that $(U\cap\Lambda)\oplus (U'\cap\Lambda)=\Lambda$, then $U=U^{(0)}M^{-1}$ for some $M\in \Gamma(\delta)$.
\end{corollary}
\begin{proof}
If $U\cap \Lambda$ has an isotropic complement $U'\cap\Lambda$ in $\Lambda$, then choose bases from $U\cap\Lambda$ and $U'\cap\Lambda$, we get a symplectic basis. 
\end{proof}

\begin{definition}
A $0$-cusp is called splitting if it is congruent to $F^{(0)}$. Otherwise it is called non-splitting. 
\end{definition}

\begin{lemma}
With $X$ and $\check{Y}\subset X$ fixed, the set of tropical abelian varieties is identified with the set of positive definite quadratic forms on $X_\mathbf{R}$, and is denoted by $\mathcal{C}(X)$. 
\end{lemma}

\begin{proof}
Since $\check{B}$ and $\phi:Y\to X$ are fixed, a polarized tropical abelian variety is equivalent to the data $\check{\phi}: Y\to U$, which is equivalent to the positive symmetric bilinear form $\langle\cdot, \check{\phi}{\phi}^{-1}(\cdot)\rangle$ on $X_\mathbf{R}$. 
\end{proof}
  
For a rational boundary component $U$, the map from $\mathcal{C}_0(\Sp(E,\mathbf{R}))$ to $\mathcal{C}(X)$ in Definition~\ref{definition of tropicalization}  is called the tropicalization, and is denoted by $\Tr(U)$ or $\Tr$. 

\begin{definition}
As in (\cite{Alex02} 5.5.1), we define 
\[
\GL(X,Y):=\{u\in \GL(X); u\check{Y}\subset \check{Y}\}.
\]
\end{definition}

Consider $X$ as a lattice in $U^*$. We identify $(S^2(U^*))^*=\Gamma^2(U)\subset U\otimes U$ with the vector space of quadratic forms on $U^*$. $\mathcal{C}(X)$ is an open cone in $\Gamma^2(U)$. $\GL(X_\mathbf{R})$ has a natural representation $\bar{\rho}$ on $\Gamma^2(U)$. $\bar{\rho}$ is injective. $\mathcal{C}(X)$ is invariant under the action of $\bar{\rho}(\GL(X_\mathbf{R}))$. Thus $\GL(X,Y)\subset\GL(X_\mathbf{R})$ is acting on $\mathcal{C}(X)$. We have 

\begin{proposition}
Let $(\check{B},\check{\phi}, \phi)$ and $(\check{B},\check{\phi}', \phi)$ be tropical abelian varieties corresponding to $Q,Q'\in \mathcal{C}(X)$ respectively. $(\check{B},\check{\phi},\phi)\cong (\check{B},\check{\phi}',\phi)$ as polarized tropical abelian varieties if and only if there is an element $u\in \GL(X,Y)$ such that $Q'=\bar{\rho}(u)(Q)$. 
\end{proposition}

\begin{definition}\label{definition of integral tropical abelian varieties}
A tropical abelian variety $(\check{B},\check{\phi}, \phi)$ is called integral if $\check{\phi}(Y)\subset X^*$. A quadratic form $Q\in \Gamma^2(U)$ is called integral if the associated symmetric bilinear form $B$ satisfies $B(\check{Y},X)\in\mathbf{Z}$. The set of integral elements is a lattice, denoted by $\mathbb{L}^*$, in $\Gamma^2(U)$. The intersection of this lattice and $\mathcal{C}(X)$ is the set of integral tropical abelian varieties. 
\end{definition}

Note that $\Gamma^2X\subset \mathbb{L}^*$. It is convenient to write down everything in terms of a fixed basis. Define 
\[
\GL(\mathfrak{d}):=\{u\in \GL(r,\mathbf{Z}): \mathfrak{d}u\mathfrak{d}^{-1}\in\GL(r,\mathbf{Z})\}.
\]

With respect to a compatible basis of $X$, $\GL(X,Y)= \GL(\mathfrak{d})^T$. $\mathcal{C}(X)\cong \{Q\in \Sym_r(\mathbf{R}); Q>0\}$. The action is $Q\mapsto (u^T)^{-1}Qu^{-1}$. The set of integral tropical abelian varieties corresponds to $\{Q\in \mathcal{C}(X); Q\mathfrak{d}\in \M(r, \mathbf{Z})\}$. With respect to the corresponding basis of $\check{Y}$, $\GL(X,Y)\cong \GL(\mathfrak{d})$. 

Given two cusps $F$ and $F'$, and an element $R\in\Sp(E,\mathbf{Q})$ such that $F'=R(F)$, $R$ gives a map $\mathcal{C}(X)\to \mathcal{C}(X')$, still denoted by $R$, as follows: Since $U'=UR^{-1}$, $R^{-1}\vert_U: U\to U'$ is a linear isomorphism. It induces an isomorphism $R^{-1}\vert_U\otimes R^{-1}\vert_U: \Gamma^2(U) \to \Gamma^2(U')$. This linear isomorphism sends $\mathcal{C}(X)$ to $\mathcal{C}(X')$. 
\begin{lemma}
We have the following commutative diagram
\[
\begin{diagram}
\mathcal{C}_0(\Sp(E,\mathbf{R}))&\rTo^R&\mathcal{C}_0(\Sp(E,\mathbf{R}))\\
\dTo^\Tr&&\dTo^\Tr\\
\mathcal{C}(X)&\rTo^R&\mathcal{C}(X')
\end{diagram}.
\]
The first line is the conjugation by $R$. Furthermore, if $R\in \Gamma(\delta)$, then $\Tr(J)$ is isomorphic to $\Tr(RJ)$ as polarized tropical abelian varieties. 
\end{lemma}

\begin{proof}
First $R\in \Sp(E,\mathbf{Q})$. Notice that $E(\cdot J,\cdot)\vert_{U}=\langle \phi\check{\phi}^{-1}(\cdot),\cdot\rangle$ on $U$. For any $J\in \mathcal{C}_0(\Sp(E,\mathbf{R}))$, $R(J)=RJR^{-1}$. $\Tr(J)$ is represented by $E(\cdot J,\cdot)\vert_U$. The bilinear form $\Tr(R(J))$ is  
\[
E(\cdot R(J),\cdot)\vert_{R(U)}=E(\cdot RJR^{-1}, \cdot)\vert_{UR^{-1}}=E(\cdot RJ,\cdot R)\vert_{UR^{-1}}
\]

This is the same as the bilinear form induced by $R^{-1}$ in the definition.

Now $R\in \Gamma(\delta)$. Because $\cdot R^{-1}$ is a bijection $\mathbf{R}^{2g}\to \mathbf{R}^{2g}$, $(S\cap U)R^{-1}=SR^{-1}\cap UR^{-1}$ for any subset $S\subset\mathbf{R}^{2g}$. Since $R\in \Gamma(\delta)$, $(\Lambda\cap U)R^{-1}=\Lambda R^{-1}\cap UR^{-1}=\Lambda\cap U'$ and $(\Lambda\cap U^{\perp})R^{-1}=\Lambda\cap U'^{\perp}$. Consider the following commutative diagram for $J'=RJR^{-1}$, 
\[
\begin{diagram}
\Lambda\cap U &\rTo^{\cdot J}&UJ\oplus U^{\perp}&\rTo&V/U^{\perp}\\
\dTo^{\cdot R^{-1}}&&\dTo^{\cdot R^{-1}}&&\dTo^{\cdot R^{-1}}\\
\Lambda\cap U'&\rTo^{\cdot J'}&UJ'\oplus U'^{\perp}&\rTo &V/(U')^{\perp}
\end{diagram}.
\]

It shows that $f_{J'}(\Lambda\cap U')=(f_{J}(\Lambda\cap U))R^{-1}$ and $\check{B}\cong \check{B'}$ as tropical tori. The polarization is preserved because of the first part.
\end{proof}

If $F=F'$, we can apply the above proof to $M\in \mathcal{P}(F)$, and get

\begin{corollary}
Fix a cusp $F$. The action of $\mathcal{P}(F)$ on $\mathcal{C}(X)$ is denoted by $\rho_X$. Then $\Tr:\mathcal{C}_0(\Sp(E,\mathbf{R}))\to \mathcal{C}(X)$ is $\mathcal{P}(F)$-equivariant. The image $\rho_X(P(F))$ is inside $\GL(X,Y)$. 
\end{corollary}

Let's consider the map $\Tr: \mathfrak{S}_g\to \mathcal{C}(X)$ for the cusps $F^{(g')}$. For $F^{(g')}$, $U^{(g')}=\{(0,x); 0\in\mathbf{R}^{g+g'},x\in\mathbf{R}^{g-g'}\}$. We have the natural basis for every lattice. Using the basis of $X$, that is $\{y_i\}_{i>g'}$, we regard $\mathcal{C}(X)$ as an open cone in $\Sym_{g-g'}(\mathbf{R})$. Assume $\tau\in\mathfrak{S}_g$ corresponds to $J\in\mathcal{C}_0(\Sp(E,\mathbf{R}))$. Since $E(\cdot J,\cdot)=\Re{H}$, and with respect to the coordinates $z_i=(x\tau+y\delta)_i$, $H$ is $(\Im{\tau})^{-1}$, the restriction of $E(\cdot J, \cdot)$ to $U^{(g')}$ is 
\[
E(\cdot J,\cdot)\vert_{U^{(g')}}=H\vert_{U^{(g')}}=\sum_{g'<i,j\leqslant g}\delta_iy_i((\Im{\tau})^{-1})_{ij}\delta_jy_j.
\]

It follows that the matrix for $\Tr(\tau)$ is the inverse matrix of $T=((\Im{\tau})^{-1})_{g'<i,j\leqslant g}$ with respect to the basis of $\check{Y}$. We write $\tau$ in blocks, where $\tau_1$ is a $g'\times g'$ matrix,
\[
\tau=
\begin{pmatrix}
\tau_1 &\tau_3\\
\tau_3^T& \tau_2\\
\end{pmatrix},
\]

\begin{lemma}
Assume given an invertible matrix and its inverse, both in blocks,
\[
\begin{pmatrix}
A&B\\
C&D
\end{pmatrix}
=\begin{pmatrix}
a&b\\
c&d
\end{pmatrix}^{-1}.
\]

If $A$ and $D$ are both invertible, then 
\[
D^{-1}=d-ca^{-1}b.
\]
\end{lemma}
\begin{proof}
We have 
\begin{align*}
Aa+Bc&=Id,\\
Ab+Bd&=0,\\
Ca+Dc&=0,\\
Cb+Dd&=Id.
\end{align*}

Therefore
\[
D(d-ca^{-1}b)=Dd-Dca^{-1}b=Id-Cb+Caa^{-1}b=Id-Cb+Cb=Id.
\]
\end{proof}

\begin{proposition}
With respect to the basis $\{\delta_iy_i\}_{i>g'}$ of $\check{Y}$, 
\[
\Tr(\tau)=\Im{\tau_2}-(\Im{\tau_3^T})(\Im{\tau_1})^{-1}\Im{\tau_3}.
\]
\end{proposition}
\begin{proof}
Since $\tau$ is positive definite, $\tau_1$ and $\tau_2$ are positive definite. So we can apply the lemma.
\end{proof}

Now we can compare $\mathcal{C}(X)$ with $\mathcal{C}(F)$ for $F=F^{(g')}$. 

\begin{proposition}\label{splitting component}
For $F=F^{(g')}$, there exists an isomorphism $h(F):\mathcal{P}'(F)\to \Gamma^2(U)$ restricting to a bijection $h(F):\mathcal{C}(F)\to \mathcal{C}(X)$, such that the following diagram commutes,
\[
\begin{diagram}
\mathfrak{S}_g&\rTo^{=}&\mathcal{C}_0(\Sp(E,\mathbf{R}))\\
\dTo^{\Phi(F)} &&\dTo^{\Tr(F)} \\
\mathcal{C}(F)&\rTo^{h(F)}&\mathcal{C}(X)
\end{diagram},
\]

where the map $\Phi(F)$ is defined in (\cite{Nam80}, p. 31). 
Moreover, the following statements are true.
\begin{itemize}
\item[a)] Every map in the diagram is $\mathcal{P}(F)$-equivariant.
\item[b)] $\Tr$ is surjective. 
\item[c)] Denote the induced isomorphism $\Aut(\Gamma^2(U))\to \Aut(\mathcal{P}'(F))$ by $\rho_h$, then $\rho_h(\GL(X,Y))=\overline{P}(F)$.
\item[d)] Thus $\rho_X(P(F))=\GL(X,Y)$. 
\item[e)] $P'(F)\cap \mathcal{C}(F)$ is identified with the integral tropical abelian varieties by $h(F)$.
\end{itemize}
\end{proposition}

\begin{proof}
We identify $\Gamma^2(U)$ with $\Sym_{r=g-g'}(\mathbf{R})$ by the above basis $\{\delta_iy_i\}_{i>g'}$. We have also identified $\mathcal{P}'(F^{(g')})$ with $\Sym_{r}(\mathbf{R})$ by $[\cdot]$ in Example~\ref{groups in the principal cusp}. $h(F)$ is defined to be the composition of these two isomorphisms. By (\cite{Nam80}, p. 31 ii), $\Phi(\tau)=\Im{\tau_2}-(\Im{\tau_3^T})(\Im{\tau_1})^{-1}\Im{\tau_3}=\Tr(\tau)$. Thus the diagram commutes. 

Since $\Phi$ and $\Tr$ are both $\mathcal{P}(F)$-equivariant, and are both surjective, $h$ is also $\mathcal{P}(F)$-equivariant, and we get a). For c), we have computed $\overline{P}(F)=\GL(\delta_{g'})=\GL(X,Y)$ under the above identification. For e), $Q\in P'(F)$ if and only if $\delta_{g'}^{-1}Q\in \M(g-g',\mathbf{Z})$. Remember we are using the basis of $\check{Y}$. 
\end{proof}

Assume $F'=M(F)$, for $M\in \Sp(2g,\mathbf{Q})$. 

\begin{lemma}
The following diagram
\[
\begin{diagram}
\mathfrak{S}_g &\rTo^{M}& \mathfrak{S}_g\\
\dTo^{\Phi}&& \dTo^{\Phi} \\
\mathcal{C}(F)&\rTo^{M}&\mathcal{C}(F')
\end{diagram}
\]

commutes. 
\end{lemma}

\begin{proof}
Recall the definition of $\Phi$ in \cite{Nam80}. Embed $\mathfrak{S}_g\cong \mathcal{P}(F)/(\mathcal{P}(F)\cap K)$ into $\mathfrak{S}(F)\cong \mathcal{P}'(F)_\mathbf{C}\mathcal{P}(F)/(\mathcal{P}(F)\cap K)$. The map $\Phi: \mathfrak{S}(F)\to \mathcal{P}'(F)$ is the composition

\[
\begin{array}{ccccc}
 \mathcal{P}'(F)_\mathbf{C}\mathcal{P}(F)/(\mathcal{P}(F)\cap K)&\longrightarrow&  \mathcal{P}'(F)_\mathbf{C}\mathcal{P}(F)/\mathcal{P}(F)&\longrightarrow& \mathcal{P}'(F)\\
 u\cdot p \mod \mathcal{P}(F)\cap K &\longmapsto& u\cdot p \mod \mathcal{P}(F)=u &\longmapsto& \Im{u}
\end{array}
\]

Assume $\tau=u\cdot p \mod \mathcal{P}(F)\cap K \in \mathfrak{S}_g$, 
\[
M(\tau)=M(u\cdot p \mod \mathcal{P}(F)\cap K)=MuM^{-1}\cdot MpM^{-1} \mod M(\mathcal{P}(F)\cap K)M^{-1} .
\]

Because $M$ is a real matrix, 
\[
\Phi(M\tau)=\Im{MuM^{-1}}=M\Im{u}M^{-1}=M\Phi(\tau).
\]
\end{proof}

For any cusp $F$, there exists $M\in\Sp(2g,\mathbf{Q})$ such that $F=M(F^{(g')})$ for some $g'$. By definition, $\mathcal{P}'(F)=M\mathcal{P}'(F^{(g')})M^{-1}$. We can now define $h(F):\mathcal{P}'(F)\to\Gamma^2(U)$ to be the unique map that makes the following diagram commute,
\[
\begin{diagram}
\mathcal{P}'(F^{(g')})&\rTo^M&\mathcal{P}'(F)\\
\dTo^{h(F^{(g')})} &&\dTo^{h(F)}\\
\Gamma^2(U^{(g')})&\rTo^R&\Gamma^2(U).
\end{diagram}
\]

The definition of $h(F)$ is independent of the choice of $M$, because $h(F^{(g')})$ is $\mathcal{P}(F^{(0)})$-equivariant. $h(F)$ is an isomorphism. Most of the statements for $h(F^{(g')})$ can be generalized to a general $h(F)$.

\begin{proposition}\label{interpretation of the cone}
The following diagram commutes for any cusp $F$.

\[
\begin{diagram}
\mathfrak{S}_g&\rTo^{=}&\mathcal{C}_0(\Sp(E, \mathbf{R}))\\
\dTo^{\Phi(F)}&&\dTo^{\Tr(F)} \\
\mathcal{C}(F)&\rTo^{h(F)}&\mathcal{C}(X)
\end{diagram}
\]

Moreover, the following statements are true.
\begin{itemize}
\item[a)] Every map in the diagram is $\mathcal{P}(F)$-equivariant.
\item[b)] $\Tr$ is surjective. 
\item[c)] $\rho_X(F): \mathcal{P}(F)\to \GL(X_\mathbf{R})$ is surjective. 
\item[d)] Denote the induced isomorphism $\Aut(\Gamma^2(U))\to \Aut(\mathcal{P}'(F))$ by $\rho_h$, then $\overline{P}(F)\subset \rho_h(\GL(X,\check{Y}))$ is a subgroup of finite index.
\end{itemize}
\end{proposition} 
\begin{proof}
We have the following diagram

\begin{diagram}
\mathcal{C}(F^{(g')})  &                  &\rTo^M&                  & \mathcal{C}(F)  \\
\dTo(0,4)^{h(F^{(g')})}&\luTo^\Phi&            &\ruTo^\Phi& \dTo(0,4)_{h(F)}\\
                                       &           &\mathfrak{S}_g&           &                             \\
                                       & \ldTo_{\Tr}&            &\rdTo_{\Tr}&                              \\
\mathcal{C}(X^{(g')})   &                  &\rTo^R&                  &\mathcal{C}(X).
\end{diagram} 

Imagine the diagram as a pyramid with five faces. The diagram concerning us is the triangular face on the right. It commutes because all the other faces commute and the arrows in the bottom square are all bijections. Then a), b) and c) follow.

For d). Since $\rho_X$ is surjective and defined over $\mathbf{Q}$, by (\cite{Bor69} 8.11), the image of the arithmetic subgroup $P(F)$ is an arithmetic subgroup of $\GL(X_\mathbf{R})$. So $\rho_X(P(F))$ is commensurable with $\GL(X,Y)$. Because $\rho_X(P(F))$ is a subgroup of $\GL(X,Y)$, it is a subgroup of finite index. $\overline{P}(F)=\text{Ad}(\mathcal{P}(F))=\rho_h\rho_X(P(F))$. 
\end{proof}

Let $F$ be a $0$-cusp that corresponds to a maximal isotropic space $U$. Fix a basis $\{v_1,\ldots,v_g\}$ of $U\cap\Lambda=X^*$. Take any $R\in \Sp(E,\mathbf{Q})$, such that $U=R\cdot U^{(0)}=U^{(0)}R^{-1}$, and $R^{-1}$ maps $\{\mu_1,\ldots,\mu_g\}$ in $U^{(0)}$ to $\{v_1,\ldots,v_g\}$. Such $R$ always exists. It suffices to find a rational Lagrangian complement $L$ of $U$. Extend the basis $\{v_1,\ldots,v_g\}$ of $U\cap\Lambda$ to a basis $\{u_1',\ldots,u_g', v_1,\ldots,v_g\}$ of $\Lambda$. Assume under this basis, 
\[
E=\begin{pmatrix}
S&\mathfrak{d}\\
-\mathfrak{d}&0
\end{pmatrix}.
\]

Here $S$ is an integral skew-symmetric matrix. Denote the transformation matrix of the bases by $M'^{-1}$. $M'\in\GL(2g,\mathbf{Z})$. Assume
\[
\begin{pmatrix}
A & B\\
0 & I_g
\end{pmatrix}
M'^{-1}=R^{-1}
\]

Since 
\[
M'^{-1}\begin{pmatrix}
0 &\delta \\
-\delta & 0
\end{pmatrix}
(M'^{-1})^{T}=
\begin{pmatrix}
S & \mathfrak{d}\\
-\mathfrak{d} & 0
\end{pmatrix},
\]

we have 
\begin{align}
S&=A^{-1}B\mathfrak{d}-(A^{-1}B\mathfrak{d})^T\\
A&=\delta \mathfrak{d}^{-1}
\end{align}

We can always change $R$ by
\[
\begin{pmatrix}
I_g & \delta Q'\\
0 & I_g
\end{pmatrix}R,
\]

where $Q'$ is a symmetric rational matrix. Therefore we can always choose $R$ such that $A^{-1}B\mathfrak{d}$ is antisymmetric, that is
\begin{align}
A^{-1}B\mathfrak{d}&=\frac{1}{2}S\\
A&=\delta\mathfrak{d}^{-1}
\end{align}

\begin{corollary}\label{the lattice in tropical cone}
The lattices $P'(F)$ and $\mathbb{L}^*$ agree. In particular, the set $\mathcal{C}(F)\cap P'(F)$ is the set of integral tropical abelian varieties for every $F$.
\end{corollary}
\begin{proof}
It suffices to prove it for maximal corank boundary components. Let $U$ be a rational maximal isotropic subspace. $Q$ corresponds to an element in $\mathcal{C}(F)\cap P'(F)$ if and only if 
\[
R\begin{pmatrix}
I_g & \delta Q\\
0& I_g
\end{pmatrix} R^{-1}\in\GL(2g,\mathbf{Z}).
\]

Since $M'\in\GL(2g,\mathbf{Z})$, it is equivalent to 
\[
\begin{pmatrix}
A^{-1}&-A^{-1}B\\
0 & I_g
\end{pmatrix}
\begin{pmatrix}
I_g &\delta Q\\
0 & I_g
\end{pmatrix}
\begin{pmatrix}
A&B\\ 0&I_g
\end{pmatrix}
\in\GL(2g,\mathbf{Z}).
\]

This is true if and only if $\mathfrak{d}Q\in \mathbf{Z}$, which means $Q$ represents an integral tropical abelian variety in $\mathcal{C}(X)$. 
\end{proof}

%%%%%%%%%%%%%%%%%%%%%%%%%%%%%%%%%%%%%%%%%%%%%%%%%%%%%%%%%%%%%%%%%%%%%%%%%%%%%%%%%
%%%%%%%%%%% mirror symmetry %%%%%%%%%%%%%%%%%%%%%%%%%%%%%%%%%%%%%%%%%%%%%%%%%%%%%%%%%%%%%
%%%%%%%%%%%%%%%%%%%%%%%%%%%%%%%%%%%%%%%%%%%%%%%%%%%%%%%%%%%%%%%%%%%%%%%%%%%%%%%%%

\subsection{Mirror Symmetry for Abelian Varieties}\label{2.3}
\textbf{Warning}: We use $X$ (resp. $Y$) to denote both the complex torus and the lattice. We hope there is no confusion. Fix a complex number $t$ in the upper half plane $\mathbf{H}$, and the polarization $E$ on an abelian variety $X$, we get a symplectic manifold $(X,\Omega)$, with $\Omega=tE$ a complexified K\"ahler form. Let $V=\H_1(X,\mathbf{R})$, and $\Lambda=\H_1(X,\mathbf{Z})$. Any maximal rational isotropic subspace $U$ of $V$ gives a linear Lagrangian for $(X,\Omega)$. $Y:=\Lambda/\Lambda\cap U$ is a lattice in $V/U$. Define the affine torus $B:=(V/U)/Y$. A choice of $U$ gives a Lagrangian fiberation $f:X\to B$. We follow the construction in (\cite{Fu2002},\cite{Po} Sec. 6.2) to define the mirror torus $(Y_t,J_t)$. For $t=a+bi$, define $\Omega'=(ai-b)E$. Define $\phi: V\to V^*$ by $E(v,w)=\phi(v)(w)$. Define a complex structure $J_\Omega$ on $V\oplus V^*$ by requiring that the function
\[
I_x(v,f):=\Omega'(x,v)+if(x)
\]

is complex linear for each $x\in V$ (\cite{Fu2002} p. 11). It can be checked that 
\[
J_\Omega=
\begin{pmatrix}
b^{-1}a & -b^{-1}\phi^{-1}\\
(b+a^2b^{-1})\phi & -ab^{-1}
\end{pmatrix}
\]

satisfies the condition. The subspace of $V^*$ that consists of linear functions vanishing on $U$ is denoted by $\ann(U)$. $J_\Omega$ preserves $U\oplus \ann(U)$, and thus descends to a complex structure $J_t$ on 
\[
\check{V}:=(V\oplus V^*)/(U\oplus \ann(U))=V/U\oplus U^*.
\]
 
Recall lattices $Y=\Lambda/\Lambda\cap U^\perp$ and $X=\Hom(\Lambda\cap U,\mathbf{Z})$. Define the lattice $\Gamma:=Y\oplus X\subset V/U\oplus U^*$. The mirror torus $Y_t$ is defined to be the complex torus $(\check{V}/\Gamma, J_t)$ (\cite{Fu2002} Definition 1.17  \& 1.18). There is a natural dual fibration $\check{f}: Y_t\to B$. We continue to use the notations in the proof of Proposition~\ref{0-cusps}. Fix a compatible basis $\{y_1,\ldots,y_g\}$ of $X$. $\{u'_1,\ldots, u'_g\}\subset \Lambda$ are lifts of $d_iy_i\in X$, and $\{v_1,\ldots,v_g\}\subset X^*$ is the dual basis of $\{y_1,\ldots,y_g\}$. Let $\check{u}'_i\in Y$ be the image of $-u'_i$, and $\check{v}_i=y_i\in X$ be the dual of $v_i\in X^*$. Then 
\[
J_t \check{u}'_i=b^{-1}a\check{u}'_i -(b+a^2b^{-1})d_i\check{v}_i.
\]

Define $\check{e}_i=d_i\check{v}_i$, then $\check{u}'_i=(a+bJ_t)\check{e}_i$. With respect to the bases $\{\check{e}_i\}$ and $\{\check{u}'_i,\check{v}_j\}$, the period matrix is
\begin{equation*}
\begin{pmatrix}
tI_g\\
\mathfrak{d}^{-1}
\end{pmatrix}.
\end{equation*}

Therefore $Y_t=E_1\times E_2\times \ldots\times E_g$, where $E_i:=\mathbf{C}/d_i^{-1}\mathbf{Z}+t\mathbf{Z}$ is an elliptic curve. We write $E_i$ as $\mathbf{C}^*/(q^{d_i})^{\mathbf{Z}}$, for $q=e^{2\pi i t}$.  Take $\Delta^*$ small enough, so that $E_i$ has no complex multiplication for $q\in \Delta^*$. The fiber product of elliptic curves $\mathcal{Y}_\eta=\mathcal{E}_1\times \mathcal{E}_2\times\cdots \mathcal{E}_g\to \Delta^*$ is defined to be the mirror family for the maximal degeneration in the direction $U$.

While we can define such a mirror family for any $0$-cusp $F$, for general $F$, the mirror is not a space, but a gerbe. If the mirror is a space, then there should exist a Lagrangian section which corresponds to the structure sheaf $\mathcal{O}_{Y}$. This is equivalent to (\cite{Fu2002} Assumption 1.1), which says that there is an isotropic subspace $L\subset V$ such that $(U\cap\Lambda)\oplus(L\cap\Lambda)=\Lambda$. By Corollary \ref{splitting}, then $U$ is in the orbit of $F^{(0)}$. 

\subsubsection{the Splitting boundary component}
We consider the mirror symmetry for the splitting cusp $F^{(0)}$ first. In this case, use the standard symplectic basis $\{\lambda_1,\ldots, \lambda_g,\mu_1,\ldots\mu_g\}$ of $\Lambda$ instead of $\{u_i',v_j\}$ and use $\{\check{\lambda}_i,\check{\mu}_j\}$ instead of $\{\check{u}_i',\check{v}_j\}$. Let $L$ be the subspace generated by $\{\lambda_1,\ldots,\lambda_g\}$. Therefore $L$ is a Lagrangian section for the fibration $X\to B$, and we can identify $V$ with $U\oplus V/U$. 
\begin{remark}
By (\cite{GLO} proposition 9.6.1), in this case, for each $\tau\in \mathfrak{S}_g$, there exists $\omega_\tau$, such that $(Y_t,\omega_\tau)$ is mirror symmetric to the algebraic pairs $(X_\tau,tE)$ as defined in loc. cit. So this definition of the mirror agrees with the mirror in \cite{GLO}.  
 \end{remark}
  
A Chern class of a line bundle over an abelian variety is represented by a skew symmetric integral form of $\omega$ on $\Gamma$. Assume, with respect to the basis $\{\check{\lambda}_1,\ldots,\check{\lambda}_{g},\check{\mu}_1,\ldots,\check{\mu}_g\}$ of the lattice $\Gamma$, 
\begin{equation*}
\omega=
\begin{pmatrix}
Q_1&Q_2\\
Q_3&Q_4
\end{pmatrix}
\in \M(2g, \mathbf{Z}).
\end{equation*}

By the Riemann Conditions  $\omega(J_tx,J_ty)=\omega(x,y)$ and $\omega(J_tx,x)>0, \forall x\neq 0$, we have 
\[
Q_1=Q_4=0;\  Q_2\delta=-\delta Q_3; \ Q_2=-Q_3^T; \ Q_3>0.
\]

Let $Q=Q_3\delta^{-1}$ be a positive definite symmetric matrix. 
\[
\omega=
\begin{pmatrix}
0&-\delta Q\\
Q\delta &0
\end{pmatrix},
\]
\[
\omega=\sum_{i,j=1}^{g}-\delta_i Q_{ij} \ud\check{x}_i\wedge \ud\check{y}_j,
\]
where $\{\check{x}_i, \check{y}_j\}$ are the dual of $\{\check{\lambda}_i,\check{\mu}_j\}$, and $\ud\check{x}_i\wedge \ud\check{y}_j=\ud\check{x}_i\otimes \ud\check{y}_j-\ud\check{y}_j\otimes \ud\check{x}_i$.

It follows that $X=\oplus\mathbf{Z}\check{\mu}_i$ is a maximally isotropic subgroup with respect to any $\omega$. Define $\phi:Y\to X$ by $\check{\lambda}_i\mapsto \delta_i\check{\mu}_i$ and identify $Y$ with $\phi(Y)$. $\omega$ descends to a bilinear form $X\times Y \to \mathbf{Z}$. This bilinear form is symmetric and positive definite over $U^*_\mathbf{R}$, and thus corresponds to a positive quadratic form $Q$.  We see that the ample cone $\K(Y_t)$ is the set of all positive quadratic forms on $X$. Therefore, we have the mirror map $\mathcal{C}^{(0)}\cong \mathcal{C}(X)\cong \K(Y_t)$. Under this isomorphism, integral tropical abelian varieties are identified with the integral polarizations, i.e. $\NS(Y_t)=\mathbb{L}^*$ in Definition~\ref{definition of integral tropical abelian varieties}. They are both represented by positive symmetric maps $\check{\phi}\phi^{-1}: X\to X^*$. 
\begin{remark}
It is straightforward to check that the mirror map $\mathcal{C}(X)\to \mathcal{K}(Y_t)$, identifying $Q$, is given by the Fourier transform $\Four(L_{\check{\phi}},\mathcal{L})$ defined in (\cite{Po} Section 6.3 (6.3.1)), where $L_{\check{\phi}}$ is the Lagrangian determined by the graph of $\check{\phi}$. 
\end{remark}

Let $u\in \GL(g,\mathbf{C})$ be the representation of an automorphism under the basis $\{\check{e}_i\}$. Under the basis $\{\check{\lambda}_i, \check{\mu}_j\}$, the matrix is 
\[
\begin{pmatrix}
 u&0\\
0& \delta u\delta^{-1}\\
\end{pmatrix}.
\]

Therefore $u\in \GL(\delta)$ under the basis of $Y$. The group of automorphisms of $Y_t$ is  $\GL(X,Y)$. The action on the quadratic form $Q$ is $Q'=(u^T)^{-1}Q u^{-1}$. It is identified with the action of $\GL(X,Y)$ on $\mathcal{C}(X)$.

We construct the relative minimal models $\mathcal{Y}/\Delta$ of the mirror family $\mathcal{Y}_\eta/\Delta^*$. A projective morphism $\check{\pi}:\mathcal{Y}\to\Delta$ is a relative minimal model if $\mathcal{Y}$ is $\mathbf{Q}$-factorial, terminal, and $K_\mathcal{Y}$ is $\check{\pi}$-nef. We use the construction in (\cite{Mum72} Sect. 6). Replace $\Delta^*$ by a complete discrete valuation ring $(R,\mathfrak{m})$. Let $K$ be the quotient field, and $\kappa=R/\mathfrak{m}$ the residue field. The base is $S=\spec R$. The closed point is $s=\spec \kappa$, and the generic point is $\eta=\spec K$. Consider $X_\mathbf{R}=U^*$ as an affine plane of height $1$ in the vector space $\mathbb{X}_\mathbf{R}$. 

\begin{definition}
An integral paving $\mathscr{P}$ of $U^*$ is called $Y$-invariant if $\mathscr{P}$ is invariant under the translation action of $\phi(Y)$ on $U^*$. Assume $\mathscr{P}$ is further a triangulation. It is called minimal if each point of $X$ is contained in $\mathscr{P}$, i.e. $X\subset \mathscr{P}$. 
\end{definition}

By (\cite{CLS} Exercise 8.2.14. \& Proposition 11.4.12.), the minimal condition is necessary and sufficient for the relative complete model to be normal, $\mathbf{Q}$-factorial, $\mathbf{Q}$-Gorenstein, and terminal. Fix a $Y$-invariant integral paving $\mathscr{P}$. For each cell $\sigma\in\mathscr{P}$, construct the cone $C(\sigma)$ in $\mathbb{X}_\mathbf{R}$. The collection $\{C(\sigma)\}_{\sigma\in\mathscr{P}}$ forms a fan denoted by $\Sigma_\mathscr{P}$. Denote the infinite toric embedding $X_{\Sigma_\mathscr{P}}$ by $\widetilde{\mathcal{Y}}_\mathscr{P}$. Since the fan $\Sigma_\mathscr{P}$ contains a basis of $\mathbb{X}$, $\widetilde{\mathcal{Y}}_\mathscr{P}$ is simply connected as a complex analytic space. An embedding of $\check{\mathbb{T}}:=\mathbb{G}_m^g\times S$ into $\widetilde{\mathcal{Y}}_\mathscr{P}$ is given by the vertical ray $C(\{0\})$ for $0\in X$. To define a relatively ample line bundle over $\widetilde{\mathcal{Y}}_\mathscr{P}$, with actions of $Y$ and $\check{\mathbb{T}}$, we need additional data: a piecewise affine function $\varphi$ on $U^*$ such that
\begin{itemize}
\item[1)] For each cell $\sigma\in\mathscr{P}$, $\varphi\vert_\sigma$ is affine. We say that $\varphi$ is compatible with $\mathscr{P}$.
\item[2)] $\varphi$ is strictly convex. That means the lower boundary of the graph of $\varphi$ gives $\mathscr{P}$.
\item[3)] $\varphi$ is $Y$-quasiperiodic.
\item[4)] $\varphi$ is integral. That means for each $\sigma\in \mathscr{P}$, $\varphi\vert_\sigma$ is an element in $Aff(X,\mathbf{Z})$. 
\end{itemize}

The set of functions which satisfy 1), 2) and 3) is denoted by $CPA^Y(\mathscr{P},\mathbf{R})$. The set of functions which satisfy all 4 conditions is denoted by $CPA^Y(\mathscr{P},\mathbf{Z})$. The definition of quasiperiodic functions is Definition~\ref{definition of quasiperiodic functions} in Appendix~\ref{A}. The associated group of $CPA^Y(\mathscr{P},\mathbf{R})$ is $\Gamma(B,\mathcal{P}A/\mathcal{A}ff)$, where $\mathcal{P}A$ is the sheaf of piecewise affine functions.

By Lemma~\ref{quasiperiodic and quadratic} in Appendix~\ref{A}, $\varphi$ should satisfy 
\[
\varphi(y+\phi(\lambda))=\varphi(y)+A(\lambda)+\langle y,\check{\phi}(\lambda)\rangle,
\]

for some quadratic function $A$ on $Y$ and some linear function $\check{\phi}$, such that
\[
\langle\phi(\lambda), \check{\phi}(\mu)\rangle=A(\lambda+\mu)-A(\lambda)-A(\mu).
\]

If $\varphi\in CPA^Y(\mathscr{P},\mathbf{Z})$, it defines an ample line bundle $\widetilde{\mathcal{L}}$ on $X_{\Sigma_\mathscr{P}}$, with actions of $Y$ and $\check{\mathbb{T}}$. The data $(\mathscr{P}, \varphi, \phi,\check{\phi},A)$ gives rise to a relative complete model for the family $\mathcal{Y}_\eta\to \spec K$. Thus we can take the quotient of $(\widetilde{\mathcal{Y}}_\mathscr{P},\widetilde{\mathcal{L}})$ by $Y$ as in (\cite{Mum72} Sect. 3). The quotient $(\mathcal{Y}_\mathscr{P},\mathcal{L})$ is a relative minimal model over $S$, which contains a semiabelian group scheme as a dense open subscheme, because $K_{\mathcal{Y}_\mathscr{P}}$ is equivalent to the fiber, and is $\check{\pi}$-trivial. We use the discrete Legendre transform to write down the action of $Y$ on $\widetilde{\mathcal{L}}$.%We call $\widetilde{\mathcal{Y}}_\mathscr{P}$ the universal covering of $\mathcal{Y}_\mathscr{P}$, because there exists a covering map in the complex topology.  \\  
\begin{remark}
According to Gross--Siebert program, the mirror symmetry on the tropical level is the discrete Legendre transform. The discrete Legendre transform of $(B:=U^*/\phi(Y),\mathscr{P}, \varphi)$ is the tropical affine abelian variety with a divisor $(\check{B},\check{\mathscr{P}},\check{\varphi})$.
\end{remark}

\begin{definition}
Define the function over $U$,
\begin{equation}
\check{\varphi}(\mu):=-\inf_{y\in U^*}\{\varphi(y)+\langle y, \mu\rangle\}.
\end{equation}
\end{definition}

\begin{lemma}\label{tropical divisor}
\[
\check{\varphi}(\mu+\check{\phi}(\lambda))=\check{\varphi}(\mu)+\langle \phi(\lambda),\mu\rangle+A(\lambda),\ \forall \lambda\in Y. 
\]
\end{lemma}
\begin{proof}
\begin{align*}
\check{\varphi}(\mu+\check{\phi}(\lambda))&=-\inf_{y\in U^*}\{\varphi(y)+\langle y,\mu+\check{\phi}(\lambda)\rangle\}\\
&=-\inf_{y\in U^*}\{\varphi(y)+\langle y,\check{\phi}(\lambda)\rangle+A(\lambda)+\langle y,\mu\rangle-A(\lambda)\}\\
&=-\inf_{y\in U^*}\{\varphi(y+\phi(\lambda))+\langle y+\phi(\lambda),\mu\rangle\}+\langle \phi(\lambda),\mu\rangle+A(\lambda)\\
&=\check{\varphi}(\mu)+\langle \phi(\lambda),\mu\rangle+A(\lambda).
\end{align*}
\end{proof}

By toric geometry, the line bundle $\widetilde{\mathcal{L}}$ corresponds to a polyhedron in $\mathbb{X}^*_\mathbf{R}$. The vertices of the polyhedron are $(\mu,\check{\varphi}(\mu))$. Therefore, we have

\begin{corollary}\label{the action on the homogeneous coordinate ring}
The $Y$-action on the homogeneous coordinate ring is given as follows. Define $\zeta_\mu:=\mathrm{X}^\mu q^{\check{\varphi}(\mu)}\theta$ for $\mu\in X^*$ to be a monomial section of $\widetilde{\mathcal{L}}$. Then 
\begin{equation}
S^*_{\lambda}(\zeta_\mu)=\zeta_{\mu+\check{\phi}(\lambda)}=\mathrm{X}^{\mu+\check{\phi}(\lambda)}q^{\check{\varphi}(\mu+\check{\phi}(\lambda))}\theta=\mathrm{X}^\mu(\lambda)q^{A(\lambda)}\mathrm{X}^{\check{\phi}(\lambda)}\zeta_\mu.
\end{equation}
\end{corollary}

\begin{corollary}\label{Affine means trivial}
If $\psi=\varphi-\varphi'$ is an integral affine function, then the induced line bundles $\mathcal{L}$ and $\mathcal{L}'$ by $\varphi$ and $\varphi'$ are isomorphic on $\mathcal{Y}_\mathscr{P}$. 
\end{corollary}

\begin{proof}
$\psi$ defines an isomorphism between the homogeneous coordinate rings of $X_{\Sigma_\mathscr{P}}$. If $\psi(y)=\langle y,a\rangle +h$, then $\check{\varphi}(\mu)=\check{\varphi}'(\mu+a)-h$. The isomorphism is a translations by $-(a,h)$. It is $Y$-equivariant, and induces an isomorphism between $\mathcal{L}$ and $\mathcal{L}'$.
\end{proof}

It follows from Corollary~\ref{Affine means trivial} that we have a map $p:CPA^Y(\mathscr{P},\mathbf{R})/Aff\to \K(\mathcal{Y}_\mathscr{P})$, where $\K(\mathcal{Y}_\mathscr{P})$ is the relative ample cone of $\mathcal{Y}_\mathscr{P}$ over $S$. Since mirror symmetry only concerns the geometry over complex numbers, we can switch to the complex analytic category and use the universal covering map $\Upsilon: \widetilde{\mathcal{Y}}_\mathscr{P}\to \mathcal{Y}_\mathscr{P}$\footnote{In general, since $\mathcal{Y}_\mathscr{P}$ is the algebraization of a formal scheme, we make the argument order by order.}. Since the Cartier divisors on $\widetilde{\mathcal{Y}}_\mathscr{P}$ are described by piecewise affine functions, the pull back of line bundles from $\mathcal{Y}_\mathscr{P}$ to $\widetilde{\mathcal{Y}}_\mathscr{P}$ defines $\Upsilon^*:\K(\mathcal{Y}_\mathscr{P})\to CPA^Y(\mathscr{P},\mathbf{R})/Aff$, such that $ \Upsilon^*\circ p=\Id$. Therefore, the map $p:CPA^Y(\mathscr{P},\mathbf{R})/Aff\to \K(\mathcal{Y}_\mathscr{P})$ is an injection. 

\begin{corollary}\label{the restriction map for NS}
Regard $CPA^Y(\mathscr{P},\mathbf{R})/Aff$ as a subset of $\K(\mathcal{Y}_\mathscr{P})$. The restriction map $r: CPA^Y(\mathscr{P},\mathbf{R})/Aff\to \NS(Y_\eta)$ is given by, 
\begin{equation}
\varphi \longmapsto \omega=Q=\check{\phi}\circ\phi^{-1}.
\end{equation}
\end{corollary}

\begin{proof}
It follows from the computation in Corollary~\ref{the action on the homogeneous coordinate ring}. 
\end{proof}

Assume that the set of the irreducible components of $Y_s$ is indexed by $I=B(\mathbf{Z})$ and $\vert I\vert=\prod_i \delta_i=d$. Each irreducible component is a prime divisor $D_i, i\in I$. 

\begin{lemma}\label{the picard group is the space of function}
Let $\mathscr{P}$ be minimal. We have the following exact sequence,
\begin{equation}\label{the exact sequence for the picard group}
\begin{diagram}
0& \rTo &\mathbf{Z}^d/\mathbf{Z}&\rTo &\pic(\mathcal{Y}_\mathscr{P})&\rTo^r &\NS(Y_\eta)=\NS(Y_t) &\rTo &0.
\end{diagram}
\end{equation}

The mod $\mathbf{Z}$ is because $\sum_{i\in I} D_i=\check{\pi}^{-1}(0)\equiv 0$.\\
 
In particular, $p:CPA^Y(\mathscr{P},\mathbf{R})/Aff\to \K(\mathcal{Y}_\mathscr{P})$ is a bijection, and $\pic(\mathcal{Y}_\mathscr{P})$ is identified with the space $\Gamma(B,\mathcal{PA}/\mathcal{A}ff)$. 
\end{lemma}

\begin{proof}
Since the generic fiber $Y_\eta$ is $\mathcal{E}_{1,\eta}\times\ldots\times \mathcal{E}_{g,\eta}$ for $\mathcal{E}_{i,\eta}$ with no complex multiplication, a similar computation as that in (\cite{kawa97} Proposition 2.10) shows that $\NS(Y_\eta)\cong\mathbb{L}^*=\NS(Y_t)$. From now on, we identify $\NS(Y_\eta)$ with $\NS(Y_t)$, $\mathcal{K}(Y_\eta)$ with $\mathcal{K}(Y_t)$. 

For the exactness of the sequence \eqref{the exact sequence for the picard group}, the only thing left to check is that $\sum_{i\in I}D_i=0$ is the only relation in $\pic(\mathcal{Y}_\mathscr{P})$ for $\{D_i\}$. If $\mathcal{Y}_\mathscr{P}$ is a surface, this follows from (\cite{Sha} Theorem 4.14). For the general case, assume we have a relation $D_r$. Intersect $D_r$ with a generic hypersurface $\mathcal{S}'$ that is flat over $\Delta$. Then, by induction, $D_r\cdot\mathcal{S}'$ is a multiple of $\sum_{i\in I}D_i\cdot\mathcal{S}'$. It follows that $D_r$ is a multiple of $\sum_{i\in I} D_i$. 

The map $r$ restricted to $CPA^Y(\mathscr{P},\mathbf{R})/Aff$ is already onto $\mathcal{K}(Y_t)=\mathcal{K}(Y_\eta)$ by Corollary~\ref{the restriction map for NS}. Moreover $\mathbf{Z}^d/\mathbf{Z}$ is contained in the group $PA^Y(\mathscr{P},\mathbf{R})/Aff$. Since the sequence\eqref{the exact sequence for the picard group} is exact, $p$ is also surjective, and $\pic(\mathcal{Y}_\mathscr{P})\cong\Gamma(B,\mathcal{PA}/\mathcal{A}ff)$. 
\end{proof} 

The relative minimal model $\check{\pi}: \mathcal{Y}\to \Delta$ is not unique, but any two models are isomorphic up to codimension $1$, and are connected by a sequence of flops. See \cite{kaw}. Therefore the pseudo-effective cone $\overline{\effdivisor}(\mathcal{Y})$ in the N\'eron-Severi group is independent of the choice of the relative minimal model $\mathcal{Y}$. Recall (\cite{HK} Definition 1.3) that the types of Stein factorizations induce a fan structure, called the Mori fan, on $\overline{\effdivisor}(\mathcal{Y})$. The Mori fan is also independent of the choice of the relative minimal model $\mathcal{Y}$. The support is the rational closure of the big cone. 

\begin{lemma}\label{the pseudo-effective cone}
The pseudo-effective cone $\overline{\effdivisor}(\mathcal{Y}_\mathscr{P})$ is the closure of $r^{-1}(\K(Y_t))$. 
\end{lemma}

\begin{proof}
The restriction $r(D)$ of an effective divisor $D$ is effective. Since $Y_\eta$ is an abelian variety, $r(D)$ is in the closure of $\K(Y_\eta)$. Therefore, $D$ is in the closure of $r^{-1}(\K(Y_\eta))$. 

On the other hand, $\ker(r)$ are all effective, by adding $\sum_{i\in I}D_i$. Moreover, for any rational class $Q\in \K(Y_\eta)$, it is easy to construct an effective divisor $\varphi_Q$ such that $r(\varphi_Q)$ is $Q$. For example, make $Q$ an integral divisor $D_\eta$ by multiplying by a positive integer. Then the closure of $D_\eta$ is an element in $\overline{\effdivisor}(\mathcal{Y}_\mathscr{P})$. 
\end{proof}

By Lemma~\ref{quasiperiodic and quadratic}, a real valued function $\psi$ on $X$ is called $Y$-quasiperiodic for $\phi(Y)\subset X$ if it is a sum of a quadratic function and a $Y$-periodic function. The set of $Y$-quasiperiodic functions over $X$ is a vector space of finite dimension. Fix a $Y$-invariant integral triangulation $\mathscr{T}$ of $X_\mathbf{R}$. For any function $\psi$ over $X$, we can define a piecewise affine function $g_{\psi,\mathscr{T}}$. For each vertex $\alpha$ of $\mathscr{T}$, define $g_{\psi,\mathscr{T}}(\alpha)=\psi(\alpha)$. Then $g_{\psi,\mathscr{T}}$ is obtained by affine interpolation over each simplex of $\mathscr{T}$. 

\begin{definition}\label{definition of the cone of triangulations for av}
Let $\mathscr{T}$ be a $Y$-invariant integral triangulation. We shall denote by $\widetilde{C}^Y(\mathscr{T})$ the cone consisting of $Y$-quasiperiodic functions $\psi$ over $X$ with the following two properties: 
\begin{itemize}
\item[a)] The function $g_{\psi, \mathscr{T}}$ is convex.
\item[b)] For any $\alpha\in X$ but not a vertex of any simplex from $\mathscr{T}$, we have $g_{\psi, \mathscr{T}}(\alpha)\leqslant \psi(\alpha)$. 
\end{itemize}
\end{definition}

Define $C^Y(\mathscr{T})$ to be $\widetilde{C}^Y(\mathscr{T})/Aff$.  If $\mathscr{P}$ is minimal, $C^Y(\mathscr{P})=CPA^Y(\mathscr{P},\mathbf{R})/Aff$.  

\begin{proposition}\label{the mori fan of the mirror of av}
Every Mori chamber is of the form $C^Y(\mathscr{T})$ for some triangulation $\mathscr{T}$. Every relative minimal model is isomorphic to $\mathcal{Y}_{\mathscr{P}'}$ for some minimal triangulation $\mathscr{P}'$. 
\end{proposition}

\begin{proof}
Fix $\mathscr{P}$ a minimal triangulation. For each $Y$-invariant triangulation $\mathscr{T}$, the rational map $\mathcal{Y}_\mathscr{P}\dashrightarrow \mathcal{Y}_{\mathscr{T}}$ is a contraction. It is a small contraction if and only if $\mathscr{T}$ is also minimal. We claim that each $C^Y(\mathscr{T})^\circ$ is contained in one Mori chamber. 

For a $\mathbf{Q}$-Cartier $D\in C^Y(\mathscr{T})^\circ$ corresponding to a function $\psi_D$ over $X$, decompose $\psi_D=\psi_A+\psi_E$, where $\psi_A=g_{\psi,\mathscr{T}}$ and $\psi_E=\psi_D-g_{\psi,\mathscr{T}}$. Since $g_{\psi,\mathscr{T}}$ is strictly convex, $A$ is ample on $\mathcal{Y}_{\mathscr{T}}$. Then $D$ defines the rational map $f_D:\mathcal{Y}_\mathscr{P}\dashrightarrow \mathcal{Y}_\mathscr{T}$ and $D=f_D^*A+E$ for $E$ $f_D$-exceptional. It proves the claim. 

By Lemma~\ref{the picard group is the space of function} and Lemma~\ref{the pseudo-effective cone}, we can identify $\overline{\effdivisor}(\mathcal{Y}_\mathscr{P})$ with the closed cone $\mathcal{C}^+$ in $\Gamma(B,\mathcal{PA}/\mathcal{A}ff)$, where the associated quadratic form $Q$ is semi-positive definite. Denote the interior of $\mathcal{C}^+$ by $\mathcal{C}^\circ$, which corresponds to the big cone. The support of the Mori fan is the rational closure of $\mathcal{C}^\circ$, denoted by $\mathcal{C}^{\text{rc}}$. However, by the same argument as in (\cite{GKZ} Chap.7, Proposition 1.5), the cones $C^Y(\mathscr{T})$ already form a fan supported on $\mathcal{C}^{\text{rc}}$. Therefore, each cone in the Mori fan is of the form $C^Y(\mathscr{T})$, and the contraction is small if and only if $\mathscr{T}$ is minimal. 
\end{proof}

In the case of principally polarized abelian varieties, we have $\pic(\mathcal{Y})\cong \NS(Y_t)$ for any minimal model $\mathcal{Y}$. Furthermore, $Y=X$, each maximal dimensional cone $C^Y(\mathscr{T})$ is the cone of quadratic forms whose paving is coarser than $\mathscr{T}$. By (\cite{AN} Lemma 1.8), the paving $\mathscr{T}$ is the Delaunay decomposition of the corresponding quadratic form with respect to the lattice $Y$. The collection $\{C^Y(\mathscr{T})\}$ with their faces is the second Voronoi fan of $\mathcal{C}(X)^\rc$ with respect to the lattice $Y=X$. For the convenience of the reader, we recall the definition of the second Voronoi fan (\cite{Nam80} Definition 9.8, Theorem 9.9).
\begin{definition}[the second Voronoi fan]\label{the second Voronoi fan}
Let $\mathcal{C}(X)^\rc$ be the rational closure of the cone of positive definite quadratic forms $\mathcal{C}(X)$. For any paving $\mathscr{P}$ of $X_\mathbf{R}$, set
\[
\sigma(\mathscr{P}):=\{Q\in \mathcal{C}(X)^\rc; \text{the Delaunay decomposition of $Q$ with respect to $X$ is $\mathscr{P}$}\}. 
\]

The collection of cones $\{\sigma(\mathscr{P})\}$ is a fan, and is called \emph{the second Voronoi fan} with respect to $X$. 
\end{definition}

We have proved

\begin{theorem}\label{Two fans are the same}
Let $\mathcal{Y}^{\ppav,*}/\Delta^*$ be the mirror family for the principally polarized abelian varieties. For any relative minimal model $\mathcal{Y}/\Delta$, we have $\overline{\effdivisor}(\mathcal{Y})= \overline{\mathcal{C}}(X)$. Moreover, the Mori fan of $\mathcal{Y}$ agrees with the second Voronoi fan. 
\end{theorem}

In general, $\overline{\effdivisor}(\mathcal{Y})$ is much bigger than $\overline{\mathcal{C}}(X)$. We need a section to get a canonical fan on $\mathcal{C}(X)^{\rc}$. 

\begin{proposition}
For any irreducible component $D_\alpha$ (for $\alpha\in I$) of $Y_s$, the complement $Y_s\backslash D_\alpha$ is contractible in $\mathcal{Y}$, and the contraction is denoted by $p_\alpha:\mathcal{Y}\dashrightarrow \mathcal{Y}_\alpha$. We call $\mathcal{Y}_\alpha$ a cusp model. It is not unique, but they are all isomorphic up to codimension $1$. They are $\mathbf{Q}$-factorial, normal, and Gorenstein.
\end{proposition}

\begin{proof}
We construct the cusp model $\mathcal{Y}_\alpha$ directly by Mumford's construction. Fix $\alpha$, consider the lattice $\alpha+\check{Y}$ in $U^*$. Construct a fan from the cones over a Delaunay decomposition $\mathscr{P}_D$ with respect to $\alpha+\check{Y}$. By taking a $Y$-quasiperiodic, convex, piecewise affine function $\varphi$, we can get a relative complete model as before. By Mumford's construction, we get one of the models $\mathcal{Y}_\alpha$. 
\end{proof} 

By a similar argument as Proposition~\ref{the mori fan of the mirror of av}, the Mori fan of $\mathcal{Y}_\alpha$ is identified with the second Voronoi fan with respect to the lattice $\alpha+\check{Y}$. Furthermore, since the central fiber is irreducible, the restriction map  $\pic(\mathcal{Y}_\alpha)\to \NS(Y_t)$ is an isomorphism. The Mori fans are identified in $\NS(Y_t)$ because the second Voronoi fan doesn't depend on the translation of the lattice on $U^*$.

Now fix a cusp $\alpha$. Consider the common Mori fan in $\pic(\mathcal{Y}_\alpha)$. Each cone is an ample cone $\K(\mathcal{Y}'_\alpha)$ for some cusp model $\mathcal{Y}'_\alpha$. Define a section on this cone 
\begin{align*}
\sigma_\alpha: \K(\mathcal{Y}'_\alpha) &\longrightarrow \pic(\mathcal{Y})\\
\psi&\longmapsto (p_\alpha)^*(\psi).\\
\end{align*}

That means, on each ample cone, we just pull back the Cartier divisor. $\sigma_\alpha$ is not linear, but piecewise linear and convex. Regard the function $\sigma_\alpha$ as a piecewise linear section $\NS(Y_t)\to \pic(\mathcal{Y})$. Take the average on $\NS(Y_t)$, and define
\begin{equation}\label{the linear section is the average}
\sigma=\frac{1}{d}\sum_\alpha \sigma_\alpha.
\end{equation}

\begin{proposition}\label{the section is linear}
The section $\sigma$ is linear. 
\end{proposition}

\begin{proof}
Consider the bending parameters of $\sigma$. Since all different linear pieces of $\sigma$ are sections of $r$, all bending parameters live in the kernel $\ker(r)=\mathbf{Z}^d/\mathbf{Z}$. On the other hand, $X$ is acting on $X$ by translation. It induces an action of $X/\check{Y}$ on $\pic(\mathcal{Y})$ since elements of $\pic(\mathcal{Y})$ are functions on $X$. The induced action of $X/\check{Y}$ on $\NS(Y_\eta)$ is trivial. But $\beta\in X/\check{Y}$ maps the section $\sigma_\alpha$ to $\sigma_{\alpha-\beta}$. Therefore the section $\sigma$ is invariant under $X/\check{Y}$, and the bending parameters are also $X/\check{Y}$-invariant. This implies that all the coefficients are the same, and the bending parameters are $0$ in $\mathbf{Z}^d/\mathbf{Z}$. 
\end{proof}

Since $\sigma$ is a linear section, it is easy to compute that $\sigma: \NS(Y_t)\to \pic(\mathcal{Y})$ is 
\[
\sigma: Q\mapsto \varphi=\text{affine interpolation of } 1/2Q\vert _{X}.
\]

The image of $\sigma$ is characterized by being invariant under the action of $X/\check{Y}$. We denote it by $\pic^X(\mathcal{Y})$.
\begin{remark}
This choice of degeneration data for higher polarizations has been considered in \cite{ABH}, \cite{ols08}, and \cite{Nak10}. 
\end{remark}

Identify $\pic^X(\mathcal{Y})$ with the subspace of $X$-quasiperiodic functions in $\pic(\mathcal{Y})$, i.e., the sections are the pull-backs from $U^*/X$. Consider the Mori fan of $\pic(\mathcal{Y})$, and pull it back through $\sigma$. Each cone is of the form $\sigma^{-1}(C^Y(\mathscr{P}))=\pic^X(\mathcal{Y})\cap C^Y(\mathscr{P})$, and is denoted by $C(\mathscr{P})$. Since every element in $C(\mathscr{P})$ is $X$-quasiperiodic, it can be identified with the associated quadratic form. Therefore, each $C(\mathscr{P})$ is the cone of quadratic forms whose paving is coarser than $\mathscr{P}$, and the pull back fan is the second Voronoi fan with respect to $X$ for $\mathcal{C}(X)$. This is the fan we are going to use for the toroidal compactification, and it is denoted by $\Sigma(X)$. 

\subsubsection{Non-splitting boundary components}
Now we consider the general case. Assume $\phi: Y\to X$ is of type $\mathfrak{d}$, the mirror family is $Y_t=E_1\times E_2\times \ldots \times E_g$, where $E_i=\mathbf{C}/(d_i^{-1}\mathbf{Z}+t\mathbf{Z})$. There is no Lagrangian section for the fibration $X\to B$ but local Lagrangian sections. Therefore, if we use the Fourier--Mukai transform to define the mirror map, the images are twisted sheaves twisted by a gerbe. However, on the level of Chern classes, the only difference here is that $\overline{P}(X)$ is not the whole automorphism group of $\NS(\mathcal{Y}_\eta/\Delta^*)$. The rest goes the same as in the splitting case.

In sum, if $F$ is a $0$-cusp, we get a canonical fan supported on $\mathcal{C}(X)^\rc$, which is the second Voronoi fan with respect to the lattice $X$. If $F_\xi$ is an arbitrary cusp, and $F_\xi\succ F$ for $F$ a $0$-cusp, we need $\Sigma(F_\xi)=\Sigma(F)\cap \mathcal{P}'(F)$ because of c) in (\cite{HKW} Definition 3.66). Therefore, $\Sigma(F)_\xi$ is the second Voronoi fan for $X_{\xi,\mathbf{R}}$ with respect to $X_\xi$\footnote{For the proof, see part c) of the proof of Theorem~\ref{admissible collection of fans} below.}. 
\begin{definition}\label{the canonical fan from mirror symmetry}
For each cusp $F$, we define the fan $\Sigma(F)$ (or $\Sigma(X)$) to be second Voronoi fan with respect to $X$ (Definition~\ref{the second Voronoi fan}), supported on $\mathcal{C}(X)^\rc$. Define $\widetilde{\Sigma}:=\{\Sigma(F)\}$. 
\end{definition}

%%%%%%%%%%%%%%%%%%%%%%%%%%%%%%%%%%%%%%%%%%%%%%%%%%%%%%%%%%%%%%%%%%%%%%%%%%%%%%%%%%%%%%%%%%%%%%%%%
%%%%%%%%%%%%%%%%%%%%%%%%%%%%%%%%%%%%%% Construction of the stack %%%%%%%%%%%%%%%%%%%%%%%%%%%%%%%%%%%%%%%%%%%%%%
%%%%%%%%%%%%%%%%%%%%%%%%%%%%%%%%%%%%%%%%%%%%%%%%%%%%%%%%%%%%%%%%%%%%%%%%%%%%%%%%%%%%%%%%%%%%%%%%%

\subsection{Toroidal Compactifications: The Construction}\label{3.1}
\begin{proposition}\label{admissible collection of fans}
The collection of fans $\widetilde{\Sigma}=\{\Sigma(F)\}$ in Definition~\ref{the canonical fan from mirror symmetry} is an admissible collection.
\end{proposition}

\begin{proof}
It suffices to check the conditions in (\cite{HKW} Definition 3.66). For a), we check the conditions in (\cite{HKW} Definition 3.61), and prove that $\Sigma(F)$ is an admissible fan for every $F$. By (\cite{Nam80} Theorem 9.9), $\Sigma(F)$ is an admissible fan with respect to $\GL(X)$. Since $\overline{P}(F)\subset\GL(X,Y)\subset \GL(X)$ is of finite index (Proposition~\ref{interpretation of the cone}), $\Sigma(F)$ is also admissible with respect to $\overline{P}(F)$.

For b). If $M\in\Gamma(\delta)$, $M:\Gamma^2U\to \Gamma^2U'$ is induced from the isomorphism $M^{-1}\vert_{U}:U\to U'$. And $M(\Lambda\cap U)=\Lambda\cap U'$. So the dual map $(M^{-1}\vert_{U})^*:X'_\mathbf{R}\to X_\mathbf{R}$ maps $X$ onto $X'$. Therefore $M:\mathcal{C}(X)\to \mathcal{C}(X')$ maps the second Voronoi fan $\Sigma(X)$ to $\Sigma(X')$. 

For c). If $F'\succ F$, we have the quotient $q_\mathbf{R}:X_\mathbf{R}\to X'_\mathbf{R}$ induced from a quotient map of the lattices $q: X\to X'$. The pull back $\Gamma^2U'\to \Gamma^2U$ identifies $\mathcal{C}(X')$ with the positive semi-definite quadratic forms with nullspaces $\ker (q_\mathbf{R})$. Let $\mathscr{P}_{D'}$ be a Delaunay decomposition of $X'_\mathbf{R}$ and $\sigma'\in \Sigma(F')$ be the cone associated with $\mathscr{P}_{D'}$. The pull back $\mathscr{P}=q_\mathbf{R}^{-1}(\mathscr{P}_{D'})$ is a Delaunay decomposition of $X_\mathbf{R}$. ( The cells are infinite in the direction of $\ker(q_{\mathbf{R}})$. ) The cone $\sigma$ associated with $\mathscr{P}$ is a cone in $\Sigma(F)$, and it contains $\sigma'$. On the other hand, the supports of $\Sigma(F')$ and $\Sigma(F)\cap\Gamma^2U'$ are the same. Therefore $\Sigma(F')=\Sigma(F)\cap\Gamma^2U'$.  
\end{proof}

\begin{theorem}\label{the toroidal compactification over C}
Over $\mathbf{C}$, we have a toroidal compactification $\overline{\mathcal{A}}_\Sigma$ of $\mathcal{A}_{g,\delta}$. Furthermore, $\overline{\mathcal{A}}_\Sigma$ is projective.
\end{theorem}
\begin{proof}
The construction of $\overline{\mathcal{A}}_\Sigma$ for the admissible collection of fans $\widetilde{\Sigma}$ is (\cite{HKW} Theorem 3.82), or (\cite{AMRT} Theorem 5.2). To prove that it is projective, we apply Tai's criterion (\cite{FC} Definition 2.4) or (\cite{AMRT} Chap. \Rmnum{4} Definition 2.1, Corollary 2.3). For each cone $\mathcal{C}(F)$, the fan $\Sigma(F)$ is the second Voronoi fan with respect to $X$. We can use the polarization function provided by $\Sigma(X_\mathbf{R}/X)$ in \cite{Alex02}. 
\end{proof}

Following \cite{FC}, we are going to construct an arithmetic toroidal compactification from the admissible collection of fans $\widetilde{\Sigma}$. In order to avoid the reduction of bad primes, we work over $k=\mathbf{Z}[1/d]$. For some technical reasons, we have to work over $k=\mathbf{Z}[1/d,\zeta_{M}]$. Over $k$, the polarizations are separable and the stack $\mathscr{A}_{g,\delta}$ is a connected component of $\mathscr{A}_{g,d}$.

%Although the algebraic stack $\mathscr{A}_{g,\delta}^{FC}$ constructed in \cite{FC} admits a family over it\footnote{They only work out the case of principal polarizations.}, the family is not proper, but a semiabelian group scheme. Basically, we replace the families of semiabelian schemes in \cite{FC} by the AN families defined below. 

%%%%%%%%%%%%%%%%%%%%%%%%%%%%%%%%%%%%%%%%%%%%%%%%%%%%%%%%%%%%%%%%%%%%%%%%%%%%%%%%%%%
%%%%%%%%%%%%%%%%%%%%%%%%%%%%%%%%%%%%%%%%%%%%%%%%%%%%%%%%%%%%%%%%%%%%%%%%%%%%%%%%%%%
%%%%%%%%%%%%%%%%%%%%%%%%%%%%%%%%%%%%%%%%%%%%%%%%%%%%%%%%%%%%%%%%%%%%%%%%%%%%%%%%%%%

\section{Extending the Universal Family}\label{3}

\subsection{AN Families} \label{3.2}
\subsubsection{The AN Construction}
For this section, $k=\mathbf{Z}$. The purpose of this section is to recall the modified Mumford's construction in \cite{AN} for an arbitrary complete normal base. Fix a base scheme $S=\spec R$, for $R$ a Noetherian, excellent, and normal integral domain, complete with respect to an ideal $I=\sqrt{I}$. Denote the residue ring $R/I$ by $\kappa$, and the field of fractions by $K$.  The closed subscheme is denoted by $S_0=\spec \kappa$, and the generic point by $\eta=\spec K$. Recall the categories $\DEG$ and $\DD$ from \cite{FC}. The objects of the category $\DEG_{\text{ample}}$ are pairs $(G, \mathcal{L})$, where $G$ is semiabelian over $S$, with $G_\eta$ abelian over $K$, and $\mathcal{L}$ is an invertible sheaf on $G$, with $\mathcal{L}_\eta$ ample. The morphisms are group morphisms that respect the invertible sheaves. An object of $\DD_{\text{ample}}$ is the following degeneration data:

\begin{enumerate}
\item An abelian scheme $A/S$ of relative dimension $g'$, a split torus $T/S$ defined by the character group $X\cong \mathbf{Z}^r$, with $r=g-g'$, and a semiabelian group scheme $\widetilde{G}$ defined by $c:X\to A^t$. We use the same letter to denote the group $X$ and the corresponding constant sheaf $X$ on $S$. 
\[
\begin{diagram}
1&\rTo &T&\rTo&\widetilde{G}&\rTo^\pi&A&\rTo&0.
\end{diagram}
\]
\item A rank $r$ free abelian group $Y\cong\mathbf{Z}^r$ and the constant sheaf $Y$ over $S$.  
\item A homomorphism between group schemes $c^t: Y\to A$. This is equivalent to an extension
\[
\begin{diagram}
1&\rTo &T^t&\rTo&\widetilde{G}^t&\rTo^{\pi^t}&A^t&\rTo&0.
\end{diagram}
\]
\item  An injection $\phi: Y\to X$ of type $\mathfrak{d}_1$. 
\item A homomorphism $\iota: Y_\eta\to \widetilde{G}_\eta$ over $S_\eta$ lying over $c_\eta^t$. This is equivalent to a trivialization of the biextension $\tau: 1_{Y\times X} \to (c^t\times c)^*\mathcal{P}^{-1}_{A,\eta}$. Here $\mathcal{P}_A$ is the Poincar\'{e} sheaf on $A\times A^t$ which has a canonical biextension structure. We require that the induced trivialization $\tau\circ(\Id\times \phi)$ of $(c^t\times c\circ\phi)^*\mathcal{P}^{-1}_{A,\eta}$ over $Y\times Y$ is symmetric. The trivialization $\tau$ is required to satisfy the following positivity condition: $\tau(\lambda,\phi(\lambda))$ for all $\lambda$ extends to a section of $\mathcal{P}_A^{-1}$ on $A\times_SA^t$, and is $0$ modulo $I$ if $\lambda\neq 0$. 
\item An ample sheaf $\mathcal{M}$ on $A$ inducing a polarization $\lambda_A: A\to A^t$ of type $\delta'$ such that $\lambda_A c^t=c\phi$. This is equivalent to giving a $T$-linearized sheaf $\widetilde{\mathcal{L}}=\pi^*\mathcal{M}$ on $\widetilde{G}$.
\item An action of $Y$ on $\widetilde{\mathcal{L}}_\eta$ compatible with $\phi$. This is equivalent to a cubical trivialization $\psi: 1_Y\to (c^t)^*\mathcal{M}_\eta^{-1}$, which is compatible with the trivialization $\tau\circ(\Id\times \phi)$. 
\end{enumerate}

%Here $\psi$ being compatible with $\tau$ means as follows. First, for an integer $n$ and a subset $I\subset\{1,\ldots,n\}$, let $m_I: A^n\to A$ be the morphism that adds the $i$-th coordinates for all $i\in I$. If $I$ empty, $m_I$ is the constant morphism to the identity section of $A$. For any sheaf $\mathcal{M}$ over an abelian scheme $A$, define a sheaf $\Lambda(\mathcal{M})$ over the product $A\times A$ by
%\[
%\Lambda(\mathcal{M}):=\bigoplus_{I\subset \{1,2\}} m_I^*\mathcal{M}^{(-1)^{\#I}}. 
%\]

%By the theorem of cube, if $A$ is an abelian scheme, and $\mathcal{M}$ defines a polarization $\lambda_A$, then $\Lambda(\mathcal{M})$ is a symmetric biextension and agrees with $(\Id\times\lambda_A)^*\mathcal{P}_A$. Therefore $(c^t\times c^t)^*\Lambda(\mathcal{M})^{-1}=(c^t\times c\circ\phi)^*\mathcal{P}_A^{-1}$. We say $\psi$ and $\tau$ are compatible, if $\Lambda(\psi)=\tau\circ(\Id\times \phi)$ as trivializations of the symmetric biextension over $Y\times Y$. \\

Denote the rigidified line bundle $c(\alpha)$ by $\mathcal{O}_\alpha$, and the rigidified line bundle $\mathcal{M}\otimes \mathcal{O}_\alpha$ by $\mathcal{M}_\alpha$. For any $S$-point $a: S\to A$, and any line bundle $\mathcal{L}$ on A, the pull back $a^*\mathcal{L}$ is denoted by $\mathcal{L}(a)$. For any $\lambda\in Y$, $\alpha\in X$, $\tau(\lambda,\alpha)\in \mathcal{O}_\alpha(c^t(\lambda))^{-1}_\eta=\mathcal{P}_A^{-1}(c^t(\lambda),c(\alpha))_\eta$, and $\psi(\lambda)\in \mathcal{M}(c^t(\lambda))^{-1}_\eta$ are $K$-sections. Therefore the data $\tau$ and $\psi$ gives trivializations $\psi(\lambda)^d\tau(\lambda,\alpha): (\mathcal{M}^d\otimes \mathcal{O}_\alpha)^{-1}_\eta(c^t(\lambda))\cong \mathcal{O}_K$ for every $\lambda\in Y$, $\alpha\in X$. By (\cite{ols08} Lemma 5.2.2.), the data $\psi(\lambda)^d\tau(\lambda,\alpha)$ is equivalent to an isomorphism of line bundles 
\[
\psi(\lambda)^d\tau(\lambda,\alpha): T^*_{c^t(\lambda)}(\mathcal{M}^d\otimes \mathcal{O}_\alpha)_\eta\to \mathcal{M}^d\otimes \mathcal{O}_{\alpha+d\phi(\lambda),\eta}. 
\]

By the property of biextensions, (\cite{ols08} Proposition 2.2.13),  it defines an action of the group $Y$ on the graded algebra
\[
\mathcal{S}=\Bigg(\prod_{d\geqslant 0}\big(\bigoplus_{\alpha\in X}\mathcal{O}_\alpha\big)\otimes \mathcal{M}^d\theta^d\Bigg)\otimes_{R}K,
\]
denoted by $S_\lambda^*$. 

Assume $P$ is a toric monoid, $P=\sigma_P^\vee\cap P^{gp}$, and $\alpha: P\to R$ is a prelog structure which maps the toric maximal ideal $P\backslash\{0\}$ to $I$. Fix data: an integral, $Y$-quasiperiodic, $P$-convex piecewise affine function $\varphi: X_\mathbf{R}\to P^{\gp}_\mathbf{R}$ determined by a collection of bending parameters $\{p_\rho\in P\backslash\{0\}\}$. For simplicity, the composition $\alpha\circ \varphi: X_\mathbf{R}\to R$ is also denoted by $\varphi$. Suppose the associated quadratic form $Q$ is positive definite, i.e. $Q(x)\in \sigma_P^\vee\backslash\{0\}$ for $x\neq 0$, and the associated paving $\mathscr{P}$ is bounded. Construct the following functions $a_t: Y\to \mathbb{G}_m(K)$ and $b_t:Y\times X\to \mathbb{G}_m(K)$ by 
\begin{align}
a_t(\lambda)&=\mathrm{X}^{\varphi(\phi(\lambda))}, \quad \forall \lambda\in Y, \\
\mathrm{X}^{\varphi(\alpha+\phi(\lambda))}&=a_t(\lambda)b_t(\lambda,\alpha)\mathrm{X}^{\varphi(\alpha)}, \quad \forall \alpha\in X.
\end{align}

By the properties of $\varphi$, $a_t$ is quadratic, $b_t$ is bilinear and $b_t(\lambda, \phi(\mu))$ is symmetric on $Y\times Y$. By (\cite{ols08} 2.2.8), for any biextension over $Y\times X$, the automorphism is classified by $\Hom(Y\otimes X, \mathbb{G}_m)$. If we require that the extension is symmetric on $Y\times \phi(Y)$, then the bilinear form $b\in \Hom (Y\otimes X,\mathbb{G}_m)$ should satisfy that $b(\lambda,\phi(\mu))$ be symmetric. Furthermore, the automorphism of a central extension over $Y$ is classified by quadratic functions over $Y$. Fix a biextension (resp, central extension), regard the set of trivializations as a torsor over $\Hom(Y\times X, \mathbb{G}_m)$ (resp, quadratic forms). For any trivialization $\tau$ of a biextension, (resp, any trivialization $\psi$ of a central extension),  denote the trivialization obtained by the action of $b_t^{-1}$ (resp, $a_t^{-1}$) by $b_t^{-1}\tau$ (resp, $a_t^{-1}\psi$). 

\begin{definition}
Fix the data $b_t, a_t$ obtained from $\varphi$. We say the combinatorial data $\varphi$ is compatible with the trivializations $\tau$ and $\psi$, if $b_t^{-1}\tau$ can be extended to a trivialization of the biextension $(c^t\times c)^*\mathcal{P}_A^{-1}$ over $S$, and $a_t^{-1}\psi$ can be extended to a trivialization of the central extension $(c^t)^*\mathcal{M}^{-1}$ over $S$. 
\end{definition}
The data $\varphi$ gives the choice of embeddings $\mathcal{M}_\alpha\to \mathcal{M}_{\alpha,\eta}$, therefore we can use constructions in \cite{AN} to construct the family of varieties over $S$. More explicitly, consider the piecewise linear function $\tilde{\varphi}: \mathbb{X}_\mathbf{R}\to P^{\gp}_\mathbf{R}$ and its graph in $\mathbb{X}_\mathbf{R}\times P^{\gp}_\mathbf{R}$. The piecewise linear function $\tilde{\varphi}$ is integral because the bending parameters $p_\rho$ are all in $P$. Denote the convex polyhedron $P$-above the graph by $Q_{\tilde{\varphi}}$. Consider the graded $\mathcal{O}_A$-algebra
\[
\mathcal{R}:=\bigoplus_{(\alpha,d,p)\in Q_{\tilde{\varphi}}}\mathrm{X}^p\otimes\mathcal{O}_{\alpha}\otimes \mathcal{M}^d\theta^d.
\]

Since $\varphi$ and $\tau,\psi$ are compatible, the action $S_\lambda^*,(\lambda\in Y)$ is 
\begin{equation}\label{explicit action}
\psi'(\lambda)^d\tau'(\lambda,\alpha):T^*_{c^t(\lambda)}(\mathrm{X}^p\otimes\mathcal{O}_{\alpha}\otimes\mathcal{M}^d\theta^d)\cong\mathrm{X}^p a_t(\lambda)^db_t(\lambda,\alpha)\otimes \mathcal{O}_{\alpha+d\phi(\lambda)}\otimes\mathcal{M}^d\theta^d,
\end{equation}

for $\psi'(\lambda)^d\tau'(\lambda,\alpha)$ $R$-sections of $\mathcal{M}^d\otimes\mathcal{O}_\alpha(c^t(\lambda))$. Therefore the $Y$-action preserves the subalgebra $\mathcal{R}$. Let $\widetilde{\mathcal{X}}$ denote the scheme $\sheafproj\mathcal{R}$. Construct the infinite toric embedding $X_\varphi$ from the convex polyhedron $Q_{\tilde{\varphi}}$. This is a toric degeneration over the toric variety $\spec k[P]$. Denote the pull back of $X_\varphi$ along $k[P]\to R$ by $\widetilde{P}^r$. As in (\cite{AN}, 3B, 3.22), the total space $\widetilde{\mathcal{X}}$ is isomorphic to contracted product $\widetilde{P}^r\times ^T\widetilde{G}$ over $S$. 

\begin{lemma}\label{convergence}
Every irreducible component of $\widetilde{X}_0$ is reduced, and proper over $S_0$. 
\end{lemma}
\begin{proof}
Use the isomorphism $\widetilde{\mathcal{X}}\cong \widetilde{P}^r\times ^T\widetilde{G}$ over $S$. The total space $\widetilde{\mathcal{X}}$ is covered by $U(\omega)\times^T\widetilde{G}$, where $U(\omega)$ is the affine toric variety corresponding to the vertex $\omega$ of $Q_{\tilde{\varphi}}$. The action $S_\lambda^*$ maps $U(\omega)\times^T\widetilde{G}$ to $U(\omega+\phi(\lambda))\times^T\widetilde{G}$. Reduce to $\kappa$, $\widetilde{X}_0$ is the contracted product of $\widetilde{P}^r_0:=\widetilde{P}^r\times_{k[P]}S_0$ with $\widetilde{G}_0$. Since $\varphi$ is integral, $\widetilde{P}^r_0$ is reduced. Since $P\backslash\{0\}$ is mapped into $I$, $\widetilde{P}^r_0$ is $\varinjlim X_\sigma$. The inductive limit is over $\sigma\in \mathscr{P}$. Each irreducible component of $\widetilde{P}^r_0$ is the toric variety $X_\sigma$ associated to the maximal cell $\sigma\in \mathscr{P}$. As a result, each irreducible component of $\widetilde{X}_0$ is a fiber bundle over $A_0$ with the fiber $X_\sigma$, and is thus reduced and proper over $S_0$. 
\end{proof}

\begin{lemma}
 The data $(\widetilde{\mathcal{X}},\mathcal{O}(1), S_\lambda^*, \widetilde{G})$ is a relative complete model as defined in (\cite{FC}, \Rmnum{3}. Definition 3.1). 
 \end{lemma}
 \begin{proof}
 The proof is similar to (\cite{AN}, Lemma 3.24). It is not necessary to check the complete condition, since we have shown that every irreducible component of $\widetilde{X}_0$ is proper over $S_0$. To construct an embedding $\widetilde{G}\to \widetilde{\mathcal{X}}$, consider the dual fan $\Sigma(\tilde{\varphi})$ of $Q_{\tilde{\varphi}}$ in $\mathbb{X}_\mathbf{R}^*\times (P_\mathbf{R}^{\gp})^*$. It defines a toric degeneration $\widetilde{P}^r$ over the affine toric variety $\spec k[P]$. It follows that we can lift the fan $\sigma_P$ as a subfan in $\Sigma(\tilde{\varphi})$. Since $\sigma_P$ is a fan defined by faces of one rational polyhedral cone, the subfan induces an embedding of the trivial torus bundle $T$ into $\widetilde{P}^r$. By (\cite{AN}, B 3.22), the embedding $T\to \widetilde{P}^r$ induces the embedding $\widetilde{G}\cong T\times ^T\widetilde{G}\to \widetilde{P}^r\times ^T\widetilde{G}\cong \widetilde{\mathcal{X}}$. 
 \end{proof}
 
 As in \cite{AN}, reduce to $I^n$ for various $n\in\mathbf{N}$, and take the quotient by $Y$, then by Grothendieck's existence theorem (\cite{EGA3} EGA \Rmnum{3}$_{1}$ 5.4.5), we get an algebraic family $\pi: (\mathcal{X},\mathcal{L})\to S$. 
 
 %%%%%%%%%%%%%%%%%%%%%%%%%%%%%%%%%%%%%%%%%%%%%%%%%%%%%%%%%%%%%%%%%%%%%%%%%%%%%%%%%%
 %%%%%%%%%%%%%%%%%%%%%%%%%%%%%%%%%%%%%%%%%%%%%%%%%%%%%%%%%%%%%%%%%%%%%%%%%%%%%%%%%%
 
\subsubsection{convergence problem}
Notice that in the construction above, the only place we use the condition that $P\backslash\{0\}$ is mapped to $I$ is in the proof of Lemma \ref{convergence}. In this case, closed subscheme $S_0$ is sent to the $0$-stratum in the affine toric variety $\spec k[P]$. However, we may also consider other closed strata. 

Suppose $P=\sigma_P^\vee\cap P^{\gp}$ is a toric monoid obtained from a rational polyhedral cone $\sigma_P$, and $\varphi: X_\mathbf{R}\to P^{\gp}_\mathbf{R}$ is a function as above. Let $\mathfrak{m}$ be the maximal toric ideal $P\backslash\{0\}$, $J$ be a prime toric ideal of $P$, and $F$ be the face $P\backslash\ J$. Assume $F=\tau^\perp\cap \sigma_P^\vee$ for $\tau$ a face of $\sigma_P$. Consider the monoid $P'=P/F$ and the piecewise affine function $\varphi'$ defined as the composition of $\varphi$ with $P_\mathbf{R}^{\gp}\to (P')_\mathbf{R}^{\gp}$. The associated paving $\mathscr{P}'$ of $\varphi'$ is coarser than the associated paving $\mathscr{P}$ of $\varphi$. Suppose that $\mathscr{P}'$ is still a bounded paving. Let $P_F$ be the localization of $P$ with respect to the face $F$, and $J_F$ be the toric ideal of $P_F$ generated by $J$. 

\begin{definition}\label{data associated with a face}
All the notations as above. We call the data $(P', \varphi', \mathscr{P}')$, the data associated with the face $F$. 
\end{definition}

\begin{lemma}\label{toric prime}
Every irreducible component of $\widetilde{P}^r\times_{k[P]}\times\spec k[P]/J$ is proper over $\spec k[P]/J$.
\end{lemma}
\begin{proof}
The family $\widetilde{P}^r\to \spec k[P]$ is an infinite toric degeneration constructed from the infinite piecewise affine function $\varphi$. The restriction to $\spec k[P]/J$ is still a toric degeneration. Since $\spec k[P_F]/J_F$ is open and dense in $\spec k[P]/J$, and the closure of an irreducible space is irreducible, each irreducible component of $\widetilde{P}^r\times_{k[P]}\times\spec k[P]/J$ is a toric degeneration of an irreducible component of $\widetilde{P}^r\times_{k[P]}\times\spec k[P_F]/J_F$. Therefore, it suffices to prove the statement for $\widetilde{P}^r\times_{k[P]}\times\spec k[P_F]/J_F$. 

Since $P$ is toric monoid from a rational polyhedral cone $\sigma_P^\vee$, we can choose a splitting $P_F=P'\times F^{\gp}$. Consider the function $\tilde{\varphi}': \mathbb{X}_\mathbf{R}\to (P')_\mathbf{R}^{\gp}$ and the convex polyhedron $Q_{\tilde{\varphi}'}$ over the graph. Let $\widetilde{P}'$ be the toric embedding assoicated to the convex polyhedron $Q_{\tilde{\varphi}'}$. The localization $\widetilde{P}^r_F:=\widetilde{P}^r\times_{k[P]}\spec k[P_F]$ is isomorphic to $\widetilde{P}'\times \spec k[F^{\gp}]$ over $\spec k[P']\times\spec k[F^{\gp}]$. Since $J$ corresponds to the maximal ideal $\mathfrak{m}'=P'\backslash \{0\}$, it reduces to the case $\widetilde{P}'\times_{k[P']}\spec k[P']/\mathfrak{m}'$. Again this is $\varinjlim_{\sigma'}X_{\sigma'}$ for $\sigma'\in\mathscr{P}'$, with each irreducible component the proper toric variety $X_{\sigma'}$ for a maximal cell $\sigma'\in \mathscr{P}'$. 
\end{proof}

\begin{definition}
For any prime toric ideal $J$, associate the face $F$, the quotient monoid $P'$, the function $\varphi': X_\mathbf{R}\to (P')^{\gp}_\mathbf{R}$, and the paving $\mathscr{P}'$ as above. A toric prime ideal $J$ is called admissible for $\varphi$, if the paving $\mathscr{P}'$ is bounded.  
\end{definition}

\begin{lemma}\label{admissible ideal}
Given a toric ideal $J$ such that $J=\cap_{i} J_i$ for each $J_i$ an admissible prime toric ideal. Then every irreducible component of $\widetilde{P}^r\times_{k[P]}\times\spec k[P]/J$ is proper over $\spec k[P]/J$.
\end{lemma}
\begin{proof}
The irreducible components of $\spec k[P]/J$ are $\spec k[P]/J_i$. Each irreducible component of $\widetilde{P}^r\times_{k[P]}\times\spec k[P]/J$ is contained in the family over some $\spec k[P]/J_i$ for some $J_i$. It reduces to Lemma~\ref{toric prime}. 
\end{proof}

The set of all admissible prime toric ideals for $\varphi$ is finite. Take the intersection of all these admissible prime ideals, we get a toric ideal denoted by $I_\varphi$. The corresponding ideal in the ring $k[P]$ is also denoted by $I_\varphi$. 

\begin{definition}
Let $F$ be a face of the toric monoid $P$,  and $J=P\backslash F$ the associated prime toric ideal. Let $\mathscr{P}'$ be the paving associated with the face $F$. A prime ideal $\mathfrak{p}\subset R$ is said to be in the interior of the $\mathscr{P}'$-strata, if $J$ is the maximal prime toric ideal such that the corresponding ideal $J\subset \alpha^{-1}(\mathfrak{p})$. 
\end{definition}

\begin{definition}
An ideal $I\subset R$ is called admissible for $\varphi$, if its radical $\sqrt{I}$ contains $\alpha(I_\varphi)$. The set of admissible ideal is denoted by $\ad_\varphi$. 
\end{definition}

\begin{remark}
Both of the definitions depend only on the log structure not the chart $\alpha: P\to R$, because the ideals $\alpha^{-1}(I)$ and $\alpha^{-1}(\mathfrak{p})$ do not depend on the choice of the chart $\alpha$. 
\end{remark}

\begin{lemma}\label{ideal condition}
If $I\subset I'$ and $I\in \ad_\varphi$, then $I'\in \ad_\varphi$. If $\sqrt{I}\in \ad_\varphi$, then $I\in \ad_\varphi$. The collection $\ad_\varphi$ is closed under finite intersection.
\end{lemma}
\begin{proof}
The first and the second statements follow directly from the definition. The third statement follows from the fact $\sqrt{I\cap J}=\sqrt{I}\cap \sqrt{J}$. 
\end{proof}

Fix an ideal $I\subset R$, admissible for $\varphi$. Take the $I$-adic completion $\widehat{R}_I$. 

\begin{lemma}
The complete ring $\widehat{R}_I$ is a Noetherian, excellent, normal, integral domain, provided that the closed subscheme $\spec R/I$ is connected.
\end{lemma}
\begin{proof}
Since $R$ is excellent, the $I$-adic completion $\widehat{R}_I$ is regular over $R$. (\cite{EGA4} EGA \Rmnum{4}$_2$ 7.8.3 (v)).  In particular, since $R$ is normal, $\widehat{R}_I$ is normal. If $\spec R/I$ is connected, $\widehat{R}_I$ is a normal integral domain. Moreover, $\widehat{R}_I$ remains excellent by (\cite{Val} Theorem 9). 
\end{proof}

\begin{proposition}
Assume there is an object of $\DD_{\text{ample}}$ over $R$ that is compatible with $\varphi$, and $I$ is an admissible ideal for $\varphi$, then there exists a projective family $(\mathcal{X},\mathcal{L})$ over $S=\spec \widehat{R}_I$ such that the generic fiber $\mathcal{X}_\eta\cong \widetilde{G}_\eta$ is abelian with an ample polarization.
\end{proposition}
\begin{proof}
Check the construction. The only thing we need to prove is a statement similar to Lemma~\ref{convergence}. This follows from Lemma~\ref{admissible ideal}. 
\end{proof}

\begin{lemma}[\cite{FC} \Rmnum{3}. Proposition 4.11]
The total space $\mathcal{X}$ is irreducible.
\end{lemma}

\begin{remark}
We call this construction AN construction because it is introduced in the paper \cite{AN} when the base is a complete discrete valuation ring.
\end{remark}

%%%%%%%%%%%%%%%%%%%%%%%%%%%%%%%%%%%%%%%%%%%%%%%%%%%%%%%%%%%%%%%%%%%%%%%%%%%
%%%%%%%%%%%%%%%%%%%%%%%%%%%%%%%%%%%%%%%%%%%%%%%%%%%%%%%%%%%%%%%%%%%%%%%%%%%

\subsubsection{Properities}
First we point out the relation between AN construction and the standard construction in \cite{ols08}. Recall $S(X)=\mathbf{N}\oplus X$. Introduce the monoid $S(X)\rtimes P$ on the underlying set $S(X)\times P$ with the addition law
\[
(\alpha,p)+(\beta,q)=(\alpha+\beta, p+q+\tilde{\varphi}_\mathscr{P}(\alpha)+\tilde{\varphi}_\mathscr{P}(\beta)-\tilde{\varphi}_\mathscr{P}(\alpha+\beta)).
\]

\begin{lemma}
\[
S(Q_{\varphi})\cong Q_{\tilde{\varphi}}\cong S(X)\rtimes P.
\] 
\end{lemma}
\begin{proof}
Notice that $P\to S(Q_\varphi)$ is an integral morphism of integral monoids. $S(Q_\varphi)^{\gp}=\mathbf{R}\times\overline{X}\times P^{\gp}\times\cong \mathbb{X}\times P^{\gp}$. Under this isomorphism, the morphism $P^{\gp}\to S(Q_\varphi)^{\gp}$ is the homomorphism
\begin{align*}
P^{\gp}&\longrightarrow \mathbb{X}\times P^{\gp},\\
p&\longmapsto (0,p).
\end{align*}

Then the cokernel of $f^\flat: P\to S(Q_\varphi)$ is equal to the image of $S(Q_\varphi)$ in $S(Q_\varphi)^{\gp}/P^{\gp}=\mathbb{X}$, which is $S(Q)$. Denote the projection $S(Q_\varphi)\to S(Q)$ by $b$. Since $P$ is sharp, by (\cite{ols08} Lemma 3.1.32), for any $q\in S(Q)$, there exists a unique element $\tilde{q}\in b^{-1}(q)$ such that $\tilde{q}\stackrel{P}{\leqslant}x$ for all $x\in b^{-1}(q)$. For any $q\in S(Q)$, define $\vartheta_q:=\mathrm{X}^{\tilde{q}}$. By definition, for any $x\in b^{-1}(q)$, $x-\tilde{q}\in P$. Therefore $R_{\varphi}$ is a free $k[P]$-module generated by $\{\vartheta_q\}$ for all $q\in S(Q)$. We map the lowest elements $(d,\alpha,p)$ in $Q_{\tilde{\varphi}}$ to $(d,\alpha,0)$ in $S(X)\rtimes P$. This induces an isomorphism of monoids. 
\end{proof}

As a result, $\widetilde{P}^r\cong\proj k[S(X)\rtimes P]\otimes_{k[P]}R$.

Reduction to $I^n$ for each $n$, this is the standard construction in (\cite{ols08} 5.2), except that we have a toric monoid $P$ instead of $H_\mathscr{P}$. Under this isomorphism, the data $\tau,\psi$ in \cite{ols08} corresponds to $b_t^{-1}\tau$ and $a_t^{-1}\psi$ in our case. The constructions of group actions and log structures in \cite{ols08} for the standard constructions can be applied to our case directly. In particular, we have the action $\varrho_n$ of $\widetilde{G}_n$ on $\mathcal{X}_n$ over $S_n:=\spec R/I^n$. If we do the algebraization, we get a $G$-action $\varrho$ on $\mathcal{X}$ over $S$, where $G$ is the semiabelian scheme obtained by Mumford's original construction (\cite{Mum72} pp. 297 or \cite{FC} pp. 66).

The choice of $\mathcal{M}$ gives a $T$-linearization of the line bundle $\widetilde{\mathcal{L}}$ on $\widetilde{\mathcal{X}}$. The action of $T$ gives the Fourier decomposition of $\H^0(\widetilde{\mathcal{X}}, \widetilde{\mathcal{L}})$. By the Heisenberg relation of the action of $Y$ and $\widetilde{G}$, the decomposition is preserved by the action of $Y$. Therefore, we have the Fourier decomposition
\begin{equation}\label{Fourier decomposition}
\H^0(\mathcal{X},\mathcal{L})=\bigoplus_{\alpha\in X/\phi(Y)} \H^0(A,\mathcal{M}_\alpha).
\end{equation}

Let $\mathscr{P}^0\subset X/\phi(Y)$ be a set of  representatives of the vertices of the paving $\mathscr{P}$. For every $\alpha\in \mathscr{P}^0$, choose $\vartheta_{A,\alpha}\in \H^0(A,\mathcal{M}_\alpha)$ such that the residue $\vartheta_{A,\alpha,s}\neq 0$ for any $s\in S$. Define a section $\vartheta_n\in \H^0(\mathcal{X}_n,\mathcal{L}_n)$ as a descent from 
\begin{equation}\label{theta sum}
\tilde{\vartheta}=\sum_{\lambda\in Y} \sum_{\alpha\in \mathscr{P}^0} S_\lambda^*(\vartheta_{A,\alpha}). 
\end{equation}

This is a finite sum for every admissible ideal $I^n$. Let $\Theta_n$ denote the divisor defined by $\vartheta_n$. The collection $\{\Theta_n\}$ gives a divisor $\Theta$ on $\mathcal{X}$ such that $(\mathcal{X},\Theta)$ is a stable pair in $\overline{\mathscr{AP}}_{g,d}$. 

\begin{lemma}\label{stable pairs}
For any admissible ideal $I$, and any geometric point $\bar{x}\to \spec R/I$, the geometric fiber $(X_{\bar{x}},\mathcal{L}_{\bar{x}},\varrho_{\bar{x}},\Theta_{\bar{x}})$ is a stable semiabelic pair as defined in \cite{Alex02}. In particular, if the image $x$ is in the $\mathscr{P}'$-strata, for a bounded paving $\mathscr{P}'$ of $X_\mathbf{R}$, then $(X_{\bar{x}},\mathcal{L}_{\bar{x}}, \varrho_{\bar{x}})$ is an element in $M^{\text{fr}}[\Delta_{\mathscr{P}'}, c, c^t, \mathcal{M}](\kappa(\bar{x}))$ defined by $\psi_0^{(\cdot)}(\cdot)\tau_0(\cdot, \cdot)\in Z^1(\Delta_{\mathscr{P}'},\widehat{\underline{\mathbb{X}}})$, and $(X_{\bar{x}},\mathcal{L}_{\bar{x}},\varrho_{\bar{x}},\Theta_{\bar{x}})$ is an element of $MP^{\text{fr}}[\Delta_{\mathscr{P}'}, c, c^t, \mathcal{M}](\kappa(\bar{x}))$ defined by $(\tilde{\vartheta}_0, \psi_0^{(\cdot)}(\cdot)\tau_0(\cdot, \cdot))\in Z^1(\Delta_{\mathscr{P}'}, \widehat{\mathbb{M}}^*)$. Here $\kappa(\bar{x})$ is the residue field for $\bar{x}$, and $\Delta_{\mathscr{P}'}$ is the complex defined by the paving $\mathscr{P}'$ on $X_\mathbf{R}/\check{Y}$. The notations $\psi_0,\tau_0,\tilde{\vartheta}_0$ is defined in the proof. For other notations, see \cite{Alex02}. 
\end{lemma}

\begin{proof}
Let $J$ and $F$ be the prime toric ideal and the face associated with $\mathscr{P}'$. We can localize with respect to $F$ first, and replace $\varphi$ by $\varphi'$, which gives the paving $\mathscr{P}'$. Consider the geometric fiber $\widetilde{X}_{\bar{x}}$. Since $\alpha(J)\subset \mathfrak{m}_x$, if $(m,\alpha)$ and $(n,\beta)$ is not in the same cell in $C(\mathscr{P}')$, $\tilde{\varphi}'(m,\alpha)+\tilde{\varphi}'(n,\beta)-\tilde{\varphi}'(m+n,\alpha+\beta)$ is mapped to zero in $\kappa(\bar{x})$. Therefore, denote $\mathrm{X}^{(m,\alpha),\tilde{\varphi}'(m,\alpha)}\otimes\mathcal{O}_\alpha\otimes\mathcal{M}^m\theta^m$ by $\mathcal{M}_{(m,\alpha)}$, $(\widetilde{X}_{\bar{x}},\widetilde{\mathcal{L}}_{\bar{x}})$ is the gluing of $(P,L)[\delta, c, \mathcal{M}]$ for each $\delta\in \mathscr{P}'$. Here $(P,L)[\delta, c, \mathcal{M}]$ is defined in (\cite{Alex02} Definition 5.2.5). The action of $Y$ is induced by $\psi^{(\cdot)}(\cdot)\tau(\cdot, \cdot)$, i.e. the residue of $\psi',\tau'$ in $\kappa(\bar{x})$ for $\psi',\tau'$ in Equation~\eqref{explicit action}. Let the residues be denoted by $\psi_0, \tau_0$. It is an element in $Z^1(\Delta_{\mathscr{P}'},\widehat{\underline{\mathbb{X}}})$.

The residue of $\tilde{\vartheta}$ in $\kappa(\bar{x})$ is denoted by $\tilde{\vartheta}_0$. By the definition of $\vartheta_\alpha$, the residue $\tilde{\vartheta}_0$ is in $C^0(\Delta_{\mathscr{P}'}, \widehat{\underline{\text{Fun}}}_{\geqslant 0})=\oplus_i \widehat{\text{Fun}}_{\geqslant 0,i}$, and $\Theta_{\bar{x}}$ does not contain any $G_{\bar{x}}$-strata entirely. The condition that $\tilde{\vartheta}$ is $Y$-invariant means that $(\tilde{\vartheta}_0, \psi_0^{(\cdot)}(\cdot)\tau_0(\cdot, \cdot))$ is an element in $Z^1(\Delta_{\mathscr{P}'}, \widehat{\mathbb{M}}^*)$.  
\end{proof}

\begin{lemma}\label{the stable pair is in AP}
For any admissible ideal $I$ and any $n$, the family $(\mathcal{X}_n,\mathcal{L}_n,\varrho_n, \Theta_n)$ over $\spec R/I^n$ is an object in $\overline{\mathscr{AP}}_{g,d}(\spec R/I^n)$. 
\end{lemma}
\begin{proof}
By Lemma~\ref{stable pairs}, it suffices to prove that $(\mathcal{X}_n, \mathcal{L}_n)$ is flat over $S_n:=\spec R/I^n$. First, $\widetilde{\mathcal{X}}$ is flat over $A$, since $\mathcal{R}$ is locally free as $\mathcal{O}_A$-module. Secondly, when we take the quotient by $Y$, we first take a quotient by $Y_n$, such that $Y_n\subset Y$ is of finite index, and the action of $Y_n$ on $\widetilde{\mathcal{X}}_n$ is free. So this quotient preserves the flatness. Finally, we take the quotient by the finite group $Y/Y_n$. The $Y/Y_n$-invariant part is a direct summand. Therefore $(\mathcal{X}_n,\mathcal{L}_n)$ is flat over $S_n$. \end{proof}

By Grothendieck's existence theorem (\cite{EGA3} EGA \Rmnum{3}$_{1}$ 5.4.5), we get a family $(\mathcal{X}, \mathcal{L}, G, \varrho,\Theta)$ in $\overline{\mathscr{AP}}_{g,d}(S)$.

\begin{corollary}\label{locally free}
The coherent sheaves $\pi_*\mathcal{L}^m$ are locally free of rank $dm^g$ for all $m\geqslant 0$. The coherent sheaves $R\pi_*^i\mathcal{L}^m=0$ for all $m>0, i>0$. In particular, we have $\pi_*\mathcal{L}^m\otimes \kappa(x)\cong \H^0(X_x,\mathcal{L}^m_x)$. 
\end{corollary}
\begin{proof}
For any $x\in S$, consider $\bar{x}\to S$ over the image $x$, with $\kappa(\overline{x})$ the algebraic closure of $\kappa(x)$. Consider the fiber $(X_{\bar{x}},\mathcal{L}_{\bar{x}})$. By (\cite{Alex02} Theorem 5.4.1), $H^i(X_{\bar{x}}, \mathcal{L}_{\bar{x}}^m)=0$ for any $i, m\geqslant 1$. Since $\kappa(x)\to \kappa(\overline{x})$ is flat, $H^i(X_x,\mathcal{L}_x^m)=0$ for any $i,m\geqslant 1$ by (\cite{Har77} Chap. \Rmnum{3} Proposition 9.3). Since the base $S$ is still Noetherian (\cite{AM} Theorem 10. 26) and $\mathcal{L}^m$ is flat over $S$, we can apply (\cite{Har77} Chap. \Rmnum{3} Theorem 12.11). We have $R\pi^i_*\mathcal{L}^m=0$ for all $i, m>0$, and $\pi_*\mathcal{L}^m\otimes \kappa(x)\to \H^0(X_x,\mathcal{L}^m_x)$ is surjective. Apply (\cite{Har77} Chap. \Rmnum{3} Theorem 12.11) again, $\pi_*\mathcal{L}^m$ are locally free, and $\pi_*\mathcal{L}^m\otimes \kappa(x)\cong \H^0(X_x,\mathcal{L}^m_x)$. For the generic $\eta$, $X_\eta$ is an abelian variety and $\H^0(X_\eta,\mathcal{L}^m_\eta)=dm^g$, so $\pi_*\mathcal{L}^m$ is of rank $dm^g$. 
\end{proof}

By abusing of notations, the morphism $A\to S$ is also denoted by $\pi$. Then $\pi_*\mathcal{M}$ is also locally free. The Fourier decomposition is a decomposition of locally free sheaves over $S$
\begin{equation}\label{Fourier decomposition equation}
\pi_*\mathcal{L}\cong \bigoplus_{\alpha\in X/\phi(Y)}\pi_*\mathcal{M}_\alpha. 
\end{equation}

For any coherent sheaf $\mathcal{F}$ over $S$, define $\mathcal{F}^*$ to be the subsheaf 
\[
\mathcal{F}^*(U):=\{f\in \mathcal{F}(U): f\vert_{\kappa(s)}\neq 0,\forall s\in S\} \quad \forall U\subset S.
\]

If we choose a section $\vartheta \in \oplus_{\alpha\in \mathscr{P}^0} \pi_*\mathcal{M}_\alpha^*$, then $(\mathcal{X},\mathcal{L},G,\varrho, \Theta)$ is an object in $\overline{\mathscr{AP}}_{g,d}(\widehat{R}_I)$.

The AN construction is functorial.  Assume $S$, $S'$ are as above, with degeneration data, charts $P\to R$ and $P'\to R'$, and compatible functions $\varphi: X_\mathbf{R}\to P^\gp_\mathbf{R}$, $\varphi':X_\mathbf{R}\to (P'_\mathbf{R})^\gp$, with charts $\alpha: P\to R$ and $\alpha':P'\to R'$. Suppose $(f,f^\flat)$ is a strict log morphism $S'\to S$. Moreover, the degeneration data over $S'$ is isomorphic to the pull back of degeneration data along $f$, and $\psi:=f^\sharp\alpha\circ\varphi/\alpha'\varphi'$ takes values in $(R')^*$.

\begin{proposition}[functorial]\label{AN functorial}
The pull back of the AN family $(\mathcal{X},\mathcal{L}, G,\varrho)$ over $S$ is isomorphic to the AN family $(\mathcal{X}',\mathcal{L}', G',\varrho')$ over $S'$. 
\end{proposition}
\begin{proof}
Consider the following map
\begin{align}
\Pi: k[S(X)\rtimes_{\varphi}P]\otimes_{k[P]}\otimes R'&\to k[S(X')\rtimes_{\varphi'} P']\otimes_{k[P']}R'\\
\zeta^{(\alpha,p)}&\mapsto f^\sharp(\zeta^{(\alpha,p)})\psi(\alpha)^{\deg(\alpha)}.
\end{align}

We need to check that $\Pi$ is homomorphism of algebras, i.e. 
\[
\Pi(\zeta^{(\alpha,p)}\zeta^{(\beta,q)})=\Pi(\zeta^{(\alpha,p)})\Pi(\zeta^{(\beta,q)}).
\]

Since the morphism is strict, we can write $f^\sharp(\zeta^{(\alpha,p)})$ as $\zeta^{(\alpha, p')}r'$ for $r'\in (R')^*$. Then it reduces to the following identity
\begin{equation*}
f^\sharp(\zeta^{\tilde{\varphi}(\alpha)+\tilde{\varphi}(\beta)-\tilde{\varphi}(\alpha+\beta)})\psi(\alpha+\beta)^{\deg(\alpha+\beta)}=(\zeta^{\tilde{\varphi}'(\alpha)+\tilde{\varphi}'(\beta)-\tilde{\varphi}'(\alpha+\beta)})\psi(\alpha)^{\deg(\alpha)}\psi(\beta)^{\deg(\beta)}.
\end{equation*}

Since $\psi$'s are invertible in $R'$, we get an isomorphism $\widetilde{P}^r\times_SS'\to (\widetilde{P}^r)'$. Moreover, the pull back of $\widetilde{G}$ is $\widetilde{G}'$ by assumption. Since the contracted product commutes with the pull back, we have the isomorphism $\widetilde{\mathcal{X}}'\to \widetilde{\mathcal{X}}\times_{S}S'$. By assumption, the degeneration data also commutes with the pull back, the isomorphism commutes with $Y$-action. 
\end{proof}

Given $S$ as above, with degeneration data, a function $\varphi: X_\mathbf{R}\to P^\gp_\mathbf{R}$, and a chart $\alpha: P\to R$ such that $\alpha\circ \varphi$ is compatible with the degeneration data. We get the family $(\mathcal{X}_I, \mathcal{L}_I, G_I, \varrho_I)$ over $R/I$. Suppose there is a different chart $\alpha':P\to R$, then $\psi:=\alpha\circ\varphi/\alpha'\circ\varphi$ takes values in $R^*$. Therefore $\alpha'\circ\varphi$ is also compatible with the degeneration data. Since the graded sheaf of algebra $\mathcal{R}$ as a subsheaf of $\mathcal{S}$ is defined independent of the choice of the chart $\alpha: P\to R$, the set of admissible ideals are the same for the two charts. We get the same family $(\mathcal{X}_I, \mathcal{L}_I, G_I, \varrho_I)$ over $R/I$ by using $\alpha'$ in the AN construction. Moreover the Fourier decompositions \eqref{Fourier decomposition equation} are the same for $\alpha$ and $\alpha'$. 

\begin{lemma}
The AN construction $(\mathcal{X},\mathcal{L},G,\varrho)$, and the Fourier decomposition 
\[
\pi_*\mathcal{L}=\bigoplus_{\alpha\in X/\phi(Y)}\pi_*\mathcal{M}_\alpha
\] 

only depends on the log structure induced by $\alpha: P\to R$. 
\end{lemma}
 
Since a log structure is an \'{e}tale sheaf, we can glue the above local models and generalize the base $S$. 
\begin{assumption}\label{basic assumption}
Assume the log scheme $(S,M_S)$ is log smooth over a base $(\overline{S}, \mathcal{O}_{\overline{S}}^*)$. Moreover, assume that the base $\overline{S}$ is Noetherian, excellent, and integral. 
\end{assumption} 

Denote the open subset where $M_S$ is trivial by $S^\circ$, and the boundary $S\backslash S^\circ$ by $\partial S$. We can define the degeneration data over $S$. It is defined to be an \'{e}tale sheaf of data ($A$,$A^t$, $\widetilde{G}$,$\widetilde{G}^t$, $T$, $T^t$, $c$, $c^t$, $X$, $Y$, $\phi$, $\psi$, $\tau$, $\mathcal{M}$), such that for each \'{e}tale affine neighborhood $U$ of $S$, and the completion of $U$ with respect to $U\times_S\partial S$, it is the degeneration data defined before. For any point $x\in S$, we can choose an affine \'{e}tale neighborhood $U\to S$ such that $U=\spec R$ is isomorphic to a product $\overline{S}\times \spec k[P]$ with some base scheme $\overline{S}$. A choice of such isomorphism gives a chart $\alpha: P\to R$. Fix a piecewise affine function $\varphi: X_\mathbf{R}\to P_\mathbf{R}^\gp$ as above. By choosing the local chart, we can talk about whether $\varphi$ is compatible with the degeneration data, and the set of admissible ideals. 

\begin{lemma}
Whether $\varphi$ is compatible with the degeneration data over $U$, and the set of admissible ideals are both independent of the choice of the trivialization of $U$. 
\end{lemma}

\begin{proof}
This is because different trivializations give the same log structure. 
\end{proof}

\begin{definition}
The function $\varphi$ and the log structure on $S$ is said to be compatible with the degeneration data over $S$ if they are compatible under any local trivialization. 
\end{definition}

\begin{definition}
A closed subscheme $Z$ with support inside $\partial S$ is called admissible if restricted to any affine \'{e}tale open set $U$, under any local trivialization of $U$, $Z$ corresponds to an admissible ideal. 
\end{definition}

\begin{theorem}\label{alexeev family}
Assume $S$ satisfies Assumption~\ref{basic assumption}. Assume there is degeneration data over $S$ and a piecewise affine function $\varphi: X_\mathbf{R}\to P_\mathbf{R}^\gp$ compatible with the degeneration data through the log structure of $(S, M_S)$. Then for any admissible closed subscheme $Z$, we construct a family $(\mathcal{X}_Z,\mathcal{L}_Z, G_Z,\varrho_Z)$ over $Z$. Furthermore, there is a subsheaf of $\pi_*\mathcal{L}$ defined by the Fourier decomposition \eqref{Fourier decomposition equation} such that for any $\Theta_Z$ defined by a section of this subsheaf, the family $(\mathcal{X}_Z, \mathcal{L}_Z, G_Z, \varrho_Z,\Theta_Z)$ is an object in $\overline{\mathscr{AP}}_{g,d}(Z)$. This construction is functorial with respect to morphisms of admissible closed subschemes $Z'\to Z$. 
\end{theorem}

Let $(\mathcal{X},\mathcal{L}, G, \varrho)/S$ be an AN family. Then the relatively ample line bundle $\mathcal{L}$ over $\mathcal{X}$ induces a polarization on the semiabelian scheme $G/S$. We can associate constructible \'{e}tale sheaves $\underline{X}$, $\underline{Y}$ over $S$, and a bilinear pairing $\underline{B}:\underline{Y}\times\underline{X}\to \underline{\Div} S$. Here $\underline{\Div} S$ is the sheaf of Cartier divisors. See (\cite{FC} Chap. \Rmnum{3} Theorem 10.1). For $s\in S$, the fiber $\underline{X}_s$ is the group of characters for the toric part of $G_s$, and $\underline{Y}_s$ is the group of characters for the toric part of $G^t_s$. The pairing $\underline{B}$ agrees with $\tau$. Notice that the open locus where $\underline{X}, \underline{Y}, \underline{B}$ vanishes is exactly the open locus $S^\circ$ where the log structure $M_S$ is trivial.

%%%%%%%%%%%%%%%%%%%%%%%%%%%%%%%%%%%%%%%%%
%%%%%%%%%%%%%%%%%%%%%%%%%%%%%%%%%%%%%%%%%

\subsection{Standard Data}\label{3.3}
In this section, we specify the degeneration data for the versal families, and compare them with the standard construction in \cite{ols08}. Fix an $X$-invariant integral paving $\mathscr{P}$ of $X_\mathbf{R}$. By definition, $\mathscr{P}$ is obtained as the set of affine domains of some $X$-quasiperiodic real-valued piecewise affine function $\psi$. Associated with $\mathscr{P}$ is a rational polyhedral cone $C(\mathscr{P})$ in the second Voronoi fan $\Sigma(X)$. The support of $\Sigma(X)$ is $\mathcal{C}(X)^\rc\subset \Gamma^2U$. Let $CPA(\mathscr{P},\mathbf{R})$ be the space of convex piecewise affine functions whose associated paving is coarser than $\mathscr{P}$\footnote{Notice we define $CPA(\mathscr{P},\mathbf{R})$ to be closed.}. Let $CPA^X(\mathscr{P},\mathbf{R})$ be the subspace of $X$-quasiperiodic functions. The map $\psi\to Q$, which maps $\psi$ to the associated quadratic form $Q$, identifies $CPA^X(\mathscr{P},\mathbf{R})/Aff$ with $C(\mathscr{P})$. Let $C(\mathscr{P},\mathbf{Z})$ be the image of integral convex functions in $CPA(\mathscr{P},\mathbf{R})/Aff$ and $C^X(\mathscr{P},\mathbf{Z})$ be the image of integral, $X$-quasiperiodic, convex functions in $C(\mathscr{P})$. 

Recall the definition of $H_\mathscr{P}$ in \cite{ols08}. For any $\sigma_i\in\mathscr{P}$ define $N_i=S(\sigma_i)$. Define $N_\mathscr{P}:=\varinjlim N_i$ in the category of integral monoids. The cone $S(X)=\varinjlim N_i$ in the category of sets. By universal property of $S(X)$ we have a natural map $\tilde{\varphi}': S(X)\to N_\mathscr{P}$. Since $\Hom(\varinjlim N_i, \mathbf{Z})\cong \varprojlim(\Hom(N_i,\mathbf{Z}))$, The group $N_\mathscr{P}^{\gp}$ is equal to $PA(\mathscr{P},\mathbf{Z})^*$. 

Let $\widetilde{H}_\mathscr{P}$ be the submonoid of $N_\mathscr{P}^{\gp}$ generated by
\[
\alpha*\beta:=\tilde{\varphi}'(\alpha)+\tilde{\varphi}'(\beta)-\tilde{\varphi}'(\alpha+\beta), \quad \forall \alpha,\beta\in S(X). 
\]

\begin{proposition}
Let $C(\mathscr{P},\mathbf{Z})^\vee$ be the dual of the monoid $C(\mathscr{P},\mathbf{Z})$ in $\Hom (C(\mathscr{P},\mathbf{Z}),\mathbf{Z})$. 
\[
\widetilde{H}_\mathscr{P}^\sat=C(\mathscr{P},\mathbf{Z})^\vee.
\]
\end{proposition}
\begin{proof}
It is easy to see that $\alpha*\beta\in C(\mathscr{P},\mathbf{Z})^\vee$ by evaluation. It follows that $\widetilde{H}_\mathscr{P}\subset C(\mathscr{P},\mathbf{Z})^\vee$. On the other hand, for any $\bar{\psi}\in PA(\mathscr{P},\mathbf{Z})/Aff$. Pick any lift $\psi\in PA(\mathscr{P},\mathbf{Z})$. $\psi(\alpha*\beta)\geqslant 0$ for all $\alpha,\beta$ if and only if it is convex. Therefore $(\widetilde{H}_{\mathscr{P}})^\vee=C(\mathscr{P},\mathbf{Z})$. We claim that $\widetilde{H}_\mathscr{P}^{\gp}=\ann(Aff)\cap N_\mathscr{P}^{\gp}$. Assuming this, then, since $\widetilde{H}_\mathscr{P}$ is sharp, $\widetilde{H}_\mathscr{P}^{\sat}=((\widetilde{H}_\mathscr{P})^\vee)^\vee=C(\mathscr{P},\mathbf{Z})^\vee$. 

Now we prove the claim. In the proof of (\cite{ols08} Lemma 4.1.6), it is shown that the image of $\pi: SC_1(\mathbb{X}_{\geqslant 0})'/B_1\to \widetilde{H}_\mathscr{P}$ is equal to $\widetilde{H}_\mathscr{P}$. By the description of $SC_1(\mathbb{X}_{\geqslant 0})'/B_1$, it is the monoid generated by images of all affine linear relations for $\mathbb{X}_\mathbf{R}$. Therefore $\widetilde{H}_\mathscr{P}^{\gp}$ is saturated in $\ann(Aff)$.  
\end{proof}

The group $X$ is acting on $S(X)\subset\mathbb{X}$. It induces an action of $X$ on $N_\mathscr{P}$ and $\widetilde{H}_\mathscr{P}$. The quotient of this action is denoted by $H_\mathscr{P}$. There is a natural map $H_\mathscr{P}\to C^X(\mathscr{P},\mathbf{Z})^\vee$.

\begin{proposition}
We have 
\begin{align}
\Hom(H_\mathscr{P},\mathbf{Z}_{\geqslant 0})&=C^X(\mathscr{P},\mathbf{Z}),\\
H_\mathscr{P}^\sat/(H_\mathscr{P}^\sat)_{\text{tor}}&=C^X(\mathscr{P},\mathbf{Z})^\vee.
\end{align}
\end{proposition}
\begin{proof}
We prove $C^X(\mathscr{P},\mathbf{Z})\subset \Hom(H_\mathscr{P},\mathbf{Z}_{\geqslant0})$ first. The inclusion $C^X(\mathscr{P},\mathbf{Z})\to C(\mathscr{P})\subset \Gamma^2U$ induces a map $S^2X\to H_\mathscr{P}^{\gp}$, and this is the map $s: S^2X\to H_\mathscr{P}^{\gp}$ defined in (\cite{ols08} Lemma 5.8.2). By (\cite{ols08} Lemma 5.8.16), $C(\mathscr{P})^\gp_{\mathbf{Q}}\cong \Hom(H_\mathscr{P}, \mathbf{Q})$. Therefore $C^X(\mathscr{P},\mathbf{Z})\subset \Hom(H_\mathscr{P},\mathbf{Z}_{\geqslant0})$ is an inclusion.  On the other hand, for each $\psi\in \Hom(H_\mathscr{P},\mathbf{Z}_{\geqslant 0})$, denote its image in $\Hom(\widetilde{H}_\mathscr{P},\mathbf{Z}_{\geqslant0})$ by $\tilde{\psi}$. Since $\widetilde{H}_\mathscr{P}^\sat=C(\mathscr{P},\mathbf{Z})^\vee$, $\tilde{\psi}$ is an element in $C(\mathscr{P},\mathbf{Z})$. Being invariant under the action of $X$ means exactly being $X$-quasiperiodic, therefore $\psi\in C^X(\mathscr{P},\mathbf{Z})$. 

Both $C^X(\mathscr{P},\mathbf{Z})^\vee$ and $H_\mathscr{P}^\sat/(H_\mathscr{P}^\sat)_{\text{tor}}$ are toric, and $\Hom(H_\mathscr{P},\mathbf{Z}_{\geqslant 0})=\Hom(H_\mathscr{P}^\sat/(H_\mathscr{P}^\sat)_{\text{tor}},\mathbf{Z}_{\geqslant 0})$. So the second statement follows from the first statement. 
\end{proof}

Denote the monoid $C^X(\mathscr{P},\mathbf{Z})^\vee$ by $P_\mathscr{P}$. It is a sharp toric monoid. The natural morphism $H_\mathscr{P}^\sat\to P_\mathscr{P}$ is the quotient by the torsion, since $s_\mathbf{Q}: S^2X\otimes\mathbf{Q}\to H^{\gp}_{\mathscr{P},\mathbf{Q}}$ is an isomorphism (\cite{ols08} Proposition 5.8.15).

If the paving $\mathscr{P}$ is a triangulation $\mathscr{T}$, the cone $C(\mathscr{P})$ is of maximal dimension in $\Gamma^2U$. The group $\Hom(C^X(\mathscr{P},\mathbf{Z})^{\gp},\mathbf{Z})$ is a lattice in $S^2U^*$, and is denoted by $\mathbb{L}_\mathscr{P}$. We have $S^2X\subset \mathbb{L}_\mathscr{P}$. For a general paving $\mathscr{P}$, let $I_\mathscr{P}$ denote the set of triangulations that refine $\mathscr{P}$. Define
\[
\mathbb{L}_\mathscr{P}:=\sum_{\mathscr{T}\in I_\mathscr{P}}\mathbb{L}_\mathscr{T}.
\]     
Let $C(\mathscr{P})^\vee$ be the dual cone of $C(\mathscr{P})$ in $S^2U^*$, and $S_\mathscr{P}$ be the toric monoid $C(\mathscr{P})^\vee\cap \mathbb{L}_\mathscr{P}$. If $\mathscr{P}$ is a triangulation, $S_\mathscr{P}=P_\mathscr{P}$.  

Introduce the monoid $S(X)\rtimes H_\mathscr{P}$ on the set $S(X)\times H_\mathscr{P}$ with the addition law
\[
(\alpha,p)+(\beta,q)=(\alpha+\beta, p+q+ \alpha*\beta).
\]

The morphism $H_\mathscr{P}^\sat\to P_\mathscr{P}$ induces the natural morphism $S(X)\rtimes H_\mathscr{P}^\sat\to S(X)\rtimes P_\mathscr{P}$. 

Consider the minimal models of the mirror family $\mathcal{Y}^{\ppav}$ in the principally polarized case. By Theorem~\ref{Two fans are the same}, we can identify the second Voronoi fan $\Sigma(X)$ with the Mori fan of $\mathcal{Y}^{\ppav}$. For any bounded paving $\mathscr{P}$, the nef cone $\nef(\mathcal{Y}_\mathscr{P}^0)\cap \pic(\mathcal{Y}_\mathscr{P}^0)$ is identified with $C^X(\mathscr{P},\mathbf{Z})$. Therefore
\begin{equation}
P_\mathscr{P}=\overline{\effcurve}(\mathcal{Y}_\mathscr{P}^{\ppav})\cap \pic(\mathcal{Y}_\mathscr{P}^{\ppav})^*.
\end{equation}

For an $X$-quasiperiodic $P_{\mathscr{P},\mathbf{R}}^\gp$-valued function $\varphi$ over $X_\mathbf{R}$, the bending parameters are $X$-periodic: $p_\rho=p_{\rho+X}$. Define an $X$-quasiperiodic, $\effcurve(\mathcal{Y}_{\mathscr{P}}^{\ppav})$-convex function $\varphi_\mathscr{P}: X_\mathbf{R}\to \Ncurve(\mathcal{Y}_\mathscr{P}^{\ppav})=P_{\mathscr{P},\mathbf{R}}^{\gp}$ by requiring that, at each codimension-$1$ cell $\rho\in \mathscr{P}$, the bending parameter $p_\rho\in \effcurve(\mathcal{Y}^{\ppav}_{\mathscr{P}})$ is the corresponding curve class $\Upsilon_*[V(C(\rho))]$, where $\Upsilon: \widetilde{\mathcal{Y}}^{\ppav}_\mathscr{P}\to \mathcal{Y}^{\ppav}_\mathscr{P}$ is the universal covering, and $\widetilde{\mathcal{Y}}^\ppav_\mathscr{P}$ is the infinite toric variety for the fan $\{C(\sigma)\}_{\sigma\in \mathscr{P}}$. 

\begin{lemma}\label{universal map}
Let $\rho$ be a codimension-$1$ wall between maximal cells $\sigma_i$ and $\sigma_j$. Let $\omega\in X$ be a vector that maps to a primitive generator of $X/(\mathbf{R}\rho\cap \mathbf{Z})\cong \mathbf{Z}$. For any integral, real-valued, $X$-quasiperiodic piecewise affine function $\psi$ whose paving is coarser than $\mathscr{P}$, define $\psi_i$, $\psi_j$ to be the affine extension of $\psi\vert_{\sigma_i}$ and $\psi\vert_{\sigma_j}$. The bending parameter $p_\rho$ for $\varphi_\mathscr{P}$ is 
\[
p_\rho(\psi)=\psi_i(\omega)-\psi_j(\omega). 
\] 

In particular, the function $\varphi_\mathscr{P}$ is integral with respect to the integral structure $P_\mathscr{P}$. 
\end{lemma}

\begin{proof}
Since $\psi$ is $X$ quasi-periodic, the divisor $D_\psi$ it represents is inside $\pic^X(\widetilde{\mathcal{Y}}^{\ppav}_\mathscr{P})$. Therefore
\[
p_\rho(\psi)=(\Upsilon_*[V(C(\rho))])\cdot D_\psi=\Upsilon_*([V(C(\rho))]\cdot \Upsilon^*D_\psi)=\psi_i(\omega)-\psi_j(\omega). 
\] 

The last equality is by the formula for toric varieties, since $\widetilde{\mathcal{Y}}^{\ppav}_\mathscr{P}$ is locally a toric variety.

If $\psi$ is integral, since $\omega\in X$, $p_\rho(\psi)$ is an integer. In other words, $p_\rho\in P_\mathscr{P}$, and $\varphi_\mathscr{P}$ is integral. 
\end{proof}

\begin{corollary}\label{interpretation of universal map}
For any function $\psi\in C(\mathscr{P})$, the evaluation $\varphi_\mathscr{P}$ on $\psi$ is equal to $\psi$ up to a linear function. 
\end{corollary}
\begin{proof}
Choose a piecewise affine function $\psi$ as the representative. By the Lemma~\ref{universal map}, the bending parameters of $\psi\circ\varphi_\mathscr{P}$ are equal to the bending parameters of $\psi$. 
\end{proof}

If $C(\mathscr{P}')$ is a face of a cone $C(\mathscr{P})$. Consider the dual $C(\mathscr{P}')^\vee$ in the space $\Ncurve(\mathcal{Y}_\mathscr{P}^\ppav)$. Since $\mathscr{P}$ refines $\mathscr{P}'$, there is a contraction $f:\mathcal{Y}^\ppav_\mathscr{P}\to \mathcal{Y}^\ppav_{\mathscr{P}'}$ inducing $f_*: \Ncurve(\mathcal{Y}_\mathscr{P}^{\ppav})\to \Ncurve(\mathcal{Y}_{\mathscr{P}'}^{\ppav})$. We have an exact sequence of monoids
 \[
\begin{diagram}\label{birational contraction}
0&\rTo&\mathcal{R}&\rTo&C(\mathscr{P}')^\vee &\rTo^{f_*}&\overline{\effcurve}(\mathcal{Y}_{\mathscr{P}'}^{\ppav})&\rTo &0,
\end{diagram}
\]

where $\mathcal{R}$ is the $\mathbf{R}$-vector space generated by the $f$-contracted curve classes. 

\begin{corollary}\label{compatible sections}
The standard sections $\varphi$ are compatible. 
\[
\varphi_{\mathscr{P}'}=f_*\circ\varphi_\mathscr{P}: B\to \Ncurve(\mathcal{Y}_{\mathscr{P}'}^{\ppav}). 
\]
\end{corollary} 
\begin{proof}
Regard $\Ncurve(\mathcal{Y}_\mathscr{P}^{\ppav})$ as the dual space of $\pic(\mathcal{Y}^{\ppav}_{\mathscr{P}})$, and $\pic(\mathcal{Y}^{\ppav}_{\mathscr{P}'})\to \pic(\mathcal{Y}^{\ppav}_\mathscr{P})$ as an inclusion. The map $f_*$ is the restriction of functions to the subspace $C(\mathscr{P}')$. By Lemma \ref{universal map}, the bending parameters have the same description as functionals on $C(\mathscr{P}')$. 
\end{proof}

Let $(\mathcal{X},\mathcal{L}, G,\varrho)$ be an AN family over $S=\spec R$ constructed from $\varphi_\mathscr{P}$ and a chart $\alpha: P_\mathscr{P}\to R$. We add a log structure to the family. Locally on the abelian scheme $A/S$, choose a trivialization of $\mathcal{M}$ and compatible trivializations of $\mathcal{O}_\alpha$, the algebra $\mathcal{R}$ is isomorphic to $k[S(X)\rtimes P_\mathscr{P}]\otimes_{k[P_\mathscr{P}]}\mathcal{O}_A$. Define the log structure $\widetilde{P}\to \mathcal{O}_{\widetilde{\mathcal{X}}}$ locally by the descent of the chart
\[
S(X)\rtimes P_\mathscr{P}\to k[S(X)\rtimes P_\mathscr{P}]\otimes_{k[P_\mathscr{P}]}\mathcal{O}_A. 
\]

For any $n\in \mathbf{N}$, do the reduction over $S_n:=\spec R/I^n$, the pull-back $\widetilde{P}_n$ on $\widetilde{X}_n$ descends to a log structure $P_n$ on the quotient $X_n$ by (\cite{ols08} 4.1.18, Lemma 4.1.19, \& 4.1.22).  

By (\cite{ols08} Lemma 4.1.11), $(\widetilde{\mathcal{X}},\widetilde{P})$ is integral and log smooth over $(S,M_S)$. Therefore $(X_n,P_n)$ is log smooth and integral over $(S_n, M_{S_n})$ for every $n$.

The underlying topological space of $S_n$ inherits a stratification from the toric stratification of $\spec k[P_\mathscr{P}]$. Since the image of $\spec R/I$ is in the union of the toric strata defined by $I_{\varphi_\mathscr{P}}$, each stratum corresponds to an admissible prime toric ideal $J$. Let $J$ be an admissible prime toric ideal of $P_\mathscr{P}$, and $F$ be the face $P_\mathscr{P}\backslash\ J$. The monoid $P_\mathscr{P}$ is the integral points in $C(\mathscr{P})^\vee$, for $C(\mathscr{P})$ a cone in the second Voronoi fan. Each face $\tau$ of $C(\mathscr{P})$ is a cone $C(\mathscr{P}')$ for some coarser paving $\mathscr{P}'$. Therefore $F=C(\mathscr{P}')^\perp\cap C(\mathscr{P})^\vee$. Then the stratum of $S_n$ is called the stratum associated with the paving $\mathscr{P}'$, or the $\mathscr{P}'$-stratum. 

We can associate a paving and a piecewise affine function to the face $F$ as in Definition~\ref{data associated with a face}. By Corollary~\ref{compatible sections}, the piecewise affine function is $\varphi_{\mathscr{P}'}$, and the paving is $\mathscr{P}'$.  

\begin{lemma}\label{the monoid}
\[
P_\mathscr{P}/(F\cap P_\mathscr{P})=P_{\mathscr{P}'}.
\]
\end{lemma}
\begin{proof}
The monoid $P_\mathscr{P}$ is given by the intersection $\overline{\effcurve}(\mathcal{Y}_\mathscr{P}^\ppav)\cap\pic(\mathcal{Y}_\mathscr{P}^\ppav)^*$. The dual integral structure is $C^X(\mathscr{P},\mathbf{Z})^{\gp}$. We claim that $C^X(\mathscr{P}',\mathbf{Z})=C(\mathscr{P}')\cap C^X(\mathscr{P},\mathbf{Z})^{\gp}$. Assume $\psi$ is a piecewise function and its image is in $C^X(\mathscr{P}',\mathbf{Z})$. Since $\mathscr{P}'$ is coarser than $\mathscr{P}$. Each top-dimensional cell $\sigma\in\mathscr{P}$ is contained in some top-dimensional cell $\sigma'\in \mathscr{P}'$ of $\mathscr{P}'$. Therefore, the restriction of $\psi$ to $\sigma$ is integral, and $\psi\in C(\mathscr{P}')\cap C^X(\mathscr{P},\mathbf{Z})^{\gp}$. On the other hand, each top-dimensional cell $\sigma'\in \mathscr{P}'$ contains some top-dimensional cell $\sigma\in \mathscr{P}$. If $\psi$ is integral on $\sigma$, it is integral on $\sigma'$. Therefore if $\psi\in C(\mathscr{P}')\cap C^X(\mathscr{P},\mathbf{Z})^{\gp}$, then $\psi\in C^X(\mathscr{P}',\mathbf{Z})$. This proves the claim. It follows that the dual integral structures also agree. 
\end{proof}

\begin{corollary}\label{the gluing is etale}
Assume $C(\mathscr{P}')$ is a face of $C(\mathscr{P})$. The morphism $k[S_\mathscr{P}]\to k[S_{\mathscr{P}'}]$ is \'{e}tale. 
\end{corollary}
\begin{proof}
Recall that $S_\mathscr{P}=\mathbb{L}_\mathscr{P}\cap C(\mathscr{P})^\vee$, and $\mathbb{L}_\mathscr{P}\subset\mathbb{L}_{\mathscr{P}'}$ is a sublattice of finite index. The morphism $k[S_\mathscr{P}]\to k[S_{\mathscr{P}'}]$ can be decomposed into a localization and the inclusion $f: k[C(\mathscr{P}')^\vee\cap \mathbb{L}_\mathscr{P}]\to k[C(\mathscr{P}')^\vee\cap \mathbb{L}_{\mathscr{P}'}]$. Denote $C(\mathscr{P}')$ by $\tau$, and $C(\mathscr{P}')^\perp$ by $F^\gp$. Apply Lemma~\ref{the monoid} to all $\mathscr{T}$ that refines $\mathscr{P}'$, we get exact sequences.
\[ 
\begin{diagram}
0&\rTo&F^\gp\cap\mathbb{L}_\mathscr{P}&\rTo& \mathbb{L}_\mathscr{P}\cap \tau^\vee&\rTo&P_{\mathscr{P}'}&\rTo& 0\\
&&\dTo^{f'}&&\dTo^f&&\dTo^{=}&&\\
0&\rTo&F^\gp\cap\mathbb{L}_{\mathscr{P}'}&\rTo& \mathbb{L}_{\mathscr{P}'}\cap \tau^\vee&\rTo&P_{\mathscr{P}'}&\rTo& 0\\
\end{diagram}
\]

It follows that the first square is the pushout (\cite{ogus} Chap.\Rmnum{1} Proposition 1.1.4 part. 2). The associated morphism $f'$ is a group homomorphism of tori and is \'{e}tale, the base change $f$ is thus \'{e}tale.
\end{proof}

\begin{lemma}\label{olssen's family}
For each geometric point $\bar{x}\to S$ in interior of the $\mathscr{P}'$-stratum, the fiber $(\mathcal{X}_{n,\bar{x}}$, $\mathcal{L}_{n,\bar{x}}$, $P_{n,\bar{x}}$, $G_{n,\bar{x}})$ $\to (\bar{x}, M_{\bar{x}})$ is isomorphic to the collection of data obtained from the saturation of the standard construction defined in \cite{ols08}. 
\end{lemma}
\begin{proof}
Denote $P_\mathscr{P}$ by $P$, and $P_{\mathscr{P}'}$ by $P'$ in this proof. Assume the admissible prime toric ideal associated with $\mathscr{P}'$-stratum is $J$. Use the notations $F$, $P_F$, and $J_F$ as in Lemma \ref{toric prime}. By Lemma~\ref{the monoid}, $P_F\cong P'\oplus F^\gp$. Denote the geometric point $\bar{x}\to S\to \spec k[P]$ by $\bar{y}\to \spec k[P]$.

The fiber $(\mathcal{X}_{n,\bar{x}}, \mathcal{L}_{n,\bar{x}}, G_{n,\bar{x}})$ is isomorphic to the fiber $(\mathcal{X}_{\bar{x}}, \mathcal{L}_{\bar{x}}, G_{\bar{x}})$. Since the image of $\bar{x}$ is contained in the $\mathscr{P}'$-stratum, $\bar{y}$ is contained in the localization $\spec k[P_F]$.  Therefore, we can replace $R$ by the base change $R'=k[P_F]\otimes_{k[P]}R$. As in the proof of Lemma \ref{toric prime}, after the localization $k[P_F]$, the function $\varphi_\mathscr{P}$ is equal to $\varphi_{\mathscr{P}'}+\psi$, where $\psi$ has bending parameters in $F^{\gp}$. Up to a global affine function, we can assume that $\varphi_\mathscr{P}=\varphi_{\mathscr{P}'}\oplus \psi$ with the values in $P_F=P'\oplus F^{\gp}$. Since $F^{\gp}$ are sent to $(R')^*$, the pull back of the degeneration data is compatible with $\varphi_{\mathscr{P}'}$. By Lemma~\ref{AN functorial}, the base change of $\mathcal{X}$ to $R'$ is isomorphic to the AN construction by using $P'\to R'$ and $\varphi_{\mathscr{P}'}$. Moreover, over $S'=\spec R'$, we can use the chart $P'\to R'$ for the log structure on the base. Pull back to $\Omega(k)$, $P'\to \Omega(x)$ is the chart.

Assume $\bar{x}=\spec \Omega(x)$. Since the prime ideal of the image of $\bar{y}$ contains $J_F$, which further contains $P'\backslash\{0\}$, the set $P'\backslash\{0\}$ is sent to $0$ in $\Omega(x)$. Consider the composition $H_{\mathscr{P}'}\to H^\sat_{\mathscr{P}'}\to P'\to \Omega(x)$. Since $H_{\mathscr{P}'}$ is sharp (\cite{ols08} Lemma 4.1.8), all the nonzero elements of $H_{\mathscr{P}'}$ are sent to $0\in\Omega(x)$. Further pull back the degeneration data to $\Omega(x)$. Use the pull-back of $b_t^{-1}\tau$ and $a_t^{-1}\psi$ as the trivializations, we get the data in (\cite{ols08} 5.2.1). The fiber $(\mathcal{X}_{\bar{x}}, \mathcal{L}_{\bar{x}}, G_{\bar{x}},\varrho_{\bar{x}}, S(X)\rtimes P')$\footnote{Here we use the local chart to denote the log structure.} over $(\bar{x}, P')$ is a base change of the saturation of the standard construction in (Loc.cit. 5.2). 
\end{proof}

It follows that the family $(\mathcal{X}_n, \mathcal{L}_n, P_n, G_n,\varrho_n)\to (S_n, M_n)$ is an object in $\overline{\mathscr{T}}_{g,d}$. By the fact that $\overline{\mathscr{T}}_{g,d}$ is an Artin stack, the algebraization $(\mathcal{X}, \mathcal{L}, P, G,\varrho)\to (S, M)$ is in $\overline{\mathscr{T}}_{g,d}(S)$.

\begin{theorem}\label{olssen family}
The same assumption as in Theorem~\ref{alexeev family}, with $P=P_\mathscr{P}$ and $\varphi=\varphi_\mathscr{P}$. If $Z$ is a closed subscheme admissible for $\varphi_\mathscr{P}$, then there exists a family $\pi:(\mathcal{X}_Z, \mathcal{L}_Z, P_Z, G_Z, \varrho_Z)$ over $(Z, M_Z)$, which is an object in $\overline{\mathscr{T}}_{g,d}(Z)$. Moreover, this construction is functorial with respect to $Z'\to Z$. 
\end{theorem}

\begin{remark}[The Log Structure]\label{the divisorial log structure}
When we construct the versal families, we will have $\overline{S}$ regular. Then both the base $(S, M_S)$ and $(\mathcal{X}, P)$ are log regular. Denote the open subset where $M_S$ (resp. $P$) is trivial by $S^\circ$ (resp. $\mathcal{X}^\circ$), and the boundary by $\partial S$ (resp. $\partial \mathcal{X}$). By (\cite{kato94} Theorem 11.6), the log structures $M_S$ and $P$ are isomorphic to the divisorial log structures. However, the divisorial log structures are not preserved by the finite base change, unless the base change is \'{e}tale. The log structures $M_S$ and $P$ are minimal in the sense that, when we define the versal family, we take the minimal base change such that the fibers are all reduced ($\varphi_\mathscr{P}$ is integral). 
\end{remark}

%%%%%%%%%%%%%%%%%%%%%%%%%%%%%%%%%%%%%%%%%%%%%%%%%%%%%%%%%%%%%%%%%%%%%%%%%%%%%%%%%%
%%%%%%%%%%%%%%%%%%%%%%%%%%%%%%%%%%%%%%%%%%%%%%%%%%%%%%%%%%%%%%%%%%%%%%%%%%%%%%%%%%
%%%%%%%%%%%%%%%%%%%%%%%%%%%%%%%%%%%%%%%%%%%%%%%%%%%%%%%%%%%%%%%%%%%%%%%%%%%%%%%%%%

\subsection{Construction of the Stack}\label{3.4}
From now on, fix the base ring $k=\mathbf{Z}[1/d,\zeta_{M}]$, where $M=2\delta_g$, and $\zeta_M$\label{roots of unity} is a primitive $M$-th root of unity. We also write the primitive root of unit $\zeta_{M}$ as $\exp(2\pi i /M)$. As a localization of the Dedekind domain $\mathbf{Z}[\zeta_{M}]$, $k$ is a Dedekind domain of characteristic $0$. In particular, the bases below would satisfy Assumption~\ref{basic assumption}. 

\subsubsection{Local Charts: Formal Theory}
Fix a general cusp $F_\xi$, and we decorate every notation associated with this cusp by $\xi$. Assume the associated rational isotropic subspace is $U_\xi$ of dimension $r$. Let $X_\xi^*=U_\xi\cap \Lambda$ and $X_\xi=\Hom(X_\xi^*,\mathbf{Z})$. Choose a basis $\{v_1,\ldots,v_r\}$ of $X_\xi^*$. Define $Y_\xi=\Lambda/U_\xi^\perp\cap\Lambda$. The restriction of the pairing $E$ defines the polarization $\phi: Y_\xi\to X_\xi$ of type $\mathfrak{d}$. Lift a compatible basis of $Y_\xi$ to $\{u'_1,\ldots,u_r'\}\subset \Lambda$. The restriction of $E$ is represented by a skew-symmetric integral matrix $S_\xi$ under the basis $\{u_1',\ldots, u_r'\}$. Choose a symmetric integral matrix $S'_\xi$ such that $S_\xi'\equiv S_\xi\pmod{2\mathbf{Z}}$. Define the twist data
\begin{align}
b': &Y_\xi\times X_\xi\to k, \quad a': Y_\xi\to k,\\
b'(\lambda,\alpha)&=\exp{\Bigg(-\pi i (x_1'(\lambda),\ldots,x_r'(\lambda))S_\xi\mathfrak{d}^{-1}\begin{pmatrix}v_1(\alpha)\\ \vdots \\ v_r(\alpha)\end{pmatrix}\Bigg)},\\
a'(\lambda)&=\exp\Bigg{(-1/2\pi i (x_1'(\lambda),\ldots,x_r'(\lambda))S_\xi'\begin{pmatrix}x_1'(\lambda)\\ \vdots\\ x_r'(\lambda)\end{pmatrix}\Bigg)}, 
\end{align} 

where $\{x_1',\ldots,x'_r\}$ are the coordinates on $Y_\xi$ with respect to the basis $\{u'_1,\ldots,u'_r\}$. 

The pairing $E$ also induces a nondegenerate skew-symmetric pairing on $U_\xi^\perp/U_\xi$ of type $\delta'$. Denote $\mathscr{A}_{g',\delta'}$ by $\overline{F}_\xi$. Therefore over $\overline{F}_\xi$, we have the universal family $\mathcal{A}\times_{\overline{F}_\xi}\mathcal{A}^t$, where $\mathcal{A}$ is the universal family of abelian varieties with polarization $\lambda_\xi:\mathcal{A}\to \mathcal{A}^t$ of type $\delta'$. Let $\underline{X}_\xi$ and $\underline{Y}_\xi$ be the constant sheaves over $\overline{F}_\xi$. 

From now on, choose a $0$-cusp with associated maximal isotropic subspace $U\supset U_\xi$. Extend the basis of $X_\xi^*$ to a basis of $X^*=U\cap\Lambda$. Extend $\{u'_1,\ldots,u'_r\}$ to $\{u_1',\ldots,u_g'\}$ such that it is a lift of the basis of $Y=\Lambda/U\cap\Lambda$. Let $Y'$ denote the lattice generated by $u_{r+1}',\ldots,u'_g$, and $(X')^*$ denote the lattice generated by $v_{r+1}',\ldots,v'_g$. Under this chosen basis, $E$ is 
 \[
 E=\begin{pmatrix}
 S &\mathfrak{d}\\
 -\mathfrak{d} &0
 \end{pmatrix}.
 \]
 
Write the $g\times g$-matrix $S$ in blocks
\[
S=\begin{pmatrix}
S_\xi & S_2\\
S_3 & S_4
\end{pmatrix}.
\]

Denote the type of $\phi: Y_\xi\to X_\xi$ by $\mathfrak{d}_1$ and the type of $\phi: Y\to X$ by $\mathfrak{d}$ such that
\[
\mathfrak{d}=
\begin{pmatrix}
\mathfrak{d}_1&0\\
0& \mathfrak{d}_2
\end{pmatrix}.
\]

Consider the sheaves $\sheafhom (\underline{X}_\xi,\mathcal{A}^t)$ and $\sheafhom(\underline{Y}_\xi,\mathcal{A})$ over $\overline{F}_\xi$. The bundle $\overline{F}_\xi\ltimes \overline{\mathcal{V}}_\xi$ is the subset of $\sheafhom (\underline{X}_\xi,\mathcal{A}^t)\times\sheafhom(\underline{Y}_\xi,\mathcal{A})$ such that the following diagram commutes
\begin{equation}\label{twist of c and ct}
\begin{diagram}
Y_\xi &\rTo^{\phi} & X_\xi\\
\dTo^{c^t} & & \dTo_{c}\\
\mathcal{A} &\rTo^{\lambda_\xi} & \mathcal{A}^t
\end{diagram}.
\end{equation}

Introduce an automorphism $\iota$ on the set of the data $(c^t,c)$.
\begin{align}
c^t(\lambda)&\longmapsto \exp(-\pi i  (x_1'(\lambda),\ldots,x_r'(\lambda))S_2\mathfrak{d}_2^{-1})c^t(\lambda) \quad \forall \lambda\in Y_\xi,\label{32}\\
c(\alpha)&\longmapsto\exp(-\pi i S_3\mathfrak{d}_1^{-1}(v_1(\alpha),\ldots,v_r(\alpha))^T)c(\alpha)\quad \forall \alpha\in X_\xi.\label{33}
\end{align}

Let's explain the notations. Recall over $\mathbf{C}$, $\mathcal{A}$ can be regarded as a quotient of $\mathbb{G}_m^{g'}$, $g'=g-r$, by the periods $Y'\cong \mathbf{Z}^{g'}$. In Equation~\eqref{32}, the row vector $\exp(-\pi i  (x_1'(\lambda),\ldots,x_r'(\lambda))S_2\mathfrak{d}_2^{-1})$, as an element of $\mathbb{G}_m^{g'}$ is acting on the sections of $\mathcal{A}$. In general, notice that $d$ is invertible in the base ring $k$. The row vector is an element of $\mu_{2d_{r+1}}\times\ldots\times\mu_{2d_g}$. The choice of the $0$-cusp determines a maximal isotropic subgroup for every finite subgroup scheme of $\mathcal{A}$, and thus a morphism of $\mu_{2d_{r+1}}\times\ldots\times\mu_{2d_g}$ into $\mathcal{A}$. In Equation~\eqref{33}, the column vector $\exp(-\pi i S_3\mathfrak{d}_1^{-1}(v_1(\alpha),\ldots,v_r(\alpha))^T)$ should be regarded as an element in $\mu_M\times\ldots\times \mu_M$, acting on $\mathcal{A}^t$. The morphism of $\mu_M\times\ldots\times \mu_M$ into $\mathcal{A}^t$ is determined by the choice of the $0$-cusp and the polarization. The map $\iota$ is an automorphism of the set of $(c^t,c)$ where the diagram~\eqref{twist of c and ct} commutes, because $S_2=-S_3^T$ and they are both integral matrices. 

Therefore, after using $\iota$, there are the tautological extensions over $\overline{F}_\xi\ltimes \overline{\mathcal{V}}_\xi$ 
\[
\begin{diagram}
1&\rTo& T&\rTo&\widetilde{G}&\rTo^\pi&\mathcal{A}&\rTo&0,\\
1&\rTo&T^t&\rTo&\widetilde{G}^t&\rTo^{\pi^t}&\mathcal{A}^t&\rTo&0.
\end{diagram}
\]

Let $\mathbb{L}_\xi^*\subset \Gamma^2U_\xi$ be the integral structure from the integral polarized tropical abelian varieties, and $\mathbb{L}_\xi\subset S^2U_\xi^*$ be its dual.  Identify $\mathcal{C}(F_\xi)$ with $\mathcal{C}(X_\xi)$. The torus is $T_{\xi}=\mathbb{L}_\xi^*\otimes \mathbb{G}_m$.  The $T_{\xi}$-bundle $\Xi_{\xi}$ over $\overline{F}_\xi\ltimes\overline{\mathcal{V}}_\xi$ is defined as follows: For each character $\phi(\lambda)\otimes\alpha\in \mathbb{L}_\xi\subset S^2U_\xi^*$, the push-out along $\phi(\lambda)\otimes\alpha: T_{\xi}\to \mathbb{G}_m$ is defined to be the rigidified $\mathbb{G}_m$-torsor $(c^t(\lambda)\times c(\alpha))^*\mathcal{P}_A^{-1}$ over $\overline{F}_\xi\ltimes\overline{\mathcal{V}}_\xi$. Here $\mathcal{P}_A$ is the pull-back of the Poincar\'e bundle $\mathcal{P}_A$ over $\mathcal{A}\times_{\overline{F}_\xi}\mathcal{A}^t$. Since $(c^t(\lambda)\times c(\alpha))^*\mathcal{P}_A^{-1}$ is defined to be the push-out, the pull-back of it over $\Xi_{\xi}$ is canonically trivial. Denote this tautological trivialization by $\tau_\Xi$. The space $\Xi_{\xi}$ is the moduli space of  the trivializations of biextensions $(c^t\times c)^*\mathcal{P}_A^{-1}$. Since $(\Id\times \phi)^*b'$ is symmetric on $Y_\xi\times Y_\xi$, $b'$ is acting on the trivializations of trivial biextensions over $Y_\xi\times X_\xi$. Define $\tau$ to be
\[
\tau(\lambda,\alpha)=b'(\lambda,\alpha)\tau_\Xi(\lambda,\alpha). 
\]

\'{E}tale locally, we can choose a line bundle $\mathcal{M}$ over $\mathcal{A}$ which gives the polarization $\lambda_\xi$. This gives a $T$-linearized sheaf $\widetilde{\mathcal{L}}=\pi^*\mathcal{M}$ over $\widetilde{G}$. Moreover, \'{e}tale locally, there exists a trivialization $\psi$ compatible with the universal trivialization $\tau$. We make an \'{e}tale base change so that $\psi$ and $\mathcal{M}$ are defined over $\Xi_\xi$. 

The standard data is the second Voronoi fan $\Sigma(X_\xi)$ with the support $\mathcal{C}(X_\xi)^\rc$. For each cone $C(\mathscr{P})\in \Sigma(X_\xi)$ that is in the interior of $\mathcal{C}(X_\xi)$, we have the lattice $\mathbb{L}_\mathscr{P}$. Let $T_\mathscr{P}$ denote the torus with character group $\mathbb{L}_\mathscr{P}$. \'{E}tale locally, pick a section of the $T_\xi$-torsor $\Xi_\xi$ and define the projection $\Xi_\xi\to T_\xi$. Make the \'{e}tale base change $\Xi_{T,\mathscr{P}}:= \Xi_\xi\times_{T_\xi}T_\mathscr{P}$ along the morphisms $T_\mathscr{P}\to T_\xi$. Define $\Xi_\mathscr{P}:=\Xi_{T,\mathscr{P}}\times^{T_\mathscr{P}}U_{C(\mathscr{P})}$, and the closed subscheme $\mathfrak{Z}_\mathscr{P}=\Xi_{T,\mathscr{P}}\times^{T_\mathscr{P}}V(C(\mathscr{P}))$. Define $\widehat{\Xi}_\mathscr{P}$ to be the completion of $\Xi_\mathscr{P}$ along $\mathfrak{Z}_\mathscr{P}$. For each cone $C(\mathscr{P})$, the fiber bundle $\widehat{\Xi}_\mathscr{P}$ satisfies the Assumption~\ref{basic assumption}, and serves as the formal base along the boundary. While the fiber bundle $\Xi_\mathscr{P}$ provides the \'{e}tale neighborhoods near the boundary.

%\begin{remark}\label{the stacky issue}
%In our toroidal compactification over $\mathbf{C}$, \'{e}tale locally, the coarse moduli space over the cusp $F_\xi$ is the arithmetic quotient of a bundle fibered by the toric variety $X_{\Sigma(F_\xi)}$. Here the toric variety $X_{\Sigma(F_\xi)}$ is defined by using the lattice $\mathbb{L}_\xi$. The difference between lattices $\mathbb{L}_\xi\subset \mathbb{L}_\mathscr{P}$ suggest that the moduli space is stacky. 
%\end{remark}

We extend the family over the formal base $\widehat{\Xi}_\mathscr{P}$. Choose a sharp chart $\alpha: P_\mathscr{P}\to S_\mathscr{P}=\mathbb{L}_\mathscr{P}\cap C(\mathscr{P})^\vee$ for the log structure on the affine toric variety $U_{C(\mathscr{P})}$. Since $\Xi_\mathscr{P}$ is locally a product of the toric variety $U_{C(\mathscr{P})}$ with a regular scheme, $\Xi_\mathscr{P}$ has a log structure with a chart $P_\mathscr{P}$. This is actually the divisorial log structure, by Remark~\ref{the divisorial log structure}. The tautological section $\varphi_\mathscr{P}: X_{\xi,\mathbf{R}}\to S^2U_\xi^*$, with the log structure on $(\Xi_\mathscr{P}, M)$ is compatible with the degeneration data $\tau, \psi$. By Theorem~\ref{olssen family} plus Grothendieck's existence theorem (\cite{EGA3} EGA \Rmnum{3}$_{1}$ 5.4.5), we have an algebraic family $\pi:(\mathcal{X}_\mathscr{P}, \mathcal{L}_\mathscr{P}, G_\mathscr{P},P_\mathscr{P},\varrho_\mathscr{P})\to (\widehat{\Xi}_\mathscr{P}, M_\mathscr{P})$. The polarization of the generic fiber is of type $\delta$. Since, over $\mathbf{Z}[1/d]$ the type of polarization is constant over a connected base, it suffices to check the type on a $\mathbf{C}$-point. By Appendix~\ref{B}, we know the degeneration data gives the polarization of type $\delta$. To simplify the notations and to follow the constructions in \cite{FC}, when $\bar{\sigma}=C(\mathscr{P})$ is a cone in $\Sigma(F_\xi)$, we also denote the formal base $\widehat{\Xi}_\mathscr{P}$ by $S_\sigma$, and the family by $\pi:(\mathcal{X}_\sigma,\mathcal{L}_\sigma,G_\sigma, P_\sigma,\varrho_\sigma)/ (S_\sigma, M_\sigma)$. We call this miniversal family a good formal $\sigma$-model.

\begin{remark}
By (\cite{FC} Chap. \Rmnum{1} Proposition 2.7), the semiabelian group scheme $G_\mathscr{P}$ is unique up to a unique isomorphism, and $\tau$ can be intrinsically defined by $G_\mathscr{P}$. Therefore the definition of $\tau$ does not depend on the choice of $S_\xi$, the $0$-cusp, and the matrix $S$. Moreover, the isomorphism class of the family $(\mathcal{X}_\mathscr{P}, G_\mathscr{P},P_\mathscr{P},\varrho_\mathscr{P})$ does not depend on the choices of $\psi,\mathcal{M}$. The choices of $\psi,\mathcal{M}$ are part of the data called a framing in (\cite{Alex02} Definition 5.3.7). It is possible that, if we start from different data, $\mathcal{L}_\mathscr{P}$ is changed to $\mathcal{L}_\mathscr{P}\otimes \pi^*\mathcal{N}$ for some invertible sheaf $\mathcal{N}$ over the base $\widehat{\Xi}_\mathscr{P}$ . We allow this in our definition of an isomorphism of families.
\end{remark}

If $C(\mathscr{P}')$ is a face of $C(\mathscr{P})$, the gluing map is $k[S_\mathscr{P}]\to k[S_{\mathscr{P}'}]$. By Corollary~\ref{the gluing is etale}, this morphism is \'{e}tale. Moreover, by Corollary~\ref{compatible sections}, the difference $f^\sharp\alpha\varphi_\mathscr{P}/\alpha'\varphi_{\mathscr{P}'}$ is invertible in $k[S_{\mathscr{P}'}]$ and thus invertible in $\Xi_\mathscr{P}$, and the families can be glued by Proposition~\ref{AN functorial}. There is another type of \'{e}tale pre-equivalence relation described by the finite group scheme $k[\mathbb{L}_\mathscr{P}/\mathbb{L}_\xi]$.

For each $\alpha\in B(\mathbf{Z}):=X/\phi(Y)$, choose $\vartheta_{\mathcal{A},\alpha}\in \H^0(\mathcal{A},\mathcal{M}_\alpha)$ such that $\vartheta_{\mathcal{A},\alpha}$ does not vanish along any fiber $A_s$, $s\in S_\sigma$. Define $\vartheta$ as the descent ($Y$-action) from
\[
\tilde{\vartheta}:=\sum_{\alpha\in B(\mathbf{Z})}\sum_{\lambda\in Y} S_\lambda^*\vartheta_{\mathcal{A},\alpha}.
\] 

Let $\Theta$ be the zero locus of $\vartheta$.  By Lemma~\ref{the stable pair is in AP}, the family $\pi:(\mathcal{X}_\sigma, \mathcal{L}_\sigma,G_\sigma,\Theta,\varrho_\sigma)\to S_\sigma$ is an object in $\overline{\mathscr{AP}}_{g,d}$. 

\begin{remark}
When we apply the isomorphisms from $\spec k[\mathbb{L}_\mathscr{P}/\mathbb{L}_\xi]$, each component $\sum_{\lambda\in Y} S_\lambda^*\vartheta_{\mathcal{A},\alpha}$ will be changed by a constant in $\mu_M$, and this constant depends on $\alpha$. Therefore $\Theta$ only exists \'{e}tale locally, and does not descend over the moduli stack. 
\end{remark}

If the cusp $F_\xi$ is a $0$-cusp, the abelian part $\mathcal{A}$ is trivial,
\[
\tilde{\vartheta}=\sum_{\alpha\in B(\mathbf{Z})}\sum_{\lambda\in Y}S_\lambda^*(\mathrm{X}^{\alpha, \varphi_\mathscr{P}(\alpha)}\theta).
\]

The choices of $\vartheta$ become finite. In general, the set of stable sections $\vartheta$ can be characterized intrinsically by making use of a representation of a group scheme $G(M)$. We will discuss this issue in Section~\ref{4}. 

The discrete group $\overline{P}(F_\xi)$ is a subgroup of $\GL(X_\xi,Y_\xi)$ of finite index, by Proposition~\ref{interpretation of the cone} d). The group $\GL(X_\xi,Y_\xi)$ is acting on both $U_\xi^*$ and $S^2U_\xi^*$. The sections $\varphi_\mathscr{P}$ are $\GL(X_\xi,Y_\xi)$-equivariant, and thus $\overline{P}(F_\xi)$-equivariant. That implies the construction is compatible with the action of $\overline{P}(F_\xi)$. Moreover, since $\GL(X_\xi,Y_\xi)$ preserves $X_\xi$ and $\phi(Y_\xi)$, it preserves the set $B(\mathbf{Z})$. It follows that the construction is $\overline{P}(F_\xi)$-equivariant.

\begin{remark}
Ideally, we would like to define the moduli space over $k=\mathbf{Z}[1/d]$. Since $\mathbf{Z}[1/d]\to \mathbf{Z}[1/d, \zeta_M]$ is faithfully flat, the base change for $\Xi_\mathscr{P}$ (resp. $\widehat{\Xi}_\mathscr{P}$) is also faithfully flat. It is possible to construct the miniversal families by descent. Or equivalently, we hope to construct the degeneration data $(c^t,c,\tau, \psi)$ over $\mathbf{Z}[1/d]$ from a faithfully flat descent. 
\end{remark} 

%%%%%%%%%%%%%%%%%%%%%%%%%%%%%%%%%%%%%%%%%%%%%%%%%%%%%%%%%%%%%%%%%%%%%%%%%%%%%%%%%%%%
%%%%%%%%%%%%%%%%%%%%%%%%%%%%%%%%%%%%%%%%%%%%%%%%%%%%%%%%%%%%%%%%%%%%%%%%%%%%%%%%%%%%

\subsubsection{Algebraization and Glues}
After we have obtained the miniversal families over the formal bases (complete rings), we need to extend them to \'{e}tale bases (rings of finite type over $k$). The \'{e}tale neighborhoods of the boundary in $\Xi_\mathscr{P}$ serve as the \'{e}tale bases. 

We follow the procedure in \cite{FC}. Let $\sigma$ denote the rational cone $C(\mathscr{P})$ associated with a bounded paving $\mathscr{P}$. Since $\Xi_\mathscr{P}$ is a fibered by the toric variety $U_\sigma$, $\Xi_\mathscr{P}$ has the natural stratification induced by the toric stratification of $U_\sigma$. Then we can define the \'{e}tale constructible sheaf $\underline{X}_\xi$ (resp. $\underline{Y}_\xi$). If $\tau$ is a face of $\sigma$, the elements in $\tau$ are quadratic forms over $X_{\tau,\mathbf{R}}$, and $X_\tau$ (resp. $Y_\tau$) is the quotient of $X_\sigma$ (resp. $Y_\sigma$). Then, over the $\tau$-stratum, define the sheaf $\underline{X}_\xi$ (resp. $\underline{Y}_\xi$) to be the constant sheaf $X_\tau$ (resp. $Y_\tau$). Moreover, over each $\tau$-stratum, we have the tautological bilinear pairing $B: Y_\tau\times X_\tau\to \mathbb{L}_\tau$. The elements in $\mathbb{L}_\tau$ are sections of $\mathcal{K}^*/\mathcal{O}^*$ over the toric variety $U_\tau\subset U_\sigma$. Therefore, we get a pairing $\underline{B}_\xi: \underline{Y}_\xi\times \underline{X}_\xi\to \underline{\Div} \Xi_\mathscr{P}$ from the toric data.

If $(\mathcal{X}_\sigma, \mathcal{L}_\sigma, G_\sigma, P_\sigma, \varrho_\sigma)$ over $(S_\sigma, M_\sigma)$ is a good formal $\sigma$-model, we can forget about the other data, and get a good formal $\sigma$-model $(G_\sigma, \lambda)$ over $S_\sigma$, with a general polarization $\lambda$ (\cite{FC} Chap. \Rmnum{4} Definition 3.2 \& Proposition 3.3 (i)). Notice that all the results in (\cite{FC} Chap. \Rmnum{3} Sect. 9 \& 10) are proved for general polarizations, therefore (loc. cit. Chap. \Rmnum{4} Definition 3.2 \& Proposition 3.3 (i)) can be naturally generalized for the general polarizations. Since $(S_\sigma, M_\sigma)$ is the divisorial log structure for the toroidal boundary $\partial S_\sigma$, let the log differential $\Omega^1_{(S_\sigma, M_\sigma)}$ be denoted by $\Omega^1_{S_\sigma}[\ud\log \infty]$, the differential with log poles along the boundary. In particular for good formal $\sigma$-model, we have 

\begin{proposition}\label{properties of a good formal model}
Suppose $(\mathcal{X}_\sigma, \mathcal{L}_\sigma, G_\sigma, P_\sigma, \varrho_\sigma)$ over $(S_\sigma, M_\sigma)$ is a good formal $\sigma$-model.
\begin{enumerate}
\item The sheaves and the pairing $(\underline{B}, \underline{X},\underline{Y})$ obtained from the polarized semiabelian scheme $(G_\sigma,\lambda)$ agree with the sheaves and the pairing $(\underline{B}_\xi,\underline{X}_\xi,\underline{Y}_\xi)$ obtained from the toric variety fibers. 
\item Let $\Omega=\underline{\Omega}(G_\sigma/S_\sigma)$ be the dual of the relative invariant Lie algebra of $G_\sigma$ and $\Omega^t=\underline{\Omega}(G_\sigma^t/S_\sigma)$ be the dual of the relative invariant Lie algebra of $G_\sigma^t$. The Kodaira--Spencer map induces an isomorphism $S^2\Omega^t\cong \Omega^1_{S_\sigma}[\ud \log \infty]$. 
\end{enumerate}
\end{proposition}

Notice that our formal base is the same as that in \cite{FC}, we can do the same approximation to the rings of finite type over $\Xi_\mathscr{P}$.  The point is the two properties in Proposition~\ref{properties of a good formal model} are preserved in the process of approximation. Therefore we have (loc. cit. \Rmnum{4} Proposition 4.3, 4.4) for our case.

\begin{proposition}\label{algebraic family}
Let $\sigma=C(\mathscr{P})$, $R$ be the strict local ring of a geometric point $\overline{x}$ of the $\sigma$-stratum of $\Xi_\mathscr{P}$, $I$ be the ideal defining the $\sigma$-stratum, and $\widehat{R}$ be the $I$-adic completion of $R$. There exists an \'{e}tale neighborhood $S'=\spec R'$ of $\overline{x}$, a family $(\mathcal{X}, \mathcal{L}, P, G,\varrho)$ over $(S', M')$, and an embedding $R'\to \widehat{R}$ close to the canonical inclusion, such that
\begin{enumerate}
\item The family $(\mathcal{X}, \mathcal{L}, P, G,\varrho)$ is an object in $\overline{\mathscr{T}}_{g,d}(S')$. The log structure is the divisorial log structure.
\item The map $\spf \widehat{R}\to \widehat{\Xi}_\mathscr{P}$ coincides on the $\sigma$-stratum with the map induced by the inclusion $R'\to \widehat{R}$.
\item Over $R/I$, $(\mathcal{X}, \mathcal{L}, P, G,\varrho)$ is isomorphic to the pull back of the good formal $\sigma$-model $(\mathcal{X}_\sigma, \mathcal{L}_\sigma, P_\sigma, G_\sigma, \varrho_\sigma)$.
\item  The \'{e}tale sheaves and the pairing $(\underline{B}, \underline{X},\underline{Y})$ obtained from the family $G$ coincide with the pull backs of the \'{e}tale sheaves and the pairing $(\underline{B}_\xi,\underline{X}_\xi,\underline{Y}_\xi)$ over $\Xi_\mathscr{P}$.
\item The Kodaira-Spencer map induces an isomorphism $S^2\Omega^t\cong \Omega^1_{R'}[\ud \log \infty]$. 
\end{enumerate}
\end{proposition} 
\begin{proof}
For (2)--(5), the same proof as that of (\cite{FC} Chap. \Rmnum{4} Proposition 4.4). The statement (1) follows from (3) and (\cite{ols08} Theorem 5.9.1). 
\end{proof}

We call the family obtained above a good algebraic $\sigma$-model. For any geometric point $\bar{x}$ over $x\in S$ in the $\tau$-stratum of a good algebraic $\sigma$-model over $S$, denote the strict henselization by $\widehat{R}_{\bar{x}}$. The polarized family $G$ over $\widehat{R}_{\bar{x}}$ gives the degeneration data. By Proposition~\ref{algebraic family} (4), we know $\underline{B}$ agrees with $\underline{B}_\xi\vert_\tau$. By the universal property described in (at the end of the third paragraph \cite{FC} p. 106), this defines a morphism $f:\spf \widehat{R}_{\bar{x}}\to \Xi_\tau$. The family $G$ over $\widehat{R}_{\bar{x}}$ is a good formal $\tau$-model. The boundary divisors agree under $f$. It follows that $f$ is strict for the log structures. Moreover, by Proposition~\ref{algebraic family} (5), the log differentials $\Omega^1_{(S, M)}\cong f^*\Omega^1_{(\Xi_\tau, M_\tau)}$. Therefore, $f$ is log \'{e}tale. Openness of versality follows from the following lemma. 
\begin{lemma}
If a log morphism $f: (X,M_X)\to (Y, M_Y)$ is strict, then $f$ being formally log \'{e}tale implies that the underlying morphism $f$ between schemes is formally \'{e}tale.
\end{lemma}
\begin{proof}
Consider the affine thickening diagram
\[
\begin{diagram}
T_0&\rTo^{a^0}& X\\
\dInto^j&&\dTo_f\\
T&\rTo^a&Y
\end{diagram},
\]

where $T$ is affine, and $j$ is a closed immersion defined by a nilpotent ideal. Provide $T$ (resp. $T_0$) the log structure $a^*M_Y$ (resp. $a_0^*M_X$), the diagram is still commutative. Because $f$ is strict, $j$ is strict. Since $f$ is formally log \'{e}tale, there is a unique lift $b: T\to X$. Thus $f$ is formally \'{e}tale.
\end{proof}

The good algebraic models are the \'{e}tale neighborhoods of the compactification. Define $U$ to be the disjoint union of finitely many good algebraic models that cover all strata $\mathfrak{Z}_\mathscr{P}$ up to the action of $\Gamma(\delta)$. This is possible because there are only finitely map $0$-cusps up to $\Gamma(\delta)$-action, and for each $0$-cusp $F$, the number of $\overline{P}(F)$-orbits of cones in $\Sigma(F)$ is finite. The cone $\{0\}$ corresponds to the moduli space $\mathscr{A}_{g,\delta}$. However, by (\cite{ols08} Proposition 5.1.4), the universal family is interpretated as the family $(\mathcal{X},\mathcal{L},\varrho,G)$ with $G$ abelian. Let $U_0$ denote the dense open stratum of $U$ where $G$ is abelian. Let $R_0$ denote the \'{e}tale relation $R_0:=U_0\times_{\mathscr{A}_{g,\delta}} U_0$. Define $R$ to be the normalization of the image $R_0\to U\times U$, i.e. normalization of the closure $Z$ of the image of $R_0$ in $U\times U$ with respect to the finite extensions of the function fields at the maximal points of $Z$. We can extend the isomorphisms by an explicit construction. To save the work, we use the following analog of (\cite{FC} Chap. \Rmnum{1} Proposition 2.7). 

\begin{proposition}\label{proper diagonal}
Let $S$ be a Noetherian normal scheme, and $\pi_i:(\mathcal{X}_{i}, P_i, \mathcal{L}_i,G_i, \varrho_i)$ for $i=1,2$ be two good algebraic models over $S$. Suppose that over a dense open subscheme $U$ of $S$, the log structures $P_i$ are trivial, and there is an isomorphism $f_U$ between $(\mathcal{X}_i, P_i, \mathcal{L}_i,G_i, \varrho_i)\vert_U$. Then $f_U$ extends to a unique isomorphism $f: (\mathcal{X}_1, P_1,G_1, \varrho_1)\to (\mathcal{X}_2, P_2,G_2, \varrho_2)$ and there exists a line bundle $\mathcal{M}$ over $S$,  such that $\mathcal{M}\vert_U$ is trivial, and $\mathcal{L}_1\otimes \pi_1^*\mathcal{M}\cong f^*\mathcal{L}_2$. 
\end{proposition}

\begin{proof}
For the extension problem of morphisms between data $(\mathcal{X}_i,G_i, \varrho_i)$, by a standard reduction process (\cite{FC} Chap. \Rmnum{1} Proof of Proposition 2.7 Step (a)), it suffices to consider the case when $S$ is the spectrum of a discrete valuation ring. This case is proved by (\cite{ols08} Proposition 5.11.6). Let the unique extension be $f$. From this reduction to discrete valuation rings, we also know that $f^*\mathcal{L}_2$ and $\mathcal{L}_1$ are isomorphic on each fiber. Since $\mathcal{X}_1$ is projective over $S$ and the fibers are all integral,  by (\cite{Har77} Chap. \Rmnum{3} Exercise 12.4), there exists $\mathcal{M}$ over $S$ such that $\mathcal{L}_1\otimes \pi_1^*\mathcal{M}\cong f^*\mathcal{L}_2$. Finally, by (\cite{ols08} Proposition 5.10.2), $f$ preserves the log structures. 
\end{proof}

By Proposition~\ref{proper diagonal}, the isomorphism between the universal families over $R_0$ is extended to the isomorphism between the good algebraic models over $R$\footnote{Again we say $f$ preserves the line bundles if $\mathcal{L}_1\otimes \pi_1^*\mathcal{M}\cong f^*\mathcal{L}_2$.}. Since morphisms in $\overline{\mathscr{A}}_{g,d}$ is defined up to an action of $\mathbb{G}_m$, we have 

\begin{corollary}\label{faithful for abelian}
Let $R':=U\times_{\overline{\mathscr{A}}_{g,d}}U$. We have defined a morphism $R\to R'$, and this map is injective. 
\end{corollary}

Denote the two projections $R\to U$ by $s$ and $t$. Notice that the definitions of $U$ and $R$ only involve the collection of fans $\widetilde{\Sigma}$ and the algebraization of the bases. Therefore we can use the same $U$ and $R$ as those in (\cite{FC} Chap. \Rmnum{4}, Sect. 5.). We simply replace the semiabelian schemes over $U$ and $R$ by AN families.

\begin{proposition}[\cite{FC} Chap.\Rmnum{4} Lemma 5.2]\label{the universal property of the moduli space}
Suppose $R$ is a normal complete local ring which is strictly henselian, $K$ its field of fractions and $\kappa$ its residue field. Assume furthermore that $G^\dag$ is a semiabelian scheme over $R$ whose generic fiber $G_K^\dag$ is an abelian variety with a polarization $\lambda_K$ of type $\delta$. Associated with this we have the character group $X_\xi$ (resp. $Y_\xi$) of the torus part of the special fiber $G_s$ (resp. $G_s^t$), a polarization $\phi: Y_\xi\to X_\xi$, and a bimultiplicative form $b$ on $Y_\xi\times X_\xi\to K^*$, defined up to units in $R$, such that $b(Y_\xi,\phi(Y_\xi))$ is symmetric. The following statements are equivalent
\begin{enumerate}
\item Write $X_\xi$ as the quotient lattice of $X=\mathbf{Z}^g$. If $v:K^*\to \mathbf{Z}$ is any discrete valuation defined by a prime ideal of $R$ of height one, $v\circ b$ defines a positive semi-definite symmetric bilinear form on $X$. There exists a closed cone $C(\mathscr{P})$ in the second Voronoi fan $\Sigma(X)$ which contains all $v\circ b$ obtained this way.
\item $(G^\dag,\lambda)$ is the semiabelian scheme of the pull-back of a good formal $\sigma$-model $(\mathcal{X}_\sigma, \mathcal{L}_\sigma, G_\sigma, \varrho_\sigma)/S_\sigma$ via a morphism $\spec R\to \Xi_\sigma$ (equivalently, a map $\spf R\to S_\sigma$) for some (open) cone $\sigma$ in $\Sigma(X)$. 
\end{enumerate}
The cone $C(\mathscr{P})$ in (1) and the cone $\sigma$ in (2) are in the same $\GL(X,Y)$-orbits. 
\end{proposition}

\begin{theorem}\label{compactification}
The morphism $(s,t): R\to U\times U$ defines an \'{e}tale groupoid in the category of $k$-schemes. Let $\overline{\mathscr{A}}_{g,\delta}^m$ be the stack $[U/R]$ over $k$. It is a proper Deligne-Mumford stack, containing $\mathscr{A}_{g,\delta}$ as an open dense substack. It admits a coarse moduli space.  Over $\mathbf{C}$, the coarse moduli space is the toroidal compactification $\overline{\mathcal{A}}_{\Sigma}$. 
\end{theorem}

\begin{proof}
Notice that none of the arguments in (\cite{FC} Chap. \Rmnum{4} Proposition 5.4, Corollary 5.5 \& Theorem 5.7 (1)) uses the principal polarization. By (\cite{ols08} Theorem 1.4.2), the coarse moduli space exists. By the construction, the coarse moduli space over $\mathbf{C}$ is $\overline{\mathcal{A}}_{\Sigma}$. 
\end{proof}

From now on, let $\overline{\mathcal{A}}_{\Sigma}$ also denote the coarse moduli space of $\overline{\mathscr{A}}_{g,\delta}^m$ over the ring $k$. 

\begin{proposition}\label{log smooth}
The algebraic stack $\overline{\mathscr{A}}_{g,\delta}^m$ has a log structure $M_\Sigma$ such that $(\overline{\mathscr{A}}_{g,\delta}^m, M_\Sigma)$ is log smooth over $k=\mathbf{Z}[1/d,\zeta_M]$. 
\end{proposition}
\begin{proof}
It follows from the description of a good algebraic model in Proposition~\ref{algebraic family}. 
\end{proof}

\begin{lemma}\label{normal}
The algebraic space $\overline{\mathcal{A}}_{\Sigma}$ is normal.
\end{lemma}
\begin{proof}
Recall that the local chart for the algebraic stack $\overline{\mathscr{A}}_{g,\delta}^m$ is a toric monoid. By (\cite{ogus} Chap. \Rmnum{1} Proposition 3.3.1.), $\overline{\mathscr{A}}_{g,\delta}^m$ is normal. Since $\overline{\mathscr{A}}_{g,\delta}^m$ is a separated Deligne-Mumford stack, \'{e}tale locally, it is presented as $[U/G]$ for $U$ a scheme \'{e}tale over the stack, and $G$ a finite group (\cite{AV02} Lemma 2.2.3). \'{E}tale locally, the coarse moduli space is the coarse moduli space $[U/G]\to U/G$. Now $U$ is normal, and $U/G$ is a quotient by a finite group, thus is also normal. 
\end{proof}

Since $d$ is invertible in $k$, $\mathscr{A}_{g,d}$ is a disjoint union of $\mathscr{A}_{g,\delta}$ for all polarization type $\delta$ of degree $d$. Correspondingly, $\overline{\mathscr{A}}_{g,d}[1/d]$ is a disjoint union of connected components denoted by $\overline{\mathscr{A}}_{g,\delta}[1/d]$ (\cite{ols08} 6.2.1). Moreover, $\overline{\mathscr{A}}_{g,\delta}[1/d]$ is log smooth in this case. The coarse moduli space $\overline{\mathcal{A}}_{g,d}$ is a disjoint union of connected components $\overline{\mathcal{A}}_{g,\delta}[1/d]$. 

\begin{proposition}\label{relation between two compactifications}
Over $k=\mathbf{Z}[1/d, \zeta_M]$, the morphism $\overline{F}: \overline{\mathscr{A}}_{g,\delta}^m\to \overline{\mathscr{A}}_{g,\delta}[1/d]$ induced by the AN family is proper, surjective, and representable. The coarse moduli space $\overline{\mathcal{A}}_\Sigma$ is the normalization of $\overline{\mathcal{A}}_{g,\delta}[1/d]$ over $k$. In particular, over a field of characteristic zero, $\overline{F}$ is an isomorphism.
\end{proposition}

\begin{proof}
There is a morphism $U\to \overline{\mathscr{A}}_{g,\delta}[1/d]$ induced from Proposition~\ref{algebraic family} (1), and an injection $R\to R'$. This induces a morphism between stacks $[U/R]\to [U/R']$. By (\cite{LMB} Proposition (3.8)), $[U/R']$ is a substack of $\overline{\mathscr{A}}_{g,\delta}[1/d]$. The composition is a morphism $\overline{F}: \overline{\mathscr{A}}_{g,\delta}^m\to \overline{\mathscr{A}}_{g,\delta}[1/d]$. Since both $\overline{\mathscr{A}}_{g,\delta}^m$ and $\overline{\mathscr{A}}_{g,\delta}[1/d]$ are proper, by (\cite{ols13} Proposition 10.1.4 (iv)), $\overline{F}$ is proper in the sense of (\cite{ols13} Definition 10.1.3). The morphism $\overline{F}$ is identity on the dense open substack $\mathscr{A}_{g,\delta}$ of $\overline{\mathscr{A}}_{g,\delta}^m$. Then it is surjective onto the closure $\overline{\mathscr{A}}_{g,\delta}[1/d]$.

By Lemma~\ref{injection implies representable} and Corollary~\ref{faithful for abelian}, $\overline{F}$ is representable. In particular, by (\cite{ols13} Proposition 10.1.2), $\overline{F}$ is proper as a representable morphism. 

Let the morphism between the coarse moduli spaces be denoted by $f: \overline{\mathcal{A}}_{\Sigma}\to \overline{\mathcal{A}}_{g,\delta}[1/d]$. As a morphism between proper spaces, $f$ is proper. Moreover, by Lemma~\ref{finite fiber}, $f$ is quasi-finite, and thus finite (\cite{EGA4} 8.11.1). Now $f$ is also birational, and $\overline{\mathcal{A}}_\Sigma$ is normal by Lemma~\ref{normal}, hence $f$ is the normalization of $\overline{\mathcal{A}}_{g,\delta}[1/d]$. Here the normalization of an algebraic space is defined in (\cite{KMc} Appendix N.3).

Over a field of characteristic zero, $\overline{\mathscr{A}}_{g,\delta}[1/d]$ is Deligne--Mumford. Therefore, we have an \'{e}tale chart $U'$ over $\overline{\mathscr{A}}_{g,\delta}[1/d]$ which is normal. Then the coarse moduli space $\overline{\mathcal{A}}_{g,\delta}[1/d]$, locally a finite group quotient of $U'$, is also normal. Therefore, $f: \overline{\mathcal{A}}_\Sigma\to\overline{\mathcal{A}}_{g,\delta}[1/d]$ is an isomorphism. On the other hand, the pullback of $U'$ along $\overline{F}$ is an \'{e}tale chart $U$ for $\overline{\mathscr{A}}^m_{g,\delta}$. Since $U\to U'$ is the normalization, we also have $U'\cong U$. Use (\cite{AV02} Lemma 2.2.3) and the fact that the coarse moduli spaces commute with \'{e}tale base change, we get the following \'{e}tale local picture. $\overline{F}: [U/G]\to [U/G']$ for finite groups $G$ and $G'$ , such that 
\begin{itemize}
\item[1)] The coarse moduli spaces $U/G\to U/G'$ are isomorphic, 
\item[2)] $G\to G'$ is injective, 
\item[3)] There is a dense open subscheme $V\subset U/G$ over which $\overline{F}$ is isomorphic. 
 \end{itemize}
 We claim that $G\to G'$ is also surjective. By 1), for any point $x\in U$ and its stabilizer $G'_x$ , we have $G\cdot G'_x=G'$. If $G\to G'$ were not surjective, there would exist $x$ over $V$ such that $G'_x\neq G_x$. This contradicts with 3). Therefore $\overline{F}$ is an isomorphism. 
\end{proof}

\begin{lemma}\label{injection implies representable}
If $\overline{F}$ induces an injection $R\to R'$, then it is representable between algebraic stacks. 
\end{lemma}
\begin{proof}
We claim that the functor $\overline{F}$ is faithful. By the property of categories fibered in groupoids (\cite{ols13} Lemma 3.1.8), it suffices to prove that for any $S\in (\text{Sch}/k)$, and $u,u'\in \mathscr{T}_Q(S)$, the map $\overline{F}_S: \text{Isom}_{\mathscr{T}_Q(S)}(u,u')\to \text{Isom}_{\mathscr{K}_Q(S)}(\overline{F}(u),\overline{F}(u'))$ is injective. Since $\text{Isom}_{\mathscr{T}_Q(S)}(u,u')$ is a torsor over $\Aut_{\mathscr{T}_Q(S)}(u)$, if it is not empty, it can be reduced to the case $u=u'$. These are global sections of presheaves $\Isom(u,u)$ and $\Isom(\overline{F}(u),\overline{F}(u))$. Since both $\mathscr{T}_Q$ and $\mathscr{K}_Q$ are algebraic stacks, both of the presheaves are sheaves for the \'{e}tale topology on $(\text{Sch}/S)$. So we can check the injectivity on some \'{e}tale covering of $S$. Let $\xi:S'\to S$ be the pullback of $U\to \mathscr{T}_Q$ along $u: S\to \mathscr{T}_Q$.
\begin{diagram}
S'&\rTo^v &U\\
\dTo^\xi & &\dTo\\
S&\rTo^u&\mathscr{T}_Q.
\end{diagram}

By definition, there is an isomorphism $\rho: \xi^*u\to v$ in $\mathscr{T}_Q(S')$. Using $\rho$, replace $\xi^*u$ by $v$. Consider the Cartesian diagrams
\begin{equation*}
\begin{diagram}
\Isom(v,v)&\rTo&S'\\
\dTo &&\dTo \\
\mathscr{T}_Q&\rTo^{\Delta}&\mathscr{T}_Q\times\mathscr{T}_Q,
\end{diagram}
\quad\quad\quad
\begin{diagram}
R&\rTo^{(s,t)} &U\times U\\
\dTo &&\dTo\\
\mathscr{T}_Q &\rTo^{\Delta}&\mathscr{T}_Q\times\mathscr{T}_Q.
\end{diagram}
\end{equation*}

By the universal property of $R$, there is a $1$-morphism $H: \Isom(v,v)\to R$ such that the following diagram commutes,
\begin{diagram}
\Isom(v,v) &\rTo &S'\\
\dTo^H&&\dTo_{(v,v)}\\
R&\rTo^{(s,t)}& U\times U.
\end{diagram}

This diagram is Cartesian. Similarly, after applying $\overline{F}$, we have the Cartesian diagram
\begin{diagram}
\Isom(\overline{F}(v),\overline{F}(v)) &\rTo &S'\\
\dTo^{\overline{F}(H)}&&\dTo_{(v,v)}\\
R'&\rTo^{(s,t)}& U\times U.
\end{diagram}

Since the composition of two pullbacks is a pullback, the diagram
\begin{diagram}
\Isom(v,v)&\rTo&\Isom(\overline{F}(v),\overline{F}(v))\\
\dTo^{H}&&\dTo^{\overline{F}(H)}\\
R&\rTo &R'
\end{diagram}

is Cartesian. Then $R\to R'$ is injective implies that the map $\overline{F}_{S'}:\Aut_{\mathscr{T}_Q(S')}(v)\to \Aut_{\mathscr{K}_Q(S')}(\overline{F}(v))$ is injective. The claim that $\overline{F}$ is faithful is thus proved.

By (\cite{stackproject} Lemma 67.15.2), $\overline{F}$ is representable.  
\end{proof}

\begin{lemma}\label{finite fiber}
The morphism $f: \overline{\mathcal{A}}_{\Sigma}\to \overline{\mathcal{A}}_{g,\delta}[1/d]$ is quasi-finite. 
\end{lemma}
\begin{proof}
It suffices to check that for any algebraic closed field $\kappa(\bar{x})$ and point $\bar{x}: \spec \kappa(\bar{x})\to \overline{\mathscr{A}}_{g,d}$, the fiber $\overline{F}^{-1}(\bar{x})$ is discrete, since $f$ is of finite type and thus quasi-compact. So we fix such a point $\bar{x}$ on the boundary of $\overline{\mathscr{A}}_{g,d}$. Forget about the log structure, it represents a polarized stable semiabelic variety $(X, G, \mathcal{L},\varrho)$ over $\kappa(\bar{x})$. By (\cite{ols08} 5.3.4), $(X,G,\mathcal{L},\varrho)$ determines the decomposition $\mathscr{P}$, the subgroup $\phi: Y\to X$, the maps $c: X\to A$ and $c^t: Y\to A^t$. Since the $\mathscr{P}$-stratum in $\overline{\mathcal{A}}_\Sigma$ is a quotient of the interior $\mathfrak{Z}_\mathscr{P}^\circ$ by a discrete group, we can consider the inverse image in $\mathfrak{Z}_\mathscr{P}^\circ$, and still denote it by $\overline{F}^{-1}(\overline{x})$. More precisely, $\overline{F}^{-1}(\overline{x})$ is contained in the fiber over $(A,A^t,\phi,\lambda, c,c^t)\in \overline{F}_\xi\ltimes\overline{\mathcal{V}}_\xi$. Let's denote the fiber by $V(\mathscr{P})$. It is a torus and is isomorphic to $\H^0(\Delta_\mathscr{P}, \widehat{\underline{\mathbb{L}}})$ in \cite{Alex02}. The claim is that there is no positive dimensional sub locus of $V(\mathscr{P})$ where the fibers are all isomorphic as polarized stable semiabelic varieties (SSAV).

Recall the framing for a polarized SSAV in (\cite{Alex02} Definition 5.3.6). By (\cite{Alex02} Theorem 5.3.8), the framed polarized SSAV are classified by the groupoid $M^{\text{fr}}[\Delta_{\mathscr{P}},c,c^t,\mathcal{M}](\kappa(\overline{x}))$ which is equivalent to $[Z^1(\Delta_\mathscr{P},\widehat{\underline{\mathbb{X}}})/C^0(\Delta_\mathscr{P},\widehat{\underline{\mathbb{X}}})]$. Fix the framing, an element in $Z^1(\Delta_\mathscr{P},\widehat{\underline{\mathbb{X}}})$ is the gluing data denoted by $\psi_0\tau_0$. Fix $(A, A^t, \phi, \lambda, c,c^t)$, we can construct the unframed groupoid $M[\Delta_{\mathscr{P}}, c, c^t](\kappa(\overline{x}))$, i.e. the groupoid of polarized SSAV as in (loc. cit. 5.4.4). Let $\mathrm{Pic}^\lambda A$ be the component of $\mathrm{Pic}\, A$ of polarization $\lambda$. Over $\mathrm{Pic}^\lambda A$, take the family $Z^1(\Delta_\mathscr{P},\widehat{\underline{\mathbb{X}}})$, and denote the union by $M^{\text{fr}}$. Then quotient out the choice of the framing. We won't write down all the equivalent relations, because it suffices to only consider the choices that are not discrete. There are two type of non-discrete choices we have to quotient out. One is the choice of the projection to $A$, i.e. a choice of a point in the minimal $A$-orbit. See (loc. cit. Definition 5.3.7 (4)). The other is the action of $C^0(\Delta_\mathscr{P},\widehat{\underline{\mathbb{X}}})$. 

In our construction, we have chosen $\mathcal{M}$ and $\psi$ \'{e}tale locally. Therefore, there is a natural framing for our family over $V(\mathscr{P})$. Since we construct the base $\Xi_\xi$ as the moduli space for the degeneration data $\tau$. It defines a finite map \footnote{Here we mean every fiber of the map is finite.} from $V(\mathscr{P})$ to the fiber $Z^1(\Delta_\mathscr{P},\widehat{\underline{\mathbb{X}}})$ over $\mathcal{M}\in \mathrm{Pic}^\lambda A$. Consider the image of this map as a subset in $M^{\text{fr}}$ and denote it by $\mathfrak{Z}$. It suffices to show that the intersection of $\mathfrak{Z}$ with the group orbits is discrete.

First, consider the group orbits generated by changing the projections to $A$. If the projection to $A$ is changed by an element $a\in A$, then $\mathcal{M}$ is changed to $T_{-a}^*\mathcal{M}$. However, $\mathcal{M}$ is ample, there are only finitely many $a\in A$ such that $\mathcal{M}\cong T_{-a}^*\mathcal{M}$. Therefore, this group orbit intersects the fiber $Z^1(\Delta_\mathscr{P},\widehat{\underline{\mathbb{X}}})$ over $\mathcal{M}\in \mathrm{Pic}^\lambda A$ at finitely many points, thus intersects $\mathfrak{Z}$ at only finitely many points. 

Secondly, the action of $C^0(\Delta_\mathscr{P},\widehat{\underline{\mathbb{X}}})$ on the fiber $Z^1(\Delta_\mathscr{P},\widehat{\underline{\mathbb{X}}})$ over $\mathcal{M}\in \mathrm{Pic}^\lambda A$. For any $t\in C^0(\Delta_\mathscr{P},\hat{\underline{\mathbb{X}}})$, it maps $\psi_0$ to $\psi_0/t$\footnote{Use the notations in (loc. cit. Definition 5.3.4), $\psi_0$ is a function over $C_0(\Delta_\mathscr{P},\underline{\mathbb{X}})$.} and leaves $\tau_0$ unchanged. However, there is no positive dimensional locus where $\tau_0$ is constant in $V(\mathscr{P})$. Therefore, the intersection is also discrete. 

In sum, there is no positive dimensional locus in $V(\mathscr{P})$ that is mapped to a point in the groupoid of polarized SSAV, and our claim is proved.
\end{proof}

\begin{remark}
Although the algebraic stack $\overline{\mathscr{A}}_{g,d}[1/d]$ is normal over $\mathbf{Z}[1/d]$, we are unable to prove that its coarse moduli space is also normal. The problem is we don't know if $\overline{\mathscr{A}}_{g,d}$ is Deligne--Mumford or not. Therefore, locally, we can only choose a finite flat groupoid representation $R\to U\times U$, instead of an \'{e}tale representation. Then we don't know how to choose $U$ to be normal. 
\end{remark}

%%%%%%%%%%%%%%%%%%%%%%%%%%%%%%%%%%%%%%%%%%%%%%%%%%%%%%%%%%%%%%
%%%%%%         Representation of Heisenberg groups and stable pairs     %%%%%%%%%%%%%%%%%%%%%%%%%
%%%%%%%%%%%%%%%%%%%%%%%%%%%%%%%%%%%%%%%%%%%%%%%%%%%%%%%%%%%%%%

\section{Representations of the Theta Group and the Stable Pairs}\label{4}
\subsection{The Heisenberg Relation and the Fourier Decomposition}\label{The Heisenberg Relation and the Fourier Decomposition}
Let's study a single abelian variety over $k=\mathbf{C}$ first. Fix a polarized abelian variety $(G,\lambda)$ over $\mathbf{C}$. Suppose the polarization $\lambda: G\to G^t$ is induced by an ample line bundle $\mathcal{L}$. Let $K(\lambda)$ be the kernel of $\lambda$. The automorphism group of the pair $(G,\mathcal{L})$ over the translations of $G$ is the theta group $\mathcal{G}(\mathcal{L})$. $\mathcal{G}(\mathcal{L})$ is a central extension of $K(\lambda)$ by $\mathbf{C}^*$ 
\begin{diagram}
1&\rTo&\mathbf{C}^*&\rTo &\mathcal{G}(\mathcal{L}) & \rTo^p &K(\lambda)& \rTo & 0. 
\end{diagram}

The space of global sections $\Gamma(G,\mathcal{L})$ is an irreducible $\mathcal{G}(\mathcal{L})$-module. We refer the reader to \cite{Mum66} for the basic facts about the theta group.

Write $G$ as $V/\Lambda$, where $V$ is the universal cover of $G$ and $\Lambda=\H_1(G,\mathbf{Z})$. The polarization is represented by an integral skew symmetric form $E$ over $\Lambda$. Define $\Lambda^\vee:=\{v\in V; E(v,v')\in \mathbf{Z}, \forall v'\in\Lambda\}$. $K(\lambda)=\Lambda^\vee/\Lambda$. 

\begin{definition}
Define $\mathcal{H}(V,E)$ to be a central extension of $V$ by $\mathbf{C}^*$,
\[
\begin{diagram}
1 &\rTo&\mathbf{C}^*&\rTo&\mathcal{H}(V,E)&\rTo^p&V&\rTo&0.
\end{diagram}
\]

that induces the pairing $\exp (-2\pi iE)$ on $V$. In other words, it is the Heisenberg group with the commutator  
\begin{equation}\label{commutator in Heisenberg group}
(t,v)(t',v')(t,v)^{-1}(t',v')^{-1}=\big(\exp(-2\pi i E(v,v')),0 \big),
\end{equation}
\end{definition}

Fix a Lagrangian $U\subset V$. Denote $U\cap \Lambda$ by $X^*$ and $\Lambda/X^*$ by $Y$. Choose a classical factor of automorphy $e_{\mathcal{L}}\in Z^1(\Lambda, \Gamma(V, \mathcal{O}_V^*))$ that is trivial over $X^*$ (See \cite{BL} p. 50 for notations.)
\[
e_{\mathcal{L}}(\lambda, v)=\chi(\lambda)\exp\bigg(\pi (H-B)(v,\lambda)+\frac{\pi}{2}(H-B)(\lambda,\lambda)\bigg), \quad \lambda\in \Lambda, v\in V. 
\]

The semi-character $\chi$ defines a lift $\Lambda\to \mathcal{H}(V,E)$ via $\lambda\mapsto(\chi^{-1}(\lambda),\lambda)$. From now on, we regard $\Lambda$ as a subgroup of $\mathcal{H}(V,E)$. Notice that $X^*$ is lifted to $(1,X^*)$.  Take the normalizer $N(\Lambda)$ of $\Lambda$ and the normalizer $N(X^*)$ of $X^*$ in $\mathcal{H}(V,E)$. 

Consider $h:\Lambda\times V\to \mathbf{C}^*$ defined by
\[
h(\lambda, v):=\exp\bigg(\pi (H-B)(v,\lambda)+\frac{\pi}{2}(H-B)(\lambda,\lambda)\bigg),\quad \lambda\in \Lambda, v\in V.
\]

Extend $h$ to $V\times V\to \mathbf{C}^*$. Then $h$ satisfies
\begin{equation}\label{h relation}
h(v_1+v_2,v')e^{\pi i E(v_1,v_2)}=h(v_1,v'+v_2)h(v_2,v').
\end{equation}

Define the action of $\mathcal{H}(V,E)$ on the trivial bundle $\mathbf{C}\times V$ 
\begin{equation}\label{action of Heisenberg group}
S_{(t,v)}(t', v')=(t^{-1}t' h(v,v'),v+v').
\end{equation}

This is a group representation of $\mathcal{H}(V,E)$ by Relation~\eqref{h relation}. Since $\Lambda$ (resp. $X^*$) is in the center of $N(\Lambda)$ (resp. $N(X^*)$), the action $S$ induces an action of $N(\Lambda)/\Lambda$ (resp. $N(X^*)/X^*$) on the quotient $(G,\mathcal{L})=(V, \mathbf{C}\times V)/\Lambda$ (resp. $(T, \mathbf{C}\times T)=(V, \mathbf{C}\times V)/X^*$). Notice that if $\mathcal{L}$ is considered as the coherent sheaf of sections instead of the line bundle, then the action of the center $\mathbf{C}^*\subset N(\Lambda)/\Lambda$ is of weight $1$. Therefore, the action of $N(\Lambda)/\Lambda$ on $\Gamma(G,\mathcal{L})$ is identified with the action of $\mathcal{G}(\mathcal{L})$ on $\Gamma(G,\mathcal{L})$. On the other hand, $N(X^*)/X^*$ contains the subgroup $(1,U/X^*)$. By (\cite{BL} Lemma 3.2.2. a)), $h(U,V)=1$. It follows that the action~\eqref{action of Heisenberg group} of $a\in (1,U/X^*)$ on $\mathbf{C}\times T$ is $T_a^*:(t,v)\mapsto (t, av)$. Therefore we also identify the action of $(1,U/X^*)$ on $\mathbf{C}\times T$ with the action of the real subtorus $U/X^*\subset V/X^*=T$ on $\mathbf{C}\times T$. Define $\phi: \Lambda^\vee\to X$ by $\phi(\lambda)(u)=E(\lambda, u)$ for $\lambda\in \Lambda^\vee$, $u\in X^*$. We generalize the Heisenberg relation in \cite{Mum72}. 

\begin{lemma}[Heisenberg relation \Rmnum{1}]\label{Heisenberg one}
Consider the two groups $N(\Lambda)/X^*$ and $T$ acting on the trivial line bundle $\mathbf{C}\times T$. Denote the action of $N(\Lambda)/X^*$ by $S$, and the action of $T$ by $T$, we have
\begin{equation}\label{commutator one}
T_a^*S_g^*=\mathrm{X}^{\phi(p(g))}(a)\cdot S_g^*T_a^*, \quad \forall g\in N(\Lambda)/X^*, a\in T,
\end{equation}

where $p$ is the projection $N(\Lambda)/X^*\to \Lambda^\vee/X^*$. 
\end{lemma}

\begin{proof}
It suffices to check \eqref{commutator one} for the real points $a\in U/X^*\subset T$. In this case, we can identify it as an element in $N(X^*)/X^*$, and use the commutation relation~\eqref{commutator in Heisenberg group} in the Heisenberg group $\mathcal{H}(V,E)$ for $N(\Lambda)$ and $(1,U)$. 
\end{proof}

We can generalize this relation to a general cusp $F_\xi$. Let $U_\xi$ be a rational isotropic subspace of dimension $r$. Represent $G$ as $ \widetilde{G}/Y_\xi$, where $\widetilde{G}=V/\Lambda\cap U_\xi^\perp$ and $Y_\xi=\Lambda/\Lambda\cap U_\xi^\perp$. Denote $\Lambda\cap U_\xi$ by $X^*_\xi$ and its dual by $X_\xi$. $E$ induces the polarization $\phi: Y_\xi\to X_\xi$. The algebraic torus $T_\xi=Hom(X_\xi, \mathbf{C}^*)$ is the torus part of the semiabelian variety $\widetilde{G}$.\\

Choose a Lagrangian $U\supset U_\xi$. We have inclusions $X^*_\xi\subset X^*$, $T_\xi\subset T$, and quotients $X\to X_\xi$, $Y\to Y_\xi$. The free group $Y':=U^\perp_\xi\cap \Lambda/X^*$ is the kernel of $Y\to Y_\xi$. Then $\Lambda^\vee/U^\perp_\xi\cap \Lambda$ (resp. $(\widetilde{G},\widetilde{\mathcal{L}})$) is the quotient of $\Lambda^\vee/X^*$ (resp. $(T,T\times\mathbf{C})$) by $Y'$. $Y'$ is naturally considered as a subgroup of $N(\Lambda)/X^*$ by the lift $\chi^{-1}$. Then $N(\Lambda)/U_\xi^\perp\cap\Lambda$ is a quotient of $N(\Lambda)/ X^*$ by $Y'$. The following lemma is just the $Y'$-quotient of Lemma~\ref{Heisenberg one}. Notice that by the relation \eqref{commutator one}, $Y'$ commutes with $T_\xi$ on $T\times\mathbf{C}$. Therefore the action of $T_\xi$, and the relation \eqref{commutator one} descends to $(\widetilde{G},\widetilde{\mathcal{L}})$. 

\begin{lemma}[Heisenberg relation \Rmnum{2}]\label{Heisenberg two}
Consider the two groups $N(\Lambda)/U_\xi^\perp\cap \Lambda$ and $T_\xi$ acting on $(\widetilde{G},\widetilde{\mathcal{L}})$. Denote the action of $N(\Lambda)/U_\xi^\perp\cap\Lambda$ by $S$, and the action of $T_\xi$ by $T$, we have
\begin{equation}\label{commutator two}
T_a^*S_g^*=\mathrm{X}^{\phi(p(g))}(a)\cdot S_g^*T_a^*, \quad \forall g\in N(\Lambda)/U_\xi^\perp\cap\Lambda, a\in T_\xi,
\end{equation}

where $p$ is the projection $N(\Lambda)/(U_\xi^\perp\cap\Lambda)\to \Lambda^\vee/U_\xi^\perp\cap\Lambda$, and $\phi$ is the induced map $\Lambda^\vee/U_\xi^\perp\cap\Lambda\to X_\xi$ by $E$.  
\end{lemma}

The theta group $\mathcal{G}(\mathcal{L})=N(\Lambda)/\Lambda$ is the quotient of $N(\Lambda)/U_\xi^\perp\cap \Lambda$ by $Y_\xi$. Denote the subgroup $\Lambda^\vee\cap U_\xi/\Lambda\cap U_\xi$ by $K_2$. The restriction of the Weil pairing to $K_2$ induces the surjection $K(\lambda)\to \widehat{K}_2$. Compose with $p:\mathcal{G}(\mathcal{L})\to K(\lambda)$, we have $w:\mathcal{G}(\mathcal{L})\to \widehat{K}_2$. Denote the kernel of $w$ by $\mathcal{G}(\mathcal{L})^f$. Identify $\widehat{K}_2$ with $X_\xi/\phi(Y_\xi)$. Let $I$ be an $X_\xi/\phi(Y_\xi)$-torsor. 

\begin{lemma}\label{two decompositions are the same}
Consider the canonical representation $\H^0(G,\mathcal{L})$ of $\mathcal{G}(\mathcal{L})$.  $\H^0(G,\mathcal{L})=\oplus_{\alpha\in I} V_\alpha$ is decomposed into irreducible representations of $\mathcal{G}(\mathcal{L})^f$, labelled by $I$. This is the same decomposition as the Fourier decomposition 
\[
\H^0(G,\mathcal{L})=\bigoplus_{\alpha\in X_\xi/\phi(Y_\xi)} \H^0(A,\mathcal{M}_\alpha)
\]

in the AN construction. Moreover, for $g\in \mathcal{G}(\mathcal{L})$, the action of $g$ translates the space labelled by $\alpha$ to the space labelled by $\alpha+w(g)$. 
\end{lemma}

\begin{proof}
Recall that the Fourier decomposition is obtained by the Fourier decomposition of $\H^0(\widetilde{G},\widetilde{\mathcal{L}})$ from the action of $T_\xi$. It follows from Lemma~\ref{Heisenberg two} that the action of $\mathcal{G}(\mathcal{L})^f$ preserves this decomposition. By (\cite{ols08} Theorem 5.4.2.), the theta group $\mathcal{G}(\mathcal{M})$ is contained in the subgroup $\mathcal{G}(\mathcal{L})^f$. Since each $\H^0(A,\mathcal{M}_\alpha)$ is already an irreducible representation of the subgroup $\mathcal{G}(\mathcal{M})$, it is irreducible for $\mathcal{G}(\mathcal{L})^f$. Since $w=\phi\circ p$, the last sentence follows from Equation~\eqref{commutator two}. 
\end{proof}

%\begin{remark}
%The relation~\ref{commutator two} and Lemma~\ref{two decompositions are the same} hold for families of complex abelian varieties over a sufficiently small neighborhood. Here, a neighborhood is sufficiently small, if it can be lifted to $\mathfrak{S}_g$. In this case, all the argument above still work as long as we add a parameter $\tau$.
%\end{remark}
%%%%%%%%%%%%%%%%%%%%%%%%%%%%%%%%%%%%%%%%%%%%%%%%%%%%%%%%%%%%%%%%%%%%%%%%%%%%%%
%%%%%%%%%%%%%%%%%%%%%%%%%%%%%%%%%%%%%%%%%%%%%%%%%%%%%%%%%%%%%%%%%%%%%%%%%%%%%%
%%%%%%%%%%%%%%%%%%%%%%%%%%%%%%%%%%%%%%%%%%%%%%%%%%%%%%%%%%%%%%%%%%%%%%%%%%%%%%

\subsection{The Stable Pairs}
It is not hard to generalize the above picture to the relative case. Assume the base $S$ satisfies the assumptions in (\cite{FC} Chap. \Rmnum{2} Sect. 3). 
\begin{assumption}\label{assumption for the only if part}
Suppose $S=\spec R$ for $R$ Noetherian, normal integral, local $\mathbf{Z}[1/d,\zeta_M]$-algebra, complete w.r.t. an ideal $I=\sqrt{I}$. Further assume that for any \'etale $R/I$-algebra, the unique lift to a formally \'etale $I$-adically complete $R$-algebra is normal. Denote the generic point by $\eta=S_\eta$, closed point by $S_0$ and the field of fractions by $K$. 
\end{assumption}

Consider a polarized semiabelian scheme $\pi: (\mathcal{G},\mathcal{L})\to S$ with the generic fiber $\pi: (\mathcal{G}_\eta,\mathcal{L}_\eta)\to S_\eta$ abelian. By (\cite{FC} Chap. \Rmnum{2} Sect. 2), we have two Raynaud extensions $0\to T\to \widetilde{G}\to A\to 0$ and $0\to T^t\to \widetilde{G}^t\to A^t\to 0$. Let $X_\xi$ (resp. $Y_\xi$) be the character group of $T$ (resp. $T^t$). We also have $c:X_\xi\to A^t$, $c^t:Y_\xi\to A$, $\lambda_\eta:G_\eta\to G_\eta^t$, $\lambda: G\to G^t$, $\lambda_T:T\to T^t$, $\lambda_A: A\to A^t$ and $\phi: Y_\xi\to X_\xi$. The data $(T, \widetilde{G}, A,\widetilde{\mathcal{L}})$ is the algebraization of the formal data $(T_{\text{for}}, G_{\text{for}}, A_{\text{for}},\mathcal{L}_{\text{for}})$. Choose a descent data for $\mathcal{L}_{\text{for}}$ w.r.t. $G_{\text{for}}\to A_{\text{for}}$. 

Consider the various group schemes in (\cite{FC} Chap. \Rmnum{2} pp. 46-47, \cite{MB} Chap. \Rmnum{4}). Let $K(\mathcal{L})_\eta$ be the kernel of $\lambda_\eta$ and $K(\mathcal{L})$ be the scheme theoretic closure of $K(\mathcal{L})_\eta$. The restriction of $\mathcal{L}$ to $K(\mathcal{L})$, regarded as a $\mathbb{G}_m$-torsor, has a natural group structure as a central extension
\begin{diagram}
1&\rTo &\mathbb{G}_m&\rTo&\mathcal{G}(\mathcal{L})&\rTo^p&K(\mathcal{L})&\rTo 0.
\end{diagram}

Since $\mathcal{G}(\mathcal{L})$ is a natural extension of the theta group $\mathcal{G}(\mathcal{L}_\eta)$ to $S$, it is also called the theta group. The quasi-finite group scheme $K(\mathcal{L})$ has the ``finite part" $K(\mathcal{L})^f$ which corresponds to $\Lambda^\vee\cap U^\perp_\xi/\Lambda\cap U^\perp_\xi$ in Section~\ref{The Heisenberg Relation and the Fourier Decomposition} and the ``multiplicative part" $K(\mathcal{L})^m:=\ker \lambda_T$ which corresponds to $\Lambda^\vee\cap U_\xi/\Lambda\cap U_\xi$ in Section~\ref{The Heisenberg Relation and the Fourier Decomposition}. $K(\mathcal{L})^m$ is denoted by $K_2$. The Weil pairing $e^{\mathcal{L_\eta}}$ over $K(\mathcal{L}_\eta)$ is extended to $e^\mathcal{L}: K(\mathcal{L})\times K(\mathcal{L})\to \mathbb{G}_m$ by using the commutator relation in $\mathcal{G}(\mathcal{L})$. $K_2$ and $K(\mathcal{L})^f$ are annihilator of each other w.r.t $e^\mathcal{L}$. $K_2\subset K(\mathcal{L})^f$ is isotropic and $K(\mathcal{L})^f/K_2=K(\mathcal{M}):=\ker \lambda_A$. The Cartier dual of $K_2$ is $\widehat{K}_2\cong X_\xi/\phi(Y_\xi)$. The restriction of the Weil pairing $e^\mathcal{L}$ defines a surjection $\phi: K(\mathcal{L})\to \widehat{K}_2$. 

\begin{lemma}[Heisenberg relation \Rmnum{3}]\label{Heisenberg three}
Under the above assumption, consider $\pi_*\mathcal{L}$ as the representation of $\mathcal{G}(\mathcal{L})$ and the representation of $K_2$. Denote the action of $\mathcal{G}(\mathcal{L})$ by $S$, and the action of $K_2$ by $T$, we have 
\begin{equation}\label{commutator three}
T^*_aS^*_g=X^{\phi(p(g))}(a)\cdot S^*_gT^*_a, \quad \forall g\in \mathcal{G}(\mathcal{L}), a\in K_2,
\end{equation}
where $p:\mathcal{G}(\mathcal{L})\to K(\mathcal{L})$ is the projection, and $\phi:K(\mathcal{L})\to \widehat{K}_2$ is the restriction of the Weil pairing. 
\end{lemma}
\begin{proof}
The point is $K_2$ is isotropic w.r.t the Weil pairing, and can be lifted as a subgroup scheme of the theta group $\mathcal{G}(\mathcal{L})$. Then the relation~\eqref{commutator three} is simply the commutator relation in $\mathcal{G}(\mathcal{L})$ because the Weil pairing $e^\mathcal{L}$ is defined in terms of the commutator pairing of $\mathcal{G}(\mathcal{L})$. 
\end{proof}

Let $\mathcal{G}(\mathcal{L})^f$ be the inverse image of $K(\mathcal{L})^f$ in $\mathcal{G}(\mathcal{L})$, $\mathcal{G}(\mathcal{L})^m$ the inverse image of $K_2$ in $\mathcal{G}(\mathcal{L})$. We have an exact sequence
\begin{diagram}\label{exact sequence of central extensions}
0&\rTo &\mathcal{G}(\mathcal{L})^m&\rTo & \mathcal{G}(\mathcal{L})^f &\rTo& K(\mathcal{M}) &\rTo &0.
\end{diagram}

The restriction of $\mathcal{G}(\mathcal{L})^f$ (resp. $\mathcal{G}(\mathcal{L})^m$) to $S_\eta$ is denoted by $\mathcal{G}(\mathcal{L})_\eta^f$ (resp. $\mathcal{G}(\mathcal{L})_\eta^m$). By (\cite{Shin}  Proposition 2.12, 2.13), $\pi_*\mathcal{L}_\eta$ (resp. $\pi_*\mathcal{M}$) is locally free of rank $d$ (resp. $d_a$) and is an irreducible representation of $\mathcal{G}(\mathcal{L}_\eta)$ (resp. $\mathcal{G}(\mathcal{M})$). Since $\mathcal{G}(\mathcal{L})^f$ is a sub group scheme of $\mathcal{G}(\mathcal{L})$, $\pi_*\mathcal{L}$ is also a $\mathcal{G}(\mathcal{L})^f$-module. 

Recall we also have the Fourier decomposition~\eqref{Fourier decomposition equation} for $\mathcal{L}$ over $G$\footnote{We emphasize that here $\mathcal{L}$ is over a semiabelian scheme $G$, not a semiabelic scheme $\mathcal{X}$. But the decompositions are similar. We will only use the decomposition restricted to the abelian part.}. We have fixed the descent data for $\mathcal{L}_{\text{for}}$. If we change the action of $T_{\text{for}}$ on $\mathcal{L}_{\text{for}}$ by a character $\alpha\in X_\xi$, then $\mathcal{L}_{\text{for}}$ descends to an ample cubical invertible sheaf $\mathcal{M}_{\alpha, \text{for}}$ over $A_{\text{for}}$. Let $\mathcal{M}_\alpha$ be the algebraization of $\mathcal{M}_{\alpha,\text{for}}$. The Fourier decomposition is induced from the decomposition
\begin{equation}\label{the decomposition upstairs}
\pi_*\widetilde{\mathcal{L}}=\bigoplus_{\alpha\in X_\xi} \pi_* \mathcal{M}_\alpha
\end{equation}
Here we abuse the notation and denote the morphism $A\to S$ also by $\pi$. 

\begin{lemma}\label{two decompositions are the same in general}
The Fourier decomposition 
\begin{equation}\label{Fourier decomposition two}
\pi_*\mathcal{L}=\bigoplus_{\alpha\in X_\xi/\phi(Y_\xi)} \pi_*\mathcal{M_\alpha}=\bigoplus_{\alpha\in \widehat{K}_2}\mathcal{V}_\alpha
\end{equation}

is the decomposition of $\pi_*\mathcal{L}$ into a direct sum of irreducible $\mathcal{G}(\mathcal{L})^f$-modules.
\end{lemma}
\begin{proof}
First, notice that the decomposition~\eqref{the decomposition upstairs} is the decomposition into irreducible $T_\xi$-representations and is determined by the characters in $X_\xi=\Hom(T_\xi, \mathbb{G}_m)$. Since $\mathcal{L}$ is the descent of $\mathcal{L}_{\text{for}}$ by $Y_\xi$-action, the Fourier decomposition~\eqref{Fourier decomposition two} is determined by the characters up to $Y_\xi$-action. Now $K_2$ is a subgroup of $T_\xi$ and the dual group $\widehat{K}_2$ is the quotient of $X_\xi$ by the action of $Y_\xi$. Therefore, the Fourier decomposition of $\pi_*\mathcal{L}$ is the decomposition of the irreducible representations of $K_2$. 

By Lemma~\ref{Heisenberg three}, The action of $\mathcal{G}(\mathcal{L})^f$ preserves $\mathcal{V}_\alpha$. It suffices to show that $\mathcal{V}_\alpha$ is irreducible as a representation of $\mathcal{G}(\mathcal{L})^f$. This is well-explained in (\cite{FC} Chap.\Rmnum{2} p.47). Since $K_2$ is isotropic w.r.t. the Weil pairing, the central extension $\mathcal{G}(\mathcal{L})^m$ is trivial. Choose a splitting of 
\begin{diagram}
0&\rTo &\mathbb{G}_m&\rTo & \mathcal{G}(\mathcal{L})^m&\rTo& K_2 &\rTo &0
\end{diagram}

Any character $\alpha\in \widehat{K}_2$ can be extended to a character $\tilde{\alpha}$ of $\mathcal{G}(\mathcal{L})^m$ by requiring that $\tilde{\alpha}$ is the identity on $\mathbb{G}_m$ and is $\alpha$ on $K_2$. The push forward of the exact sequence~\eqref{exact sequence of central extensions} along $\tilde{\alpha}$ is 
\begin{diagram}
0&\rTo &\mathbb{G}_m&\rTo & \mathcal{G}(\mathcal{L})^f_\alpha&\rTo& K(\mathcal{M}) &\rTo &0
\end{diagram}
is identified with the theta group $\mathcal{G}(\mathcal{M}_\alpha)$ as a central extension
\begin{diagram}
0&\rTo &\mathbb{G}_m&\rTo & \mathcal{G}(\mathcal{M}_\alpha)&\rTo& K(\mathcal{M}_\alpha) &\rTo &0.
\end{diagram}

The action of $\mathcal{G}(\mathcal{L})^f$ on $\mathcal{V}_\alpha$ factors through the action of $\mathcal{G}(\mathcal{L})^f_\alpha$ which is the irreducible theta representation of $\mathcal{G}(\mathcal{M}_\alpha)$ on $\mathcal{V}_\alpha$. It finishes the proof that $\mathcal{V}_\alpha$ is an irreducible $\mathcal{G}(\mathcal{L})^f$-module.
\end{proof}
 
Let $M=2\delta_g$. Now we restrict the scalars to roots of unity $\mu_M$. Consider the theta group $\mathcal{G}(\mathcal{L}_\eta)$ for an abelian scheme $G_\eta$ and the map
\begin{align}
P_{M} : &\mathcal{G}(\mathcal{L}_\eta)\to \mathcal{G}(\mathcal{L}_\eta)\\
& g\mapsto g^{M}\quad \forall g\in \mathcal{G}(\mathcal{L}_\eta)(S).
\end{align}

\begin{lemma}
The map $P_{M}$ is a group homomorphism. 
\end{lemma}
\begin{proof}
Fix a scheme $S$, and consider $g, h\in \mathcal{G}(\mathcal{L}_\eta)(S)$. Recall the commutator $ghg^{-1}h^{-1}=e^{\mathcal{L}_\eta}(p(g),p(h))$. By induction, we have
\[
g^nh^n=(e^{\mathcal{L}_\eta}(p(g),p(h)))^{n-1}ghg^{n-1}h^{n-1}=(e^{\mathcal{L}_\eta}(p(g),p(h)))^{n(n-1)/2}(gh)^n.
\]

Since the order of $e^{\mathcal{L}_\eta}(p(g),p(h))$ divides $\delta_g$, $P_{M}(g)P_{M}(h)=P_{M}(gh)$. 
\end{proof}

Let $G(M)$ denote the kernel of $P_{M}$. It is a subgroup scheme of $\mathcal{G}(\mathcal{L}_\eta)$. Since $\mu_M$ is defined as the kernel of the group homomorphism
\begin{align}
&\mathbb{G}_m\to \mathbb{G}_m\\
& g\mapsto g^{M}\quad \forall g\in \mathbb{G}_m(S),
\end{align}

we have 
\begin{equation}\label{fourth exact sequence}
\begin{diagram}
1&\rTo &\mu_M&\rTo &G(M)&\rTo^p & K(\mathcal{L})_\eta &\rTo 0. 
\end{diagram}
\end{equation}

Let $\mathcal{H}(\delta,M)$ be the Heisenberg group. It is a central extension of $H(\delta)\times\widehat{H(\delta)}$ by $\mu_M$. The multiplication is determined by the requirement that if $\tilde{g}, \tilde{h}$ are lifts of $g\in H(\delta)$ and $h\in \widehat{H(\delta)}$ in $\mathcal{H}(\delta,M)$, then the commutator $\tilde{g}\tilde{h}\tilde{g}^{-1}\tilde{h}^{-1}=h(g)$. \'Etale locally, $G(M)$ is isomorphic to the constant group $\mathcal{H}(\delta, M)$.

Assume $\pi:(\mathcal{X},G,\mathcal{L},\varrho)\to S$ a polarized stable semiabelic scheme, such that the generic fiber $\pi: (\mathcal{X}_\eta,G_\eta,\mathcal{L}_\eta,\varrho_\eta)\to S_\eta$ is abelian. Compose $p :G(M)\to K(\mathcal{L})_\eta$ with $\phi: K(\mathcal{L})_\eta\to \widehat{K}_2$, we get a homomorphism $w: G(M)\to \widehat{K}_2$. Let $K_w$ be the kernel of $w$. $K_w$ can be obtained from $\mathcal{G}(\mathcal{L})^f_\eta$ by restricting the scalars to $\mu_M$. Consider the irreducible $G(M)$-representation $\pi_*\mathcal{L}_\eta$. After a base change, \'{e}tale over $S_\eta$, $G(M)$ and $\widehat{K}_2$ are constant, $\pi_*\mathcal{L}_\eta$ decomposes into irreducible $K_w$-representations $\pi_*\mathcal{L}_\eta=\oplus_{\alpha\in I}\mathcal{V}_\alpha$, where $I$ is a torsor for the group $\widehat{K}_2$. For any $g\in G(M)$, the action of $g$ translates $\mathcal{V}_\alpha\to \mathcal{V}_{\alpha+w(g)}$. Fix an element of $I$ and identify $I$ with the group $\widehat{K}_2$. A lift $\sigma: \widehat{K}_2\to G(M)$ is a section for the map $w:G(M)\to \widehat{K}_2$. We do not require $\sigma$ to be a group homomorphism.

\begin{proposition}\label{the only if direction}
If $S$ satisfies Assumption~\ref{assumption for the only if part} and $\pi:(\mathcal{X}, G, \mathcal{L},\varrho)/S$ is the pull-back of the AN family along a morphism $g:S\to\overline{\mathscr{A}}_{g,\delta}^m$, then after an \'{e}tale base change, we can extend $G(M)$ and $K_w$-representation $\pi_*\mathcal{L}=\oplus_{\alpha\in I}\mathcal{V}_\alpha$ over $S$. Take any section $\vartheta_0\in \mathcal{V}_0^*$, and any lift $\sigma: \widehat{K}_2\to G(M)$. Define
\begin{equation}
\vartheta:=\sum_{\alpha\in I} S_{\sigma(\alpha)}^*\vartheta_0. 
\end{equation} 

Let $\Theta$ be the zero locus of $\vartheta$. Then $(\mathcal{X},G, \mathcal{L},\Theta,\varrho)/S$ is an object in $\overline{\mathscr{AP}}_{g,d}$. 
\end{proposition}

\begin{proof}
Making an \'{e}tale base change if necessary, we can assume $g: S\to \Xi_\mathscr{P}$ over the cusp $F_\xi$ for some bounded paving $\mathscr{P}$ on $U_\xi^*$. The family $(\mathcal{X},\mathcal{L}, G,\varrho)/S$ is the pull-back of AN family along $g$. Since AN construction is functorial (Proposition~\ref{AN functorial}), $(\mathcal{X}, \mathcal{L},G,\varrho)/S$ is constructed from the following data. There is an exact sequence of abelian sheaves
\begin{equation}
\begin{diagram}
1&\rTo &T_\eta&\rTo &\widetilde{G}_\eta&\rTo&A_\eta&\rTo 0\\
&&\dInto &&\dInto&&\dInto&&\\
1&\rTo &T&\rTo &\widetilde{G}&\rTo&A&\rTo 0
\end{diagram}\ ,
\end{equation}

where the top line is the restriction to $S_\eta$. Recall the data $X_\xi$,$Y_\xi$, $\phi:Y_\xi\to X_\xi$, $c$, $c^t$, $ \tau$, $\psi$ and $\mathcal{M}$.  Over $S_\eta$, $\widetilde{G}_\eta=\spec_{A_\eta} \bigoplus_{\alpha\in X_\xi}\mathcal{O}_\alpha$, and the line bundle $\widetilde{\mathcal{L}}_\eta$ is defined by 
\[
\mathcal{S}:=\prod_{d\geqslant 0}\bigg(\bigoplus_{\alpha\in X_\xi}\mathcal{O}_\alpha\otimes \mathcal{M}_\eta^d\theta^d\bigg)\otimes \mathcal{O}_{A_\eta}=\bigoplus_{(d,\alpha)\in S(X_\xi)} \mathcal{S}_{d,\alpha}\theta^d.
\]

Over $S$, $(\widetilde{\mathcal{X}}, \widetilde{\mathcal{L}})$ is defined by the graded $\mathcal{O}_A$-algebra 
\[
\mathcal{R}:=\bigoplus_{(d,\alpha,p)\in Q_{\tilde{\varphi}}}\mathrm{X}^p\otimes\mathcal{O}_\alpha\otimes\mathcal{M}^d\theta^d=\bigoplus_{(d,\alpha)\in S(X_\xi)} \mathcal{R}_{d,\alpha}\theta^d.
\] 

The degree-$1$ part is the Fourier decomposition. By Lemma~\ref{two decompositions are the same in general}, it agrees with the pull back of the decomposition $\oplus\mathcal{V}_\alpha$. In other words, $\mathcal{V}_\alpha=\pi_*\mathcal{M}_{\alpha,\eta}$ over $S_\eta$. The inclusion $\mathcal{R}\to \mathcal{S}$ is defined by the graph of the $\mathscr{P}$-piecewise affine function $\varphi$. As the pull-back of $\varphi_\mathscr{P}$, $\varphi$ is $X_\xi$-quasiperiodic. 
 
 For $\lambda\in Y_\xi$, the action $S_\lambda^*$ on $\widetilde{G}$ is expressed as 
 \[
 \psi(\lambda)^m\tau(\lambda, \alpha): T^*_{c^t(\lambda)}(\mathcal{M}^m\otimes\mathcal{O}_\alpha)\to \mathcal{M}^m\otimes\mathcal{O}_{\alpha+m\phi(\lambda)}, 
 \]
 
for $\psi(\lambda)=a(\lambda)\psi'(\lambda)$, $\tau(\lambda,\alpha)=b(\lambda,\alpha)\tau'(\lambda,\alpha)$, with $\psi'$, $\tau'$ trivializations over $S$ and with $a, b$ having values in $\mathcal{O}_S$. Since $\{S^*_\lambda\}$ is a group action, $a$, $b$ satisfy relations
 \begin{align}\label{quadratic relation}
 a(0)&=1\\
 a(\lambda+\mu)&=b(\lambda,\phi(\mu))a(\lambda)a(\mu).
 \end{align}
 
Consider a discrete valuation $v$ associated with a height $1$ prime ideal. Let $A_\varphi(v): X_\xi\to \mathbf{Z}$ be the quadratic function associated with $v\circ\varphi$. Let $A(v): Y_\xi\to \mathbf{Z}$ be the composition $v\circ a$ and $B(v):Y_\xi\times X_\xi\to \mathbf{Z}$ be the composition $v\circ b$. Then $A(v)$ is quadratic over $Y_\xi$ with the associated bilinear form $B(v)(\cdot,\phi(\cdot))$. We have $A(v)_\varphi\vert_{Y_\xi}=A(v)$. 
 
Let $\widetilde{G}(M)$ be the pull-back
\begin{equation}
\begin{diagram}
\widetilde{G}(M)&\rTo^w&X_\xi\\
\dTo&&\dTo\\
G(M)&\rTo^w&\widehat{K}_2
\end{diagram}.
\end{equation}

The vertical arrows are quotients by $Y_\xi$. The group action of $G(M)$ on $(\mathcal{X}_\eta,\mathcal{L}_\eta)$ can be lifted to the group action of $\widetilde{G}(M)$ on $(\widetilde{G}_\eta,\widetilde{\mathcal{L}}_\eta)$ which extends the action of $Y_\xi$ on $(\widetilde{G}_\eta,\widetilde{\mathcal{L}}_\eta)$. Write the action of $\widetilde{G}(M)$ also in the form $a^db(\psi')^d\tau'$,\footnote{In other words, $\psi'$, $\tau'$ are trivializations over $S$.} and take the valuations $A_X(v):=v\circ a$ and $B_X(v):=v\circ b$. These two functions factor through $X_\xi$ and $X_\xi\times X_\xi$. This is because for different lifts of elements from $X_\xi$ to $\widetilde{G}(M)$, the differences are in $\mu_M$, and they have the same order under $v$. Since $a$, $b$ are defined from group actions, they satisfy the same relations as Relation~\eqref{quadratic relation}. Therefore $A_X(v)$ is also a quadratic function over $X_\xi$ extending $A(v)$, and $A_\varphi(v)=A_X(v)$ for all $v$.

Since $R$ is a normal Noetherian domain, $R=\cap_{\text{ht}\, \mathfrak{p}=1}R_\mathfrak{p}$ (\cite{Mat86} Theorem 11.5). It follows that the difference between the quadratic part of $\mathrm{X}^{\varphi(w(g))}$ and $a(g)$ is invertible in $R$ for all $g\in \widetilde{G}(M)$. Therefore the values of $a$, $b$ for $\widetilde{G}(M)$ are regular functions over $S$. The action of $G(M)$ is defined on $(\mathcal{X},\mathcal{L})$ over $S$ and $S_g^*$ maps $\mathcal{R}_{0,1}$ to $\mathcal{R}_{\phi(g),1}$ for any $g\in\widetilde{G}(M)$. In particular, $S_g^*(\vartheta_0)$ is in $(\pi_*\mathcal{M}_{\phi(g)})^*$ over $S$. Therefore, for any lift $\sigma: I\to G(M)$, the section 
\[
\vartheta:=\sum_{\alpha\in I}S^*_{\sigma(\alpha)}\vartheta_0
\] 

is stable. The family $(\mathcal{X},G, \Theta,\varrho)/S$ is in $\overline{\mathscr{AP}}_{g,d}$ by Theorem~\ref{alexeev family}. 
\end{proof}

\begin{definition}\label{the balanced set}
If an isotropic subgroup $K_2$ of $K(\lambda)$ or $\widehat{K}_2=X_\xi/\phi(Y_\xi)$ is well-defined for the family, call the set of divisors obtained in Proposition~\ref{the only if direction} the balanced set and denote it by $S(K_2)$. 
\end{definition}

\begin{remark}\label{K2 is defined locally}
The isotropic subgroup $K_2$ is well-defined for the family when the base is local, for example, when $S$ satisfies Assumption~\ref{assumption for the only if part}. It is also well-defined when the interior is a punctured polydisc $(\Delta^*)^n$ (\cite{CCK} Proposition 2.1) $K_2=\Lambda^\vee\cap W_0/\Lambda\cap W_0$.
\end{remark}

\begin{corollary}\label{extension over DVR}
If the base $S$ is a DVR, and we have a polarized abelian variety $(G_\eta, \mathcal{X}_\eta,\mathcal{L}_\eta)$ over the generic point, we can add the central fiber as follows. The monodromy defines an isotropic subgroup $K_2$. Pick any divisor $\Theta_\eta$ from the balanced set $S(K_2)$, we get an object $(G_\eta,\mathcal{X}_\eta,\Theta_\eta)$ in $\overline{\mathscr{AP}}_{d,g}$. Since $\overline{\mathscr{AP}}_{d,g}$ is proper, we can uniquely extend the family $(G,\mathcal{X},\Theta)$ over $S$. Then forget about the divisor $\Theta$, we get a family $(\mathcal{X}, G,\mathcal{L},\varrho)$ over $S$. This is the pull-back family from $\overline{\mathscr{A}}^m_{g,\delta}$, and is independent of the choice of $\Theta\in S(K_2)$. 
\end{corollary}
\begin{proof}
Since $\overline{\mathscr{A}}^m_{g,\delta}$ is also proper, we have a unique morphism $S\to\overline{\mathscr{A}}_{g,\delta}^m$ and we can pull back the AN family. By Proposition~\ref{the only if direction}, the pull-back family coincides with $(\mathcal{X},G,\mathcal{L},\varrho)$. 
\end{proof}

The following is the geometric description of our extended families near a cusp $F_\xi$. 

\begin{theorem}\label{stable pairs moduli interpretation}
Suppose $S$ satisfies Assumption~\ref{assumption for the only if part} and in addition $R$ strictly henselian. Let $(\mathcal{X},G,\mathcal{L},\varrho)/S$ be a polarized stable semiabelic scheme over $S$, with the generic fiber abelian. Over the generic point $\eta$, the group subscheme $K_2\subset K(\mathcal{L})_\eta$ and the stable set of divisors $S(K_2)$ are defined by the Raynaud extension. Then $(\mathcal{X},G,\mathcal{L},\varrho)$ is the pull-back of the AN family along a unique\footnote{For a morphism to a stack, ``being unique'' means that, if there is another morphism $g'$ with the same properties, the morphism $(g,g')$ factors through the diagonal.} morphism $g:S\to \overline{\mathscr{A}}_{g,\delta}^m$ if and only if the group scheme $G(M)$ can be extended over $S$ (thus $S(K_2)$ is defined over $S$), and for one (equivalently any) divisor $\Theta$ from $S(K_2)$, $(\mathcal{X},G,\Theta,\varrho)$ is an object in $\overline{\mathscr{AP}}_{g,d}$. 
\end{theorem}
\begin{proof}
The ``only if" part is Proposition~\ref{the only if direction}. Suppose there exists $\Theta_\eta\in S(K_2)$ such that $(\mathcal{X},G,\Theta,\varrho)$ is an object in $\overline{\mathscr{AP}}_{g,d}$. This is equivalent to a morphism $f: S\to \overline{\mathscr{AP}}_{d,g}$. The semiabelian scheme $G/S$ gives rise to \'{e}tale constructible sheaves $\underline{X}$, $\underline{Y}$, polarization $\phi: \underline{Y}\to \underline{X}$ and pairing $\underline{B}: \underline{Y}\times\underline{X}\to \underline{\Div} S$. Over the closed point $S_0$, the fiber is $X_\xi$, $Y_\xi$, $\phi: Y_\xi\to X_\xi$ and $b:Y_\xi\times X_\xi\to K^*$. Over any $s\in S$, $\underline{X}_s$ (resp. $\underline{Y}_s$) is the quotient of $X_\xi$ (resp. $Y_\xi$). 

For each point $s\in S$, choose a DVR $T\to S$ with the generic point mapped to $\eta$ and the closed point mapped to $s$. By Corollary~\ref{extension over DVR}, the pull-back family over $T$, denoted by $(\mathcal{X}_T, G_T,\Theta_T,\varrho_T)$, comes from AN family over $\overline{\mathscr{A}}^m_{g,\delta}$. In particular, $(\mathcal{X}_T, G_T,\Theta_T,\varrho_T)$ is constructed from an $\underline{X}_s$-quasiperiodic piecewise affine function $\varphi_T:\underline{X}_{s,\mathbf{R}}\to \mathbf{R}$. Therefore the paving $\mathscr{P}_s$ associated with $\varphi_T$ is $\underline{X}_s$-invariant. Let $b_s$ be $\underline{B}_s$ and $v$ the discrete valuation of $T$. The bilinear form $v\circ b_s$ is inside the cone $C(\mathscr{P}_s)$ of the second Voronoi fan $\Sigma(\underline{X}_s)$. The paving $\mathscr{P}_s$, regarded as a cell decomposition of $\underline{X}_{s,\mathbf{R}}/\phi(\underline{Y}_s)$, is the cell complex $\Delta_0$ in (\cite{Alex02} Definition 5.7.2) and can be defined by the fiber $(\mathcal{X}_s,\mathcal{L}_s)$. Moreover, since $s$ specializes to $S_0$, if we pull back the cell decomposition $\mathscr{P}_s$ along $X_\xi\to \underline{X}_s$, it is coarser than the paving $\mathscr{P}$ associated with the central fiber $(\mathcal{X}_0,\mathcal{L}_0)$. Therefore the cone $C(\mathscr{P}_s)$ is a face of $C(\mathscr{P})$ in $\Sigma(X_\xi)$. Write $X_\xi$ as a quotient of $X=\mathbf{Z}^g$. In particular, if $v:K^*\to \mathbf{Z}$ is any discrete valuation defined by a prime ideal of height one, then $v\circ b$ is contained in the closed cone $C(\mathscr{P})$ in the second Voronoi fan $\Sigma(X)$. By Proposition~\ref{the universal property of the moduli space}, there exists a morphism $g: S\to \overline{\mathscr{A}}_{g,\delta}^m$ such that if $(\mathcal{X}',\mathcal{L}', G',\varrho')$ is the pull-back of the AN family along $g$, then $G'\cong G$. Take the closure $\Theta'$ of $\Theta_\eta$ for $(G',\mathcal{X}',\mathcal{L}',\varrho')$, so that $\Theta'_\eta$ is identified with $\Theta_\eta$ through the isomorphism. By Proposition~\ref{the only if direction}, we get another morphism $g':S\to \overline{\mathscr{AP}}_{g,d}$. Consider the cartesian diagram
\[
\begin{diagram}
S'&\rTo&\overline{\mathscr{AP}}_{g,d}\\
\dTo^h&&\dTo_{\Delta}\\
S&\rTo^{(f,g')}&\overline{\mathscr{AP}}_{g,d}\times\overline{\mathscr{AP}}_{g,d}
\end{diagram}
\]
 
Since $\Delta$ is finite (\cite{Alex02} Theorem 5.10.1), $h$ is finite. Moreover, $S_\eta\to S$ factors through $h$. Since the local charts for $\overline{\mathscr{AP}}_{g,d}$ are integral (\cite{Alex02} 5.9, they are semigroup $k$-algebras), $S'$ is integral. By Lemma~\ref{finite to isomorphic}, $S'=S$. The families $(\mathcal{X},\mathcal{L},G,\Theta,\varrho)$ and $(\mathcal{X}',\mathcal{L}',G',\Theta',\varrho')$ are isomorphic. 

The morphism $g: S\to \overline{\mathscr{A}}_{g,\delta}^m$ thus obtained is unique, because the diagonal of $\overline{\mathscr{A}}_{g,\delta}^m$ is also finite. We use the same argument as above.
\end{proof}

\begin{lemma}\label{finite to isomorphic}
Let $h_1: R\to R'$ be a finite extension of integral domains and $h_2:R'\to K$ be a morphism such that the composition $h_2\circ h_1:R\to K$ is the inclusion of $R$ into its field of fractions. If $R$ is normal, then $h_1$ is an isomorphism.
\end{lemma}
\begin{proof}
We claim that $h_2$ is injective. Since $h_2\circ h_1$ is injective, $h_1$ is injective and $\ker h_2$ is a prime ideal lying over the zero ideal of $R$. Moreover $h_1$ is finite, hence an integral extension. By (\cite{Mat86} Theorem 9.3 (ii)), the only prime ideal lying over the zero ideal of $R$ is the zero ideal of $R'$. Therefore $\ker h_2=0$.  We have $R\subset R'\subset K$. If $R$ is normal, then $R=R'$. 
\end{proof}

\begin{remark}
Heuristically, locally near the boundary, the compactification $\overline{\mathscr{A}}^m_{g,\delta}$ should be the normalization of a slice of $\overline{\mathscr{AP}}_{g,d}$. The slice is defined by a choice of a divisor in the balanced set $S(K_2)$. The divisors in $S(K_2)$ should be regarded to be the most ``symmetric" because it is an average over the lift $\sigma(\widehat{K}_2)$. The tropical avatar of the slice is the linear section $\sigma$ in Equation~\eqref{the linear section is the average}. If we forget about the divisor $\Theta$ and only consider the family $(\mathcal{X},\mathcal{L},G,\varrho)$, then the normalization doesn't depend on the choice of $\Theta$.
\end{remark}

\begin{remark}
The families have been studied in \cite{Nak10} and \cite{ols08}. Moreover, when the polarization is separable, it is proved in (\cite{Nak1} Definition 5.11, Lemma 5.12) and \cite{ols08} that $G(M)$ and its representation can be extended to the boundary. What we are suggesting here is, one can characterize the extended families by using $G(M)$ and $\overline{\mathscr{AP}}_{g,d}$. 
\end{remark}

\begin{remark}
The balanced set $S(K_2)$ depends on $K_2$, and $K_2$ is only locally well-defined. See Remark~\ref{K2 is defined locally}. Even when the coarse moduli space has only one cusp, $K_2$ is never well--defined over the whole moduli space unless the polarization is principal. To parallel transport $W_0$, we need the Gauss--Manin connection. That means, we need an actual family over the moduli space. In all cases, we will need at least a finite ramified cover to get the family and there will be more than one cusp. Then it can be proved that there is no section stable for all $K_2$ unless the polarization is principal. 
\end{remark}

If $K_2$ is a Lagrangian, $A$ and $\mathcal{M}$ are trivial and $\mathcal{V}_0\cong \mathcal{O}_S$. $S(K_2)$ is a finite set. If the Lagrangian $K_2$ further splits, i.e., it admits an isotropic complement $K(\lambda)=K_1\oplus K_2$, we can identify $K_1$ and $I$, and require that the lifts are group homomorphisms $\sigma: K_1\to G(M)$.  In this case, the lift $\sigma(K_1)$ is a maximal level subgroup $\widetilde{K}_1$ of $\mathcal{G}(\mathcal{L})$. There are altogether $d$ choices, and each choice is equivalent to a choice of the descent data $h: \mathcal{X}\to \mathcal{X}/K_1$, with $\nu: h^*\mathcal{L}'\cong \mathcal{L}$ for a principally polarized line bundle $\mathcal{L}'$ over $\mathcal{X}/K_1$. The stable section $\vartheta$ is the pull-back of the unique (up to a scaling) section of $\mathcal{L}'$. The sections $S_{\sigma(\alpha)}^*\vartheta_0$ are the classical theta functions parameterized by $K_1$.

%%%%%%%%%%%%%%%%%%%%%%%%%%%%%%%%%%%%%%%%%%%%%%%%%%%%%%%%%%%%%%%%%%%%%%%%%%%%%%%%%%%%%%% 
 
%%%%%%%%%%%%%%%%%%%%%%%%%%%%%%%%%%%%%%%%%%%%%%%   appendix %%%%%%%%%%%%%%%%%%%%%%%%%%%%%%%%%%

%%%%%%%%%%%%%%%%%%%%%%%%%%%%%%%%%%%%%%%%%%%%%%%%%%%%%%%%%%%%%%%%%%%%%%%%%%%%%%%%%%%%%%%
\appendix
\section{The Quasiperiodic Functions}\label{A}
Let $Y\cong \mathbf{Z}^g$ be a finitely generated free abelian group acting on an affine space $V\cong \mathbf{R}^g$ by translations. Let $\psi$ be a real valued piecewise affine function on $V$.
\begin{definition}\label{definition of quasiperiodic functions}
A piecewise affine function $\psi$ is called quasi-periodic with respect to $Y$, if
\begin{equation}\label{quasi-periodic}
\psi(x+\lambda)-\psi(x)=A_\lambda(x),  \ \forall \lambda \in Y, \ \forall x\in V,
\end{equation}

for $A_\lambda(x)$ an affine function on $V$ that depends on $\lambda\in Y$. 
\end{definition}

\begin{remark}\label{Appendix section of sheaves}
A $Y$-quasiperiodic function $\psi$ can be regarded as an element in $\Gamma(V/Y, \mathcal{P}A/\mathcal{A}ff)$. 
\end{remark}

\begin{lemma}\label{quasiperiodic and quadratic}
If $\psi$ is quasi-periodic with respect to $Y$, then there exists some quadratic function $A$ such that $\psi-A$ is a $Y$-periodic function on $V$. 
\end{lemma}

\begin{proof}
Fix a point $x_0\in V$ and regard $V$ as a vector space. Embed $Y$ as a subset of $V$. For $\lambda,\mu\in Y$, we have 
\begin{align}
\psi(x+\mu+\lambda)&=\psi(x+\mu)+A_\lambda(x+\mu)\\
&=\psi(x)+A_\mu(x)+A_\lambda(x+\mu) \label{2q} \\
&=\psi(x)+A_\lambda(x)+A_\mu(x+\lambda)\label{3q} \\
&=\psi(x)+A_{\mu+\lambda}(x).  \label{4q}
\end{align}

From \eqref{3q} and \eqref{4q}
\[
A_{\lambda+\mu}(x)=A_\mu(x+\lambda)+A_\lambda(x)
\]

Therefore $\{A_\lambda\}$ is a $1$-cocycle for the $Y$-module $Aff(X_\mathbf{R},\mathbf{R})$. \\

Suppose $A_\lambda(x)=B(\lambda,x)+A(\lambda)$, for a linear function $B(\lambda,\cdot)$ on $V$ and a constant $A(\lambda)$. From \eqref{2q} and \eqref{3q}, we have
\[
B(\mu,\lambda)=A_\mu(x+\lambda)-A_\mu(x)=A_\lambda(x+\mu)-A_\lambda(x)=B(\lambda,\mu).
\]

Applying \eqref{4q},
\[
A_{\mu+\lambda}(x)-A_\mu(x)-A_\lambda(x)=B(\mu,\lambda),
\]

It follows that
\begin{align}
&B(\mu+\lambda,x)=B(\mu,x)+B(\lambda,x),\\
&A(\mu+\lambda)-A(\mu)-A(\lambda)=B(\mu,\lambda).
\end{align}

Therefore $B(\mu,\lambda)$ is a symmetric bilinear form on $Y$, and
\[
A(\gamma)=\frac{1}{2}B(\gamma,\gamma)+\frac{1}{2}L(\gamma)
\]
is a quadratic function associated with $B$. Here $L$ is a linear function on $Y$.\\

Extend $A$ and $B$ to functions on $V$. Since $A(x+y)-A(x)=B(y,x)+A(y)$, for all $x,y\in V$,
\[
\psi(x+\lambda)-A(x+\lambda)=\psi(x)-A(x),\ \forall \lambda\in Y, x\in V.
\]
In other words, $\psi-A$ is a periodic function with respect to $Y$. 
\end{proof}

If we forget about $x_0$ initially chosen, the quadratic part  of $A$ is well-defined. 

\begin{definition}\label{the associated quadratic form}
The quadratic form $1/2B(x,x)$ is called the quadratic form associated with $\psi$. 
\end{definition}

If the piecewise affine function $\varphi$ takes values in a vector space $P^{gp}_\mathbf{R}$, we can define $Y$-quasiperiodic similarly. Recall a quadratic function is an element in $\Gamma^2V^*\otimes P^{gp}_\mathbf{R}$. We can generalize Lemma~\ref{quasiperiodic and quadratic}. 

\begin{corollary}\label{quasiperiodic and quadratic for vector valued}
If $\psi$ is quasi-periodic with respect to $Y$, then there exists some quadratic function $A$ such that $\psi-A$ is a $Y$-periodic function on $V$. 
\end{corollary}

%%%%%%%%%%%%%%%%%%%%%%%%%%%%%%%%%%%%%%%%%%%%%%%%%%%%%%%%%%%%%%%%%%%%%%%%%%%%%%%%%%%%%%%%

%%%%%%%%%%%%%%%%%%%%%%%%%%%%%%  B     %%%%%%%%%%%%%%%%%%%%%%%%%%%%%%%%%%%%%%%%%%%%%%%%%%%%%%%

%%%%%%%%%%%%%%%%%%%%%%%%%%%%%%%%%%%%%%%%%%%%%%%%%%%%%%%%%%%%%%%%%%%%%%%%%%%%%%%%%%%%%%%%

\section{Degenerations of One-parameter Families}\label{B}
\subsection{Maximal Degeneration}\label{B1}
In this section, we compute the degeneration data for a general $0$-cusp. Assume the base field $k=\mathbf{C}$.

Let $\pi:\mathcal{X}_\eta\to \Delta^*$ be a one-parameter family of polarized abelian varieties and $X$ be a generic fiber. Let $V:=\H_1(X,\mathbf{R})$ and $\Lambda:=\H_1(X,\mathbf{Z})$. The polarization is defined by a skew-symmetric integral bilinear form $E$ over $\Lambda$. The log monodromy defines a weight filtration $0\subset W_0\subset W_1\subset W_2=V$. In our usual notations, $U=W_0$, $U^\perp =W_1$. Define $X^*:=\Lambda\cap U$ and $Y:= \Lambda/U^\perp \cap\Lambda$. Fix a symplectic basis $\{\lambda_1,\ldots,\lambda_g,\mu_1,\ldots,\mu_g\}$ of $\Lambda$ for $E$. We use $(x,y)$ to denote the coordinates of a vector $v=x\lambda+y\mu$, where $x$ and $y$ are both row vectors of $g$ elements. Assume the degeneration is maximal and $U$ is a rational Lagrangian of dimension $g$. Suppose $v_1,\ldots, v_g$ is a basis of $X^*$ and 
\[
v_i=\sum_{j=1}^{g} c_{ij}\lambda_j+d_{ij}\mu_j, \quad \forall i=1,\ldots, g. 
\]

Since $\{v_1,\ldots,v_g\}$ is a complex basis for $V$, $c\tau+d\delta$ is invertible, where $c$ and $d$ are $g\times g$-matrices. With respect to the complex basis $\{v_1,\ldots, v_g\}$, the vector $x\lambda+y\mu$ has coordinates $(x\tau+y\delta)(c\tau+d\delta)^{-1}$. 

For the convenience of computation, we change to a new basis. Suppose the dual basis of $\{v_1,\ldots,v_g\}$ of $X^*$ is a compatible basis of $X$. Extend $\{v_1.\dots,v_g\}$ to a basis $\{u'_1,\ldots, u'_g,v_1\ldots,v_g\}$ of $\Lambda$, and under this new basis,
 \[
 E=\begin{pmatrix}
 S &\mathfrak{d}\\
 -\mathfrak{d} &0
 \end{pmatrix},
 \]
 
 where $S$ is an integral, skew symmetric matrix. Let the transformation matrix be $M'\in \GL(2g, \mathbf{Z})$. As in Corollary~\ref{the lattice in tropical cone}, we can always choose $M'$ such that 
 
\begin{align*}
\begin{pmatrix}
A & B\\
0 & I_g
\end{pmatrix}
M'^{-1}&=\begin{pmatrix}
a&b\\
c&d
\end{pmatrix}
\in\Sp(E,\mathbf{Q}),\\
A^{-1}&=\mathfrak{d}\delta^{-1},\\
A^{-1}B\mathfrak{d}&=\frac{1}{2}S.
 \end{align*}
 
 Let $(x',y')$ be the coordinates under the basis $\{u_1'\ldots, u'_g,v_1,\ldots,v_g\}$, and $(x,y)$ be the coordinates under the basis $\{\lambda_1,\ldots, \lambda_g,\mu_1,\ldots, \mu_g\}$. The transformation of coordinates is
\[
 \begin{pmatrix}x'&y'\end{pmatrix}=\begin{pmatrix}x & y\end{pmatrix}M'.
\]
 
 Therefore,
 \[
 \begin{pmatrix} x' & y'\end{pmatrix} \begin{pmatrix} A^{-1} & -A^{-1}B\\ 0 & I_g\end{pmatrix}\begin{pmatrix} a & b\\ c & d\end{pmatrix}=\begin{pmatrix} x & y\end{pmatrix}.
 \]

Fix a punctured holomorphic disk $\Delta^*$ with coordinate $q$. The universal covering is the upper half plane $\mathbf{H}$ with coordinate $t$, and $q=e^{2\pi it}$. Fix $\tau'$, and choose $\tau_0'\in \mathfrak{S}_g$. Consider the family $\pi: \mathcal{X}_\eta/\Delta^*$ defined by the holomorphic map $\tau(t): \mathbf{H}\to\mathfrak{S}_g$
\begin{align*}
\tau_0'+\Im(\tau')t=(a\tau(t)+b\delta)(c\tau(t)+d\delta)^{-1}\delta.
\end{align*}

This family has a multiplicative uniformization, which is a trivial algebraic torus $\widetilde{G}$ over $\Delta^*$. For a point $(x,y)$ in $V$, consider the image in $\widetilde{G}_t:=V_t/X^*\cong(\mathbf{C}^*)^g$ for different $t$. We always use the dual basis of $\{v_1,\ldots,v_g\}$ as the standard coordinates of $(\mathbf{C}^*)^g$. The coordinates of $(x,y)$ in $(\mathbf{C}^*)^g$ are given by the row vector
\[
\exp{(2\pi i (x\tau(t)+y\delta)(c\tau(t)+d\delta)^{-1})}=\exp{\Bigg(2\pi i (x \ y)\begin{pmatrix} \tau(t)\\
\delta\end{pmatrix}
\begin{pmatrix}
c\tau(t)+d\delta
\end{pmatrix}^{-1}\Bigg)}.
\]

Use the new coordinates $(x',y')$,
\[
\exp{\Bigg(2\pi i \begin{pmatrix} x' & y'\end{pmatrix}\begin{pmatrix} A^{-1} & -A^{-1}B\\
0 & I_g\end{pmatrix}
\begin{pmatrix}
(\tau_0'+\Im(\tau')t)\delta^{-1}\\
I_g
\end{pmatrix}\Bigg)}.
\]

We only have to consider the periods $y'=0$. Identify $Y=\Lambda/\Lambda\cap U$ with the subgroup $\langle u'_1,\ldots, u'_g\rangle$ in $\Lambda$, and use the row vector $(x')$ as the coordinates. It defines a bilinear map $b$ from $Y\times X$ to holomorphic functions of $t$. 
\[
b(x',\alpha)=\exp{\Bigg(2\pi i x' (A^{-1}(\tau_0'+\Im(\tau')t)\delta^{-1}-A^{-1}B)\begin{pmatrix}\alpha(v_1)\\ \vdots\\ \alpha(v_g)\end{pmatrix}\Bigg)}.
\]

Here $\alpha\in X$, and $x'$ means both the vector $\sum_{i=1}^{g}x'_iu_i'$ and its coordinates. For simplicity, we also denote the coordinates of $\alpha$ by $\alpha$. Decompose $b$
\begin{align*}
 b(\lambda,\alpha)=b_0(\lambda,\alpha)\cdot b(t)(\lambda,\alpha)\cdot b'(\lambda,\alpha), 
\end{align*}
 
 where
 \begin{align*}
 b_0(\lambda,\alpha)&=\exp{(2\pi i x' (A^{-1}\tau_0'\delta^{-1})\alpha)},\\
 b_t(\lambda,\alpha)&=b(x',\alpha)=\exp{(2\pi i x' (A^{-1}\Im(\tau')t\delta^{-1})\alpha})=q^{\langle\alpha,\check{\phi}(\lambda)\rangle},\\
 b'(\lambda,\alpha)&=\exp{(-\pi i x' S\mathfrak{d}^{-1}\alpha)}.
 \end{align*}
 
Fix the notation. $\phi(\lambda)(v)=E(\lambda, v)$. Therefore $\phi(u_i')=d_iv_i^*$. Notice that since $A^{-1}(\tau_0'+\Im(\tau')t)(A^{-1})^T$ is a symmetric matrix and $S$ is integral skew symmetric, $b(\lambda_1,\phi(\lambda_2))=b(\lambda_2,\phi(\lambda_1))$. Define the symmetric bilinear form $b_S: Y\times Y \to \mu_2\cong \mathbf{Z}/2\mathbf{Z}$
\[
b_S:=\exp{(-\pi i E(\lambda,\lambda'))}=\exp{(-\pi i x_1'S(x_2')^T)}=\exp{(\pi i x_1'S(x_2')^T)}\equiv S \pmod {2}. 
\]

Notice that $b_S(2\lambda_1,\lambda_2)=1$ for any $\lambda_1,\lambda_2\in Y$. Therefore $b_S$ is a symmetric bilinear form over $Y/2Y$ which is a vector space over $\mathbf{Z}/2\mathbf{Z}$. There always exists some quadratic form $\chi$ whose associated bilinear form is $b_S$. For example we can define a symmetric matrix $S'$ by requiring that $S'_{ij}=1$ if $S_{ij}$ is odd, and everywhere else is $0$. Then
\[
\chi(\lambda):=\exp{(\frac{1}{2} \pi i x' S'(x')^T)}
\]

is such a quadratic form. There may be more than one choices of such a quadratic form (e.g. over characteristic $2$). We fix a choice and denote it by $\chi$. Let $\alpha'=\chi^{-1}$. We can even extend $\chi$ to get a semi-character $\chi: \Lambda\to\mathbb{S}^1$ by defining it to be trivial over $\Lambda\cap U$.  

Define the positive quadratic forms $Q_0$ and $Q$ by the symmetric matrices $1/2A^{-1}\tau_0' (A^{-1})^T$ and $1/2A^{-1}\Im(\tau')(A^{-1})^T$ respectively. Notice that 
\[
Q(\lambda)=\frac{1}{2}\langle \phi(\lambda),\check{\phi}(\lambda)\rangle.
\]

Define
\begin{align*}
 a(\lambda)&=\chi^{-1}(\lambda)\exp{(2\pi i (Q_0(\lambda)+tQ(\lambda)))}\\
 &=a_0(\lambda)\cdot a_t(\lambda)\cdot a'(\lambda).
 \end{align*}
 
 where
 \begin{align*}
 a_0(\lambda)&:=\exp{(2\pi i Q_0(\lambda))},\\
 a_t(\lambda)&:=\exp{(2\pi i t Q(\lambda))}=q^{Q(\lambda)},\\
 a'(\lambda)&:=\chi^{-1}(\lambda).
 \end{align*}
 
 Since $Q$ is not necessarily integral, we may need to make a base change of degree $2$ to make $a_t(\lambda)$ a function over $\Delta^*$.
 
 \begin{lemma}
 Use $a,b$ above to define the $Y$-linearization of the trivial line bundle $\mathcal{O}_{\widetilde{G}}$ over $\widetilde{G}$, then $\mathcal{L}$, the descent of $\mathcal{O}_{\widetilde{G}}$ over $\mathcal{X}$, has the polarization of type $\delta$. 
 \end{lemma}
 
 \begin{proof}
 We do the computation explicitly. Use complex coordinates $z_\alpha=d_\alpha v_\alpha^*$. We can modify the $1$-cocycle $\{e(\lambda,z)\}$ by a coboundary such that 
\[
e(u'_i, z)=\exp{(-2\pi iz_i)}.
\]

Denote the matrix of $2Q_0+2Qt$ by $\tau''(t)$, and $\tau''(t)=X(t)+iY(t)$. Let $W(t)=Y(t)^{-1}$. We have 
\begin{align*}
\ud z_\alpha=\sum_i (\tau''(t)-S/2)_{i\alpha}\ud x_i'+d_\alpha\ud y'_\alpha\\
\ud\bar{z}_\beta=\sum_j(\bar{\tau}''(t)-S/2)_{j\beta}\ud x_j'+d_\beta\ud y_\beta'
\end{align*}

Following (\cite{GH} pp. 310-311), we can define the hermitian metric by a transition function
\[
h(z)=\exp{\bigg(\frac{\pi}{2}\sum W_{\alpha\beta}(z_\alpha-\bar{z}_\alpha)(z_\beta-\bar{z}_\beta-2i Y_{\beta\beta})\bigg)}
\]
 
Doing calculations similar to those in \cite{GH}, we can check that $h$ defines a hermitian metric for the $1$-cocycle $\{e(\lambda,z)\}$. Moreover, the curvature of this hermitian metric is 
\begin{align*}
\Theta_{\mathcal{L}}&=\pi \sum_{\alpha,\beta} W_{\alpha\beta} \ud z_\alpha\wedge \bar{z}_\beta\\
&=2\pi i \bigg(\sum_{\alpha>\beta} S_{\alpha\beta}\ud x'_\alpha\wedge x'_\beta+\sum_\alpha d_\alpha \ud x'_\alpha\wedge\ud y_\alpha'\bigg). 
\end{align*}

This verifies that the degeneration data $\{a, b\}$ defines a line bundle $\mathcal{L}$ whose polarization is of type $\delta$.  
\end{proof}

The limit Hodge filtration $F_\infty$ is determined by $b_0b'$. The log monodromy $N=\check{\phi}:V\to U$ is given by $b_t$. Although this family looks special, the general degenerating family with the same log monodromy $N$ and limit Hodge filtration $F_\infty$ is asymptotic to this family in a precise sense. This is the content of the nilpotent orbit theorem. Therefore we can use the data obtained from this family as a model for the general degeneration. Without loss of generality, we can choose $\tau_0'$ to be $0$, so that $a_0$ is $1$. The conclusion is that $a'$ is the twist necessary for the direction $U$, with respect to the basis $\{u_i',v_j\}$.
%%%%%%%%%%%%%%%%%%%%%%%%%%%%%%%%%%%%%%%%%%%%%%%%%%%%%%%%%%%%%%%%%%%%%%%%%%
%%%%%%%%%%%%%%%%%%%%%%%%%%%%%%%%%%%%%%%%%%%%%%%%%%%%%%%%%%%%%%%%%%%%%%%%%%

\subsection{General Degeneration}\label{B2}
Consider a general one-parameter degeneration family $\mathcal{X}$ over $\Delta^*$, whose abelian part is non-trivial. Let the associated isotropic subspace be $U_\xi\subset V$ of dimension $r\leqslant g$. $U_\xi$ is obtained as the space of vanishing cycles $W_0^t\subset \Lambda_\mathbf{R}$. Let $J$ be the complex structure on $V$. Let $\widetilde{T}$ be the subspace of $V$ generated by $U_\xi$ and $U_\xi J$. $T:=\widetilde{T}/U_\xi\cap\Lambda$ is the toric part of dimension $r$. Since $U_\xi^\perp\cap U_\xi J=\{0\}$, $V/\widetilde{T}\cong U_\xi^\perp/U_\xi$ is a complex vector space of dimension $g':=g-r$. We have a pure Hodge structure of weight $-1$ on $U_\xi^\perp/U_\xi=W_1^t/W_0^t$. This is the period map of the abelian part $A$. The bilinear form $E$ restricted to $U_\xi^\perp/U_\xi$ is non-degenerate. This gives the polarization on $A$. Let $\widetilde{G}=V/U_\xi^\perp\cap \Lambda$. $Y_\xi=\Lambda/U_\xi^\perp\cap \Lambda$. The family of abelian varieties $\mathcal{X}$ is the quotient of the family of semiabelian varieties $\widetilde{G}$ by periods $Y_\xi$. We have the extension sequence of abelian group varieties over $\Delta^*$
\begin{equation}\label{extension sequence for general degeneration}
\begin{diagram}
0&\rTo& T&\rTo& \widetilde{G} &\rTo^\pi &A&\rTo& 0. 
\end{diagram}
\end{equation}

Let $X_\xi$ be the group of characters of $T$. It is the dual of the fundamental group $U_\xi\cap\Lambda$. For any $\alpha\in X_\xi$, $\alpha$ is a group homomorphism $\alpha: T\to \mathbb{G}_m$. The push-out of the short exact sequence \eqref{extension sequence for general degeneration} along $\alpha$ is a $\mathbb{G}_m$-torsor over $A$ whose associated invertible sheaf is denoted by $\mathcal{O}_{-\alpha}$. Since the total space has a group structure, $\mathcal{O}_{-\alpha}$ is in $\mathrm{Pic}^0(A/\Delta^*)$. This defines the map $c: X_\xi\to A^t$, where $A^t$ is the dual in the category of complex analytic spaces. Since $Y_\xi$ is a group of $\Delta^*$-sections of $\widetilde{G}$, for any $\lambda\in Y_\xi$, $\alpha\in X_\xi$, the push-out of $\lambda$ along $\alpha$ is a $\Delta^*$-section of $\mathcal{O}_{-\alpha}$. Denote the projection under $\pi$ by $c^t(\lambda)\subset A$, and the section by $\tau(\lambda,\alpha)$. This gives the trivialization $\tau$ of the biextension $(c^t\times c)^*\mathcal{P}^{-1}$. 

The relatively ample line bundle of type $\delta$ on $\mathcal{X}$ is represented as a line bundle $\widetilde{\mathcal{L}}$ over $\widetilde{G}$ with a $Y_\xi$-action. Since $\widetilde{G}$ is a $T$-torsor, $\widetilde{\mathcal{L}}$ is also equipped with a $T$-action, and we can do a partial Fourier expansion. Restricted to any section of $A$, $\widetilde{\mathcal{L}}$ is trivial. Therefore, suppose that $\widetilde{\mathcal{L}}$ descends to an ample line bundle $\mathcal{M}$ of $A$. This is a choice, and we fix this choice. Then the $\alpha$-eigenspace of $\Gamma(\widetilde{G},\widetilde{\mathcal{L}})$ is identified with $\Gamma(A,\mathcal{M}_\alpha)$ and the partial Fourier expansion is
\[
\Gamma(\widetilde{G},\widetilde{\mathcal{L}})=\bigoplus_{\alpha\in X_\xi}\Gamma(A,\mathcal{M}_\alpha).
\]

Denote the Fourier coefficients by the homomorphisms $\sigma_\alpha: \Gamma(\widetilde{G},\widetilde{\mathcal{L}})\to \Gamma(A,\mathcal{M}_\alpha)$. Restricting to the $Y_\xi$-invariant subspace $\Gamma(\mathcal{X},\mathcal{L})$, there is a relation between the homomorphisms $\sigma_\alpha$ and $\sigma_{\alpha+\phi(\lambda)}$ for $\lambda\in Y_\xi$. That is,
\begin{equation}
\sigma_{\alpha+\phi(\lambda)}=\psi(\lambda)\tau(\lambda,\alpha)T^*_{c^t(\lambda)}\circ\sigma_\alpha
\end{equation}

for $\psi(y)$ a $\Delta^*$-section of $\mathcal{M}(c^t(\lambda))^{-1}$. This defines the trivialization $\psi$ of the central extension over $Y_\xi$. 

To get the explicit data $\tau$ and $\psi$, we choose a $0$-cusp $F(U)$ that is in the closure of $F(U_\xi)$, i.e. a maximal rational isotropic subspace $U$ that contains $U_\xi$. Assume $\Lambda\cap U_\xi$ be spanned by $\{v_1,v_2,\ldots,v_r\}$, and $U\cap \Lambda$ be spanned by $\{v_1,\ldots,v_r,v_{r+1},\ldots,v_g\}$. We choose the complement $\{u'_1,\ldots,u'_g\}$ as in the above section. Therefore $U_\xi^\perp\cap \Lambda$ is spanned by 
\[
\{v_1,\ldots,v_r,v_{r+1},\ldots,v_g, u'_{r+1},\ldots,u'_{g}\}. 
\]

Use the same periods $\tau'_0+t\Im(\tau')$ for the family as in Sect.~\ref{B1}. Write $\tau_0'$ in blocks.
\[
\tau_0'=\begin{pmatrix} 
\tau_1' & \tau_2'\\
\tau_3' & \tau_4'
\end{pmatrix},\quad
S=\begin{pmatrix} 
S_\xi & S_2\\
S_3 & S_4
\end{pmatrix},
\]

where $\tau'_1$ is a $r\times r$-matrix and $\tau'_4$ is a positive-definite $g'\times g'$-matrix. Similarly, we write $S$ in blocks. Assume $\Im(\tau')$ is positive definite for the upper left $r\times r$ block and vanishes anywhere else. As in the above model, we take the quotient of the group generated by $\{v_1,\ldots,v_g\}$ and get an algebraic torus $\mathbb{G}_m^g$ over $\Delta^*$. Denote the group generated by $\{u'_{r+1}, \ldots,u'_g\}$ by $Y_1$ and the group generated by $\{u'_1,\ldots, u'_r\}$ by $Y_\xi$. Use $\lambda$ for a vector in $Y_\xi$ and $z$ for a vector in $Y_1$. Use $\mathrm{X}^\alpha$ for coordinate functions of $\mathbb{G}_m^r$, and $\mathrm{W}^\beta$ for coordinate functions on $\mathbb{G}_m^{g'}$. The family of abelian varieties $A$ is the quotient of $\mathbb{G}_m^{g'}$ by $Y_1$. $\widetilde{G}$ is the quotient of $\mathbb{G}_m^g$ by $Y_1$. By this multiplicative uniformization $\mathbb{G}_m^g$, all line bundles $\mathcal{O}_\alpha$ are trivialized canonically after pull back over $\mathbb{G}_m^{g'}$. Of course there is ambiguity from the action of $Y_1$ for the trivialization of every fiber $\mathcal{O}_\alpha(c^t(\lambda))$. However, the lift of $Y_\xi$ to points $u_1',\ldots,u_g'$ fixes this ambiguity. The upshot is $\mathcal{O}_\alpha^{-1}(c^t(\lambda))$ is trivialized, and the section $\tau(\lambda,\alpha)$ is equivalent to a function over $\Delta^*$. Similarly, the pull back of $\mathcal{M}$ to $\mathbb{G}_m^{g'}$ is trivial. A lift of $Y_\xi$ gives a canonical trivialization of $\mathcal{M}^{-1}(c^t(\lambda))$. The section $\psi(\lambda)$ is also a function over $\Delta^*$. 

Regard $\widetilde{\mathcal{L}}$ as the quotient of $(\mathbf{C}^*)^{r}\times (\mathbf{C}^*)^{g'}\times\mathbf{C}\times\Delta^*$ by $Y_1$. The action is parametrized by $q\in \Delta^*$. A section $\vartheta\in \Gamma(\mathcal{X},\mathcal{L})$ is a $Y$-invariant function over  $(\mathbf{C}^*)^{r}\times (\mathbf{C}^*)^{g'}$. Do the partial Fourier decomposition

\begin{equation}
\vartheta=\sum_{\alpha\in X_\xi} \bigg(\sum_\beta a_{\alpha\beta} \mathrm{W}^\beta\bigg)\mathrm{X}^\alpha. 
\end{equation}

The function over $(\mathbf{C}^*)^{g'}$
\[
\sigma_\alpha(\vartheta)=\sum_\beta a_{\alpha\beta} \mathrm{W}^\beta
\]

is a section of $\Gamma(A, \mathcal{M}_\alpha)$. \\

This can be easily verified. We use $e$, $a$, $b$ from the above section. For $\mu\in Y$, we have
\[
\vartheta((\mathrm{W}, \mathrm{X})+\mu)=e(\mu, (\mathrm{W},\mathrm{X}))\vartheta(\mathrm{W},\mathrm{X})=\frac{1}{a(\mu)b((\mathrm{W},\mathrm{X}),\phi(\mu))}\vartheta(\mathrm{W},\mathrm{X}). 
\]

If $\mu=z\in Y_1$,
 \[
b((\mathrm{W},\mathrm{X}),\phi(z))=\mathrm{W}^{\phi(z)}. 
\]

On the one hand, 
\begin{align}
\vartheta((\mathrm{W},\mathrm{X})+z)&=\sum_{\alpha\in X_\xi}\bigg(\sum_\beta a_{\alpha\beta} b(z,\beta)\mathrm{W}^{\beta}\bigg)b(z,\alpha)\mathrm{X}^\alpha\\
&=\sum_{\alpha\in X_\xi} \sigma_\alpha(\vartheta)(\mathrm{W}+z)b(z,\alpha)\mathrm{X}^\alpha. 
\end{align}

On the other hand, 
\begin{align}
e(z, (\mathrm{W},\mathrm{X}))\vartheta(\mathrm{W},\mathrm{X})=\frac{1}{a(z)\mathrm{W}^{\phi(z)}}\sum_{\alpha\in X_\xi}\sigma_\alpha(\vartheta)(\mathrm{W})\mathrm{X}^\alpha. 
\end{align}

Compare the coefficients of $\mathrm{X}^\alpha$, 
\begin{equation}
\sigma_\alpha(\vartheta)(\mathrm{W}+z)=\frac{1}{b(z,\alpha)a(z)\mathrm{W}^{\phi(z)}}\sigma_\alpha(\vartheta)(\mathrm{W})
\end{equation}

Here $1/b(z,\alpha)$ are the factors of automorphy for $\mathcal{O}_{\alpha}$, and $1/(a(z)\mathrm{W}^{\phi(z)})$ are the factors of automorphy for $\mathcal{M}$. Hence $\sigma_\alpha(\vartheta)$ is a section of $\mathcal{M}_\alpha$. The abelian part $A$ is determined by $b(z, \beta)$, i.e. the periods $\tau_4'$ and the twist $S_4$. 

If $\mu=\lambda\in Y_\xi$, 
\[
b(\mathrm{W},\mathrm{X},\phi(\lambda))=\mathrm{X}^{\phi(\lambda)}. 
\]

We have
\begin{align}
\vartheta((\mathrm{W},\mathrm{X})+\lambda)&=\sum_{\alpha\in X_\xi}\Big(\sum_\beta a_{\alpha\beta} b(\lambda,\beta)\mathrm{W}^{\beta}\Big)b(\lambda,\alpha)\mathrm{X}^\alpha\\
&=\sum_{\alpha\in X_\xi} \sigma_\alpha(\vartheta)(\mathrm{W}+c^t(\lambda))b(\lambda,\alpha)\mathrm{X}^\alpha. 
\end{align}

And
\begin{align}
e(\lambda, (\mathrm{W},\mathrm{X}))\vartheta(\mathrm{W},\mathrm{X})=\frac{1}{a(\lambda)}\sum_{\alpha\in X_\xi}\sigma_\alpha(\vartheta)(\mathrm{W})\mathrm{X}^{\alpha-\phi(\lambda)}. 
\end{align}

Compare the coefficients of $\mathrm{X}^\alpha$, 
\begin{align}
\sigma_{\alpha+\phi(\lambda)}(\vartheta)(\mathrm{W})&=a(\lambda)b(\lambda,\alpha)\sigma_\alpha(\vartheta)(\mathrm{W}+c^t(\lambda)),\\
&=a(\lambda)b(\lambda,\alpha)T^*_{c^t(\lambda)}\circ\sigma_\alpha(\vartheta)(\mathrm{W}). 
\end{align}

Therefore, after using the canonical trivialization from the $0$-cusp $U$, 
\begin{align}
&\psi(\lambda)=a(\lambda)\\
&\tau(\lambda,\alpha)=b(\lambda,\alpha).
\end{align}

Using the coordinates with respect to $\{u_1',\ldots,u_r'\}$, the points are
\begin{align*}
b(x,\alpha)&=\exp{\Bigg(2\pi i (x,0)(A^{-1}(\tau_0'+\Im(\tau')t)\delta^{-1}-1/2S\mathfrak{d}^{-1})\begin{pmatrix}\alpha\\0\end{pmatrix}\Bigg)}\\
&=b_0b'(x,\alpha)b_t(x,\alpha),
\end{align*}
 
where $\alpha$ is a column vector of $r$ elements, $b_0b'(x,\alpha)$ is a constant in $\mathbf{C}$, while $b_t(x,\alpha)$ is a function over $\Delta^*$. Write $b_t$ as an $r\times r$-matrix of functions $q^{2Q\mathfrak{d}^{-1}}$. Write $a(x)=a_0a'(x)a_t(x)$, where $a_0a'(x)$ is a constant, and $a_t(x)=q^{Q(x)}$ is a function of $q$ if we make a necessary base change. Note that the quadratic form $Q\in \mathcal{C}(X_\xi)$, we have
\begin{corollary}
The trivializations $\tau$ and $\psi$ are compatible with $Q$. 
\end{corollary}

Assume $a_0=b_0=1$. Denote the function field on $\Delta^*$ by $K$. The explicit data for $U_\xi$ is 

\begin{proposition}
If we make the choice of the $0$-cusp $U$, and use the data from $U$, we can write $\tau$, $\psi$ as follows. 
\begin{align}\label{trivialization data}
\tau&=b'b_t: Y_\xi\times X_\xi\to K\\
\psi&=a'a_t: Y_\xi\to K,
\end{align}

where $b_t=q^{2Q\mathfrak{d}^{-1}}$ and $a_t=q^{Q(x)}$ is defined in terms of $Q\in \mathcal{C}(X_\xi)$, and $b'$, $a'$ are twist data associated with the boundary component $U_\xi$
\begin{align}
b'(x,\alpha)&=\exp{(-\pi i xS_\xi\mathfrak{d}^{-1}\alpha)},\\
a'(x)&=\exp{(-1/2\pi i xS_\xi'x^T)}, 
\end{align}

where $S'_\xi$ is a symmetric matrix such that $S_\xi'\equiv S_\xi \pmod {2\mathbf{Z}}$. 
\end{proposition}

\bibliographystyle{amsalpha}
\bibliography{Bibliography}

\end{document}